\documentclass[a4paper,11pt,reqno]{amsart}
\usepackage[colorlinks=true]{hyperref}
\usepackage{amssymb}
\usepackage{mathtools}
\usepackage{comment}
\usepackage[makeroom]{cancel}

\usepackage[utf8]{inputenc}

\usepackage[textsize=tiny]{todonotes}
\usepackage[all]{xy}

\theoremstyle{plain}
\newtheorem{lemma}{Lemma}[section]
\newtheorem{theorem}[lemma]{Theorem}
\newtheorem{proposition}[lemma]{Proposition}
\newtheorem{prop}[lemma]{{Proposition}}
\newtheorem{corollary}[lemma]{Corollary}
\newtheorem{convention}[lemma]{Convention}

\newtheorem*{convention*}{Convention}

\theoremstyle{definition}
\newtheorem{definition}[lemma]{Definition}

\newtheorem{example}[lemma]{Example}
\newtheorem{remark}[lemma]{Remark}
\newtheorem*{definition*}{Definition}

\theoremstyle{remark}

\newcommand{\E}{\mathcal{E}}
\newcommand{\lb}{\left[ \cdot\,,\cdot\right] }
\usepackage[left=2cm, right=3cm, top=3cm, bottom=1.7cm]{geometry}
\newcommand{\dd}{\mathrm{d}}

\newcommand{\X}{\mathfrak {X} }
\title{Higher algebras of affine varieties}

\let\oldtocsection=\tocsection

\let\oldtocsubsection=\tocsubsection

\let\oldtocsubsubsection=\tocsubsubsection

\renewcommand{\tocsection}[2]{\hspace{0em}\oldtocsection{#1}{#2}}
\renewcommand{\tocsubsection}[2]{\hspace{1em}\oldtocsubsection{#1}{#2}}
\renewcommand{\tocsubsubsection}[2]{\hspace{2em}\oldtocsubsubsection{#1}{#2}}

\thanks{The author acknowledges support of the CNRS Project Miti 80Prime \emph{Granum}.}

\address{Universit\'e de Lorraine\\ CNRS\\ IECL\\ F-57000 Metz, France.}

\title{On Symmetries of singular foliations}
\author{Ruben LOUIS }
\date{\today}

\begin{document}

\maketitle
\begin{abstract}
   This paper shows that a weak symmetry action of a Lie algebra $\mathfrak{g}$ on a singular foliation $\mathcal F$ induces a unique up to homotopy Lie $\infty$-morphism from $\mathfrak{g}$ to the DGLA of vector fields on a universal Lie $\infty$-algebroid of $\mathcal F$. Such a morphism is known as $L_\infty$-algebra action in \cite{MEHTA2012576}. We deduce from this general result several geometrical consequences. For instance, we give an example of a Lie algebra action on an affine sub-variety which cannot be extended on the ambient space. Last, we introduce the notion of bi-submersion towers over a singular foliation and lift symmetries to those.

\end{abstract}
\tableofcontents
\section*{Introduction}
Singular foliations arise frequently in differential or algebraic geometry. Here, following \cite{AndroulidakisIakovos,AndroulidakisZambon,Cerveau, Debord,LLS}, we define a singular foliation on a smooth, complex, algebraic, real analytic manifold $M$  with sheaf of functions $\mathcal O$ to be a subsheaf $\mathcal F \colon U\longrightarrow\mathcal F(U )$ of the sheaf of vector fields $\mathfrak X$, which is closed under the Lie bracket and locally finitely generated as a $\mathcal O$-module. By Hermann's theorem \cite{Hermann}, this is enough to induce a partition of the manifold $M$ into {immersed} submanifolds of possibly different dimensions, called \emph{leaves} of the singular foliation. Singular foliations appear for instance as orbits of Lie group actions with possibly different dimensions or as symplectic leaves of a Poisson structure.  When all the leaves have the same dimension, we recover the usual \textquotedblleft regular foliations\textquotedblright \cite{DiederichKlas, LLL}.

The purpose of this paper is to look at symmetries of singular
foliations. Let $(M,\mathcal F)$ be a foliated manifold. A \emph{global symmetry} of a singular foliation $\mathcal F$ on $M$ is a diffeomorphism $\phi\colon M\longrightarrow M$ which preserves $\mathcal F$, that is, $\phi_*(\mathcal F)=\mathcal{F}$. The image of a leaf through a global symmetry is again a leaf (not necessarily the same leaf).

For $G$ a Lie group, a \emph{strict symmetry action} of $G$ on a foliated manifold $(M,\mathcal{F})$ is a smooth action $G\times M\longrightarrow M$ that acts by global symmetries \cite{GarmendiaAlfonso2}. Infinitesimally, it corresponds to a Lie algebra morphism $\mathfrak{g} \longrightarrow \mathfrak X(M)$ between the Lie algebra $(\mathfrak{g},\lb_\mathfrak{g})$ of $G$ and the Lie algebra of symmetries of $\mathcal F$, i.e., vector fields $X\in \mathfrak X(M)$ such that $[X,\mathcal{F}]\subseteq \mathcal{F}$. A strict symmetry action of $G$ on $M$ goes down to the leaf space $M/\mathcal F$, even though the latter space
is not a manifold. The opposite direction is more sophisticated, since a strict symmetry action of $G$ on $M/\mathcal{F}$ does not induce a strict action over $M$ in general. However, it makes sense to consider linear maps $\varrho\colon\mathfrak g\longrightarrow \mathfrak X(M)$ that satisfy $[\varrho(x),\mathcal{F}]\subset\mathcal{F}$ for all $x\in\mathfrak{g}$, and which are Lie algebra morphisms up to $\mathcal{F}$, namely, $\varrho([x,y]_\mathfrak{g})-[\varrho(x),\varrho(y)]\in\mathcal{F}$ for all $x,y\in\mathfrak{g}$. The latter linear maps are called \textquotedblleft \emph{weak symmetry actions}\textquotedblright. These actions induce a \textquotedblleft strict  action\textquotedblright on the leaf space i.e.,  a Lie algebra morphism $\mathfrak g\longrightarrow \mathfrak X(M/\mathcal{F})$, whenever $M/\mathcal{F}$ is a manifold, and an action of $G$ on $M/\mathcal{F}$, at least if $G$ is connected. 

Now, on a priori different subject. Let us emphasize on the following observation:  An infinitesimal action of a Lie algebra $\mathfrak{g}$ on a manifold $M$ is a Lie algebra morphism $\mathfrak g\longrightarrow \mathfrak{X}(M)$. Replacing $M$ by a Lie $\infty$-algebroid $(E,Q)$ seen as a $Q$-manifold, one expects to define {Lie algebra actions on $(E,Q)$} as Lie $\infty$-algebra morphisms $\mathfrak{g}[1]\longrightarrow \mathfrak{X}(E,Q)[1]$, the latter space being a DGLA of vector fields on that $Q$-manifold \cite{KotovAlexei2022Nfo}. Such Lie $\infty$-morphisms were studied by Mehta and Zambon \cite{MEHTA2012576} as “$L_\infty$-algebra actions”, and various results about those are given. In particular, these authors give several equivalent definitions and interpretations of those. It is easy to check that such a Lie $\infty$-morphism induces a weak symmetry action of $\mathfrak{g}$ on the singular foliation induced by $Q$.

In  \cite{LLS,LavauSylvain}, it is shown that behind every singular foliation or more generally any Lie-Rinehart algebra \cite{CLRL} there exists a Lie $\infty$-algebroid structure which is unique up to homotopy called the \emph{universal} Lie $\infty$-algebroid. Here is a natural question:  what does a symmetry of a singular foliation $\mathcal{F}$ induce on a universal Lie $\infty$-algebroid of $\mathcal{F}$?  Theorem \ref{main} of this paper gives an answer to that question. It states that any weak symmetry action of a Lie algebra on a singular foliation $\mathcal{F}$ can be lifted to a Lie $\infty$-morphism valued in the DGLA of vector fields on a universal Lie $\infty$-algebroid of $\mathcal{F}$. 
 Furthermore, Theorem \ref{main} says this lift is unique modulo homotopy equivalence.  This goes in the same direction as \cite{GarmendiaAlfonso2} which already underlined Lie-2-group structures associated  to strict symmetry action of Lie groups.

This result gives several geometric consequences. Here is an elementary question: can a Lie algebra action $\mathfrak{g} \to \mathfrak X(W)$ on an affine variety $W \subset \mathbb C^d$ be extended to a Lie algebra action $\mathfrak{g} \to \mathfrak X(\mathbb C^d)$ on $\mathbb{C}^d $?
    Said differently: it is trivial that any vector field on $W$ extends to $\mathbb C^d $, but can this extension be done in such a manner that it preserves the Lie bracket? Although no \textquotedblleft $\infty$-oids\textquotedblright appears in this question, which seems to be a pure algebraic geometry question, we claim that the answer goes through Lie $\infty $-algebroids and singular foliations. More precisely, the idea is then to say that any $\mathfrak g $-action on $W$ induces a weak symmetry action on the singular foliation $\mathcal I_W\mathfrak{X}(\mathbb{C}^d)$ of all vector fields vanishing on $W$ (here $\mathcal I_W$ is the ideal that defines $W$). By Theorem \ref{main}, we know that it is possible to lift any weak symmetry action of singular foliation into a  Lie $\infty $-morphism.  {The second order Taylor coefficient  of that Lie $\infty$-morphism, composed with the projection on vector fields of  arity $-1$, is of the form $\iota_{\eta(x,y)}$ where $\eta\colon \wedge^2\mathfrak g\longrightarrow \Gamma(E_{-1})$ satisfies $\varrho([x,y]_\mathfrak{g})-[\varrho(x), \varrho(y)]=\rho(\eta(x,y))$,\,(here $\rho\colon E_{-1}\longrightarrow TM$ is the anchor map of a universal Lie $\infty$-algebroid $(E,Q)$ of $\mathcal F$ and $\iota_e$ stands for the vertical vector field associated to a section $e\in \Gamma(E_{-1})$)}. But is it possible to build such a Lie $\infty $-morphism where the arity $-1$ of the second order Taylor coefficient is zero? There are cohomological obstructions. In some specific cases, obstruction classes appear on some cohomology, although in general the obstruction is rather a Maurer-Cartan-like equations that may or may not have solutions. We show if this class is non-zero, then we cannot manage to have $\eta=0$, and then no strict action exists.\\

The outline of this paper is as follows: In Section \ref{sec:1} we present some definitions and facts on weak symmetry actions of  Lie algebras on singular foliations and give some examples. In Section \ref{sec:3} we state the main results of this paper and present their proofs. In Section \ref{sec:4} we describe the relation between weak symmetry actions and Lie $\infty$-algebroids that have some special properties. In Section \ref{sec:5}  we define an obstruction class for extending a Lie algebra action on an affine variety to ambient space. In the last section of the paper, we introduce the  notion of “bi-submersion tower” over singular foliations that we denote by $\mathcal{T}_B$. The latter notion, as the name suggests, is a family of  “bi-submersions” which are built one over the other. The concept of bi-submersion over singular foliations has been introduced in \cite{AndroulidakisIakovos} and it is used  in $K$-theory \cite{Androulidakis-Iakovos-Georges} or differential geometry \cite{AndroulidakisIakovos2011Pcoa,AndroulidakisZambon,GarmendiaAlfonsoMarcoZambon}. We show that such a bi-submersion tower over a singular foliation $\mathcal{F}$ exists if and only if $\mathcal{F}$ admits a geometric resolution. Provided that it exists, we show in Theorem \ref{thm:final-sym} that  any infinitesimal action of a Lie algebra $\mathfrak{g}$ on the singular foliation $\mathcal{F}$ lifts to the bi-submersion tower $\mathcal{T}_B$.

In Appendix \ref{sec:2} and \ref{appendix-homotopies}, we review Lie $\infty$-algebroid structures and their morphisms and homotopies in order to fix notations.


\section*{Acknowledgement}
I am very grateful to my supervisor, Professor Camille Laurent-Gengoux, for guiding and giving me useful ideas all along of this paper. I would also like to acknowledge the full financial support for this PhD from R\'egion Grand Est. Also, I thank the CNRS MITI 80Prime project GRANUM, and the Institut Henri Poincaré for hosting me in November 2021. I would like to thank the referee for his/her careful reading, and his/her advices that helped to improve the presentation.
\section{Definitions and examples of weak and strict symmetry actions}
\label{sec:1}
\begin{convention}
Throughout this paper, $M$ stands for a smooth or complex manifold, or an affine variety over $\mathbb{C}$. We will denote the sheaf of smooth or complex, or regular functions on $M$ by $\mathcal O$ and the sheaf of vector fields on $M$ by $\mathfrak{X}(M)$, and $X[f]$
stands for a vector field $X\in\mathfrak{X}(M)$ applied to $f\in \mathcal O$. Also, $\mathbb K$ stands for $\mathbb R$ or $\mathbb C$.
\end{convention}

\begin{definition}
Let $\mathcal F\subset\mathfrak{X}(M)$ be a singular foliation on $M$. 
\begin{itemize}
    \item A diffeomorphism $\phi\colon M\longrightarrow M$ is said to be a \emph{symmetry} of $\mathcal F$, if $\phi_*(\mathcal F)=\mathcal F$.
    
    \item A vector field $X\in\mathfrak{X}(M)$ is said to be an \emph{infinitesimal symmetry of $\mathcal F$}, if $[X,\mathcal F]\subset\mathcal{F}$. The Lie algebra of infinitesimal symmetries of $\mathcal{F}$ is denoted by $\mathfrak{s}(\mathcal F)$.\\
\end{itemize}
In particular, $\mathcal F\subset\mathfrak{s}(\mathcal F)$, since $[\mathcal F,\mathcal F]\subset\mathcal{F}$. The latter are called \emph{internal symmetries} of $\mathcal{F}$. 
\end{definition}

\begin{proposition}\cite{AndroulidakisIakovos,GarmendiaAlfonso}\label{prop:symm}
Let $M$ be a smooth or complex manifold. The flow of an infinitesimal symmetry of $\mathcal F$, if it exists, is a symmetry of $\mathcal{F}$.
\end{proposition}

As we will see in Section \ref{sec:3}, one of the consequences of our future results is that any symmetry $X\in\mathfrak{s}(\mathcal F)$ of a singular foliation $\mathcal F$ admits a lift to a degree zero vector field on any universal $NQ$-manifold over $\mathcal{F}$ that commutes with the  homological vector field $Q$. This allows us to have an alternative proof and interpretation of Proposition \ref{prop:symm}.\\

\noindent
Let $(\mathfrak g, \lb_{\mathfrak g})$ be a Lie algebra over $\mathbb{K}=\mathbb{R}$ or $\mathbb{C}$, depending on the context.
\begin{definition}\label{def:sym}
A \emph{weak symmetry action of the Lie algebra $\mathfrak{g}$} on  a singular foliation $\mathcal F$ on $M$ is a $\mathbb K$-linear map $\varrho\colon\mathfrak{g}\longrightarrow \mathfrak X(M)$ that satisfies:\begin{itemize}
    \item $\forall\, x\in\mathfrak{g},\;[\varrho(x),\mathcal F]\subseteq \mathcal F$,
    \item  $\forall\, x,y\in\mathfrak{g},\;\varrho([x,y]_\mathfrak{g})-[\varrho(x),\varrho(y)]\in \mathcal F$.
\end{itemize} 
When $x\longmapsto \varrho(x)$ is a Lie algebra morphism, we speak of \emph{strict symmetry action} of $\mathfrak{g}$ on $\mathcal F$. There is an equivalence relation on the set of weak symmetry actions which is defined as follows: two weak symmetry actions, $\varrho, \varrho' \colon \mathfrak{g}\longrightarrow \mathfrak X(M)$ are said to be \emph{equivalent} if there exists a linear map $\varphi\colon\mathfrak{g}\longrightarrow \mathcal{F}$ such that $\varrho-\varrho'=\varphi$.
\end{definition}
\begin{remark}It is important to notice that when $\mathcal F$ is a regular foliation and $M/\mathcal{F}$ is a manifold, any weak symmetry action of a Lie algebra $\mathfrak g$ on $\mathcal{F}$ induces a strict action of $\mathfrak{g}$ over $M/\mathcal{F}$. Definition \ref{def:sym} is a way of extending this idea to all singular foliations.
\end{remark}
Here is a list of  examples.

\begin{example}\label{ex:pull-back}
Let $\pi\colon M\longrightarrow N$ be a submersion. Since any vector field on $N$ comes from a $\pi$-projectable vector  field on $M$, any Lie algebra morphism $\mathfrak{g}\longrightarrow \mathfrak{X}(N)$ can be lifted to a weak symmetry action $\mathfrak{g}\longrightarrow \mathfrak{X}(M)$  on the regular foliation $\Gamma(\ker\dd\pi)$, and any two such lifts are equivalent.

Furthermore, any weak action of a Lie algebra $\mathfrak{g}$ on a singular foliation $\mathcal{F}$ on $N$ can be lifted to a class of weak symmetry actions on the pull-back foliation $\pi^{-1}(\mathcal{F})$, (see Definition 1.9 in  \cite{AndroulidakisIakovos}). 
\end{example}
\begin{example}\label{ex:isotropy} Let $\mathcal{F}$ be a singular foliation on $M$. For any point $m\in M$, the set $\mathcal F(m) = \left\lbrace X \in\mathcal F \mid  X(m) = 0\right\rbrace$ is a Lie subalgebra of $\mathcal{F}$. Put  $\mathcal I_m = \left\lbrace f \in
C^\infty (M )\mid  f (m) = 0\right\rbrace$. The quotient space $\mathfrak{g}_m = \dfrac{\mathcal{F}(m)}{\mathcal I_m\mathcal F}$ is a Lie algebra, since $\mathcal I_m\mathcal F\subseteq \mathcal F(m)$ is a Lie ideal. The Lie algebra $\mathfrak{g}_m$ is called the \emph{isotropy Lie algebra of $\mathcal{F}$ at $m$} (see \cite{AndroulidakisIakovos2}). Let us denote, by $\lb_{\mathfrak{g}_m}$, its Lie bracket.\\

\begin{enumerate}
    \item Consider $\varrho\colon \mathfrak{g}_m\to \mathcal F(m)\subset\mathfrak{X}(M)$ a section of the projection map,
\begin{equation}\xymatrix{\mathcal I_m\mathcal F\ar@{^{(}->}[r]&\mathcal F(m)\ar@{->>}[r]&\mathfrak{g}_m\ar@/_/[l]_\varrho}\end{equation}Then, $[\varrho(x), \mathcal I_m\mathcal F]\subset\mathcal{I}_m\mathcal F$ and $\varrho([x,y]_{\mathfrak{g}_m})-[\varrho(x),\varrho(y)]\in\mathcal I_m\mathcal F$. Hence, the map $\varrho\colon \mathfrak{g}_m\to \mathfrak{X}(M)$ is a weak symmetry action of the singular foliation $\mathcal I_m\mathcal F$. A different section $\varrho'$ of  the projection map yields an equivalent weak symmetry action of $\mathfrak{g}_m$ on $\mathcal{I}_m\mathcal{F}$. An obstruction class for having a strict symmetry action equivalent to $\varrho$ will be given later in Section \ref{sec:5}.

\item In particular, for $k\geq 1$, let us denote by $\mathfrak{g}_m^k$ the isotropy Lie algebra of the singular foliation $\mathcal I^k_m\mathcal F$ at $m$. Every section $\varrho_k\colon\mathfrak{g}^k_m\longrightarrow \mathfrak{X}(M)$ of the projection map 

\begin{equation}\xymatrix{\mathcal I_m^{k+1}\mathcal F\ar@{^{(}->}[r]&\mathcal{I}^k_m\mathcal F\ar@{->>}[r]&\mathfrak{g}^k_m\ar@/_/[l]_{\varrho_k}}\end{equation}is a weak symmetry action of the Lie algebra $\mathfrak{g}^k_m$ on the singular foliation $\mathcal I_m^{k+1}\mathcal F$. 
\end{enumerate}

\end{example}
\begin{example} The following example comes from \cite{CLG-Ryv}, and follows the same patterns as in Examples \ref{ex:pull-back} and \ref{ex:isotropy}.
Let $(M,\mathcal F)$ be a singular foliation on a smooth manifold $M$ and $L\subset M$ a leaf. Let $[L,M]$ be a neighborhood of $L$ in $M$ equipped with some projection $\pi: M\longrightarrow L$. According to \cite{CLG-Ryv}, upon replacing $[L,M]$ by a smaller neighborhood of $L$ if necessary, there exists an Ehresmann connections (that is a vector sub-bundle $H \subset T[L,M]$ with $H\oplus\ker(\dd\pi) =T[L,M]$)
which satisfies that $\Gamma(H) \subset \mathcal F $. Such an Ehresmann connection is called an \emph{Ehresmann $\mathcal F$-connection} and induces a $C^\infty(L)$-linear section $\varrho_H\colon\mathfrak X(L)\longrightarrow \mathcal{F}^{\mathrm{proj}}$ of the surjection $\mathcal{F}^{ \mathrm{proj}} \longrightarrow \mathfrak X(L)$, where $\mathcal{F}^{ \mathrm{proj}}$ stands for  vector fields of $\mathcal{F}$ $\pi$-projectable on elements of $\mathfrak{X}(L)$. The section $\varrho_H$ is a weak symmetry action of $\mathfrak{X}(L)$ on the \emph{transverse} foliation $\mathcal{T}:=\Gamma(\ker\dd\pi)\cap~\mathcal F$.
When the Ehresmann connection $H$ is flat,   $\varrho_H $ is bracket-preserving, 
and defines a strict symmetry of $\mathfrak{X}(L)$ on the transverse foliation $\mathcal T $.
\end{example}

\begin{example}
Consider, for a fixed $k\in \mathbb{N}_0$, the singular foliation $\mathcal{F}_k:=\mathcal{I}_0^k\mathfrak{X}(\mathbb{R}^d)$ of all vector fields on $\mathbb{R}^d$ vanishing at order $k$ at the origin. The action of the Lie algebra $\mathfrak{gl}(\mathbb R^d)$ on $\mathbb{R}^d$ which is given by $$\mathfrak{gl}(\mathbb R)\longrightarrow \mathfrak{X}(\mathbb{R}^d),\,(a_{ij})_{1\leq i,j\leq d}\longmapsto\sum_{1\leq i, j\leq d}a_{ij}x_i\frac{\partial}{\partial x_j}$$ is a strict symmetry of $\mathfrak{g}$ on  $\mathcal F_k$.
\end{example}
\begin{example}Let $\varphi := (\varphi_1,\ldots,\varphi_r )$ be a $r$-tuple of homogeneous polynomial functions in $d$ variables over $\mathbb K$. Consider the singular foliation $\mathcal F_\varphi$ (see \cite{CLRL} Section 3.2.1) which is generated by all polynomial vector fields $X \in \mathfrak X(\mathbb K^d )$ that satisfy $X[\varphi_i ] = 0$ for all
$i \in\{1,\ldots, r\}$. The action $\mathbb K\rightarrow\mathfrak{X}(\mathbb{K}^d),\;\lambda\mapsto \lambda\overrightarrow{E}$, is a strict symmetry of $\mathbb{K}$ on $\mathcal{F}_\varphi$. Here, $\overrightarrow{E}$ stands for the Euler vector field.
\end{example}
\begin{example}\label{ex:affine-action}
Let $W$ be an affine variety realized as a subvariety of $\mathbb{C}^d$ and $\mathcal{I}_W\subset\mathbb{C}[x_1,\ldots,x_d]$ its vanishing ideal. Let us denote by $\mathfrak{X}(W):=\mathrm{Der}(\mathbb{C}[x_1,\ldots,x_d]/\mathcal{I}_W)$ the Lie algebra of vector fields on $W$. Let $\mathcal{F}_W:=\mathcal I_W\mathfrak{X}(\mathbb{C}^d)$ the singular foliation made of vector fields vanishing on $W$. Since every vector field on $W$ can be extended to a vector field on $\mathbb{C}^d$ tangent to $W$, every Lie algebra morphism $\varrho\colon\mathfrak{g}\longrightarrow \mathfrak{X}(W)$ extends to a linear map $\widetilde{\varrho}\colon\mathfrak{g}\longrightarrow \mathfrak{X}(\mathbb{C}^d)$ that makes this diagram commutes,
\begin{equation*}
    \xymatrix{&\mathfrak{X}(\mathbb{C}^d)\ar@{->>}[d] \\\mathfrak{g}\ar[ru]^{\widetilde{\varrho}}\ar[r]_{\varrho}& \mathfrak{X}(W)}
\end{equation*}
{For  $x,y \in \mathfrak g$, the extension $\widetilde{\varrho}$ satisfies: $\widetilde{\varrho}(x)[\mathcal I_W]\subset \mathcal I_W $, so that $[\widetilde{\varrho}(x), \mathcal I_W\mathfrak{X}(\mathbb{C}^d)]\subset \mathcal I_W\mathfrak{X}(\mathbb{C}^d)$. We have   $\widetilde\varrho([x,y]_\mathfrak{g})-[\widetilde{\varrho}(x),\widetilde{\varrho}(y)]\in \mathcal I_W\mathfrak{X}(\mathbb{C}^d )$ because $\widetilde\varrho$ is a Lie algebra morphism when restricted to $W$. Hence, $\widetilde\varrho$ is a weak symmetry action of $\mathfrak{g}$ on $\mathcal{F}_W$. Two different extensions give equivalent symmetry actions. Here is a natural question: Can we extend the Lie algebra action of $\mathfrak g $ on $W $ to a Lie algebra action on $\mathbb{C}^d$? This example shows that this question can be reformulated as: is any extension $\widetilde \varrho$ of $\varrho $ equivalent to a strict symmetry action? Corollary \ref{cor:affine} of Section \ref{sec:5} gives an obstruction class of extending this weak symmetry action to a strict one.}
\end{example}

\section{A Lie $\infty$-morphism lifting a weak symmetry of a foliation}\label{sec:3}
We refer the reader to Appendix \ref{sec:2} for the notion of universal Lie $\infty$-algebroids of a singular foliation and for notations. We denote them by $(E, Q)$ and their functions by $\E$. The triple $(\mathfrak{X}_\bullet(E), \lb, \text{ad}_Q)$ is a differential graded Lie algebra, where $\mathfrak{X}_\bullet(E)$ stands for the module of graded vector fields (=graded derivations of $\E$) on $E$, the bracket $\lb$ is the graded commutator of derivations and $\text{ad}_Q:=[Q,\cdot\,]$.

Also, see Appendix \ref{appendix-homotopies} for the notion of Lie $\infty$-morphism of differential graded Lie algebras and for notations.\\

{We now state the main theorem of the paper. In Appendix \ref{appendix-homotopies}, Proposition \ref{prop:induced-action} shows that a Lie $\infty$-morphism between a Lie algebra $\mathfrak{g}$ and the DGLA of graded vector fields of a Lie $\infty$-algebroid $(E,Q)$, induces a weak symmetry action of $\mathfrak{g}$ on the basic singular foliation $\mathcal{F}=\rho(\Gamma(E_{-1}))$ of $(E,Q)$. In this section, we show  that any weak symmetry action of a Lie algebra $\mathfrak{g}$ on a singular foliation $\mathcal{F}$ arises this way.}

\begin{convention}
{From now on and in the sequel, the Lie algebra $(\mathfrak{g}, \lb_\mathfrak{g})$ (possibly of infinite dimension) is concentrated in degree $0$ so that $\mathfrak{g}$ shifted by $1$, namely  $\mathfrak{g}[1]$, is concentrated  in degree $-1$. The Lie bracket $\lb_\mathfrak{g}\colon \mathfrak{g}[1]\times \mathfrak{g}[1]\longrightarrow \mathfrak{g}[1]$ of  $\mathfrak{g}[1]$ is of degree $+1$}.

\end{convention}
\begin{convention}
{In this paper, vector bundles are of finite rank. Lie $\infty$-algebroids are of finite rank except in Theorem \ref{alt-thm-res} we notice that the result holds true without this assumption.}
\end{convention}

\begin{definition}\label{def:lift}
Let $\mathcal{F}$ be a singular foliation on $M$ and $(E, Q)$ a Lie $\infty$-algebroid over $\mathcal F$. Consider a weak symmetry action $\varrho\colon\mathfrak{g}\longrightarrow\mathfrak{X}(M)$ of $\mathfrak g$ on  $\mathcal F$.
\begin{itemize}
    \item We say that a Lie $\infty$-morphism of differential graded Lie algebras \begin{equation}
        \Phi\colon (\mathfrak{g}[1],\lb_\mathfrak{g})\rightsquigarrow (\mathfrak X_\bullet(E)[1],\lb,\text{ad}_{Q})
    \end{equation} \emph{lifts the weak symmetry action $\varrho$ to $(E, Q)$} if for all $x\in\mathfrak{g}, f\in\mathcal{O}$,\, $\Phi_0(x)(f)=\varrho(x)[f]$.
    \item When $\Phi$ exists, we say then $\Phi$ is a \emph{lift} of $\varrho$ on $(E,Q)$.
\end{itemize}
\end{definition}

 We now state the main theorem of this paper.
\begin{theorem}\label{main}Let $\mathcal F$ a be a singular foliation on a smooth manifold (or an affine variety) $M$ and $\mathfrak{g}$ a Lie algebra. Let $\varrho\colon\mathfrak{g}\longrightarrow\mathfrak{X}(M)$ be a weak symmetry action  of $\mathfrak g$ on $\mathcal F$. The following assertions hold:
\begin{enumerate}
    \item for any universal Lie $\infty$-algebroid $(E,Q)$ of the singular foliation $\mathcal F$, there exists a Lie $\infty$-morphism $\Phi\colon(\mathfrak{g}[1],\lb_\mathfrak g)\rightsquigarrow \left(\mathfrak X_\bullet(E)[1],\lb, \emph{ad}_Q \right)$ that lifts $\varrho$ to $(E,Q)$,
    \item any two such Lie $\infty$-morphisms are homotopy equivalent over the identity of $M$,
    \item any two such lifts of any two  equivalent weak symmetry actions of $\mathfrak{g}$ on $\mathcal{F}$ are homotopy equivalent over the identity of $M$.
\end{enumerate}

\begin{remark}
Lie $\infty$-morphisms in item ${1}$ of Theorem \ref{main} are called $\mathfrak{g}$-actions on $(E,Q)$ in \cite{MEHTA2012576}.
\end{remark}

\begin{remark}
Item ${1}$ in Theorem \ref{main} implies that
\begin{enumerate}
    \item 

there exists a linear map $\Phi_0\colon\mathfrak{g}[1]\longrightarrow\mathfrak{X}_0(E)[1]$ such that \begin{equation}\label{eq:compatibility-with-Q}
    \Phi_0(x)[f]=\varrho(x)[f], \;\text{and}\,\; [Q,\Phi_0(x)]=0,\quad \forall x\in \mathfrak{g}[1], f\in \mathcal{O}.
\end{equation}
$\Phi_0$ is not a graded Lie algebra morphism, but there exist a linear map $\Phi_1\colon\wedge^2\mathfrak{g}[1]\longrightarrow \mathfrak{X}_{-1}(E)[1]$ such that for all $x,y, z\in \mathfrak{g}[1]$,\begin{equation*}
    \Phi_0([x,y]_{\mathfrak{g}})-[\Phi_0(x),\Phi_0(y)]=[Q,\Phi_1(x,y)].
\end{equation*}
Also, \begin{equation*}
    \Phi_1\left([x,y]_\mathfrak{g},z\right)-[\Phi_0(x),\Phi_1(y,z)]+\circlearrowleft(x,y,z)=[Q,\Phi_2(x,y,z)]
\end{equation*}for some linear map $\Phi_2\colon\wedge^3\mathfrak{g}[1]\longrightarrow \mathfrak{X}_{-2}(E)[1]$. These compatibility conditions continue to higher multilinear maps.  

\item
 For every element $x\in\mathfrak{g}$ and $i\geq 1$, there is a degree zero map $\nabla_x\in\mathrm{Der}(E)$ (i.e. $\nabla_x(fe)=~f\nabla_x(e) + \varrho(x)[f]e$, for $f\in \mathcal{O}, e\in \Gamma(E)$) depending linearly on $x$, such that  \begin{equation}\label{eq:dual-action}
     \left\langle \Phi_0(x)^{(0)}(\alpha),e\right\rangle=\varrho(x)[\langle\alpha,e\rangle]-\left\langle\alpha,\nabla_x(e)\right\rangle,\quad\text{for all $\alpha\in \Gamma(E^*),\; e\in \Gamma(E)$}
 \end{equation}
 where $\Phi_0(x)^{(0)}$ stands for the arity zero component of $\Phi_0(x)$. Therefore, by using Equations \eqref{eq:compatibility-with-Q}, \eqref{eq:dual-action} and the dual correspondence between Lie $\infty$-algebroids and
$NQ$-manifolds \cite{Poncin,LeanMadeleineJotz2020Maru,Voronov2}, we obtain these compatibility conditions:
 $$\ell_1\circ\nabla_x=   \nabla_x\circ \ell_1\quad\text{and}\quad  \rho\circ\nabla_x=\text{ad}_{\varrho(x)}\circ \rho,$$
$\ell_1$ stands for the corresponding unary bracket of $(E, Q)$. Also, for $X\in \mathfrak{X}(M)$, $\text{ad}_{X}:=[X,\,\cdot\,]$. In general, the map\;$\mathfrak g[1]\longrightarrow \mathrm{Der}(E),\,x\mapsto \nabla_x$ is not a Lie algebra morphism  even when the action $\varrho$ is strict. In fact, there exists a bilinear map $\gamma\colon \wedge^2\mathfrak g[1]\longrightarrow \mathrm{End}(E)[1]$ of degree ~$0$ that satisfies
\begin{equation}
    \nabla_{[x,y]_\mathfrak g}-[\nabla_x, \nabla_y]=\gamma(x,y)\circ\ell_1- \ell_1\circ \gamma(x,y)+ \ell_2(\eta(x,y),\cdot\,),
\end{equation}
here $\ell_2$ is the corresponding $2$-ary bracket of $(E,Q)$, and  $\eta\colon\wedge^2\mathfrak g\longrightarrow \Gamma(E_{-1})$ is such that $\varrho([x,y]_\mathfrak{g})-[\varrho(x),\varrho(y)]=\rho(\eta(x,y))$.
\end{enumerate}
\end{remark}

\end{theorem}
\begin{corollary}\label{1symmetry}
{Let $(E,Q)$ be a universal Lie $\infty$-algebroid of a singular foliation $\mathcal F$. For any symmetry $X\in\mathfrak{X}(M)$ of  $\mathcal F$, there exists a degree zero vector field $Z\in\mathfrak X_0(E)$\begin{enumerate}
\item that commutes with $Q$, i.e., such that $[Z,Q]=0$,
\item and that extends $X$ in the sense that the following diagrams commute \begin{equation}\label{eq:1symmetry}
\xymatrix{C^\infty(M)\ar[d]_X\ar[r]^{p^*}&\Gamma\left(S^\bullet( E^*)\right)\ar[d]^Z\\C^\infty(M)\ar[r]^{p^*}&\Gamma\left(S^\bullet (E^*)\right)}\qquad\text{and}\qquad\xymatrix{\Gamma\left(S^\bullet( E^*)\right)\ar[d]_Z\ar[r]^{\iota^*}&C^\infty(M)\ar[d]^X\\\Gamma\left(S^\bullet (E^*)\right)\ar[r]^{\iota^*}&C^\infty(M)}
\end{equation} 
where $p\colon E\longrightarrow M$ is the projection map and $\iota \colon M\hookrightarrow E$  the zero section.
\end{enumerate}}
\end{corollary}
\begin{remark}
{Geometrically, Equation \eqref{eq:1symmetry} means that $p_*(Z)=X$ and $Z|_M=X$.}
\end{remark}
\begin{proof}
To construct $Z$, it suffices to apply Theorem \ref{main} for $\mathfrak{g}=\mathbb{R}$ and take $Z$ to be the image of $1$ through $\Phi_0\colon \mathbb R[1]\longrightarrow \mathfrak{X}_0(E)[1]$. 
\end{proof}

\begin{remark}
\label{rmk:flot}
 In particular, Corollary \ref{1symmetry} has the following consequences: \begin{enumerate}
    \item for any admissible $t$, the flow $\Phi^Z_t\colon \E \longrightarrow \E$ of $Z$ being an isomorphism of Lie $\infty$-algebroids, it induces an isomorphism of vector bundles $E_{-1}\longrightarrow E_{-1}$. Since $[Q,Z]=0$, the following diagram commutes,$$\xymatrix{\Gamma(E_{-1})\ar[d]_{\rho}\ar[r]^{(\Phi^Z_t)^{(0)}}&\Gamma(E_{-1})\ar[d]^\rho\\\mathfrak{X}(M)\ar[r]_{(\varphi^X_t)_*}&\mathfrak{X}(M)}$$where $\phi^X_t$ is the flow of $X$ at $t$.
    \item Consequently, for any open set $U\subset M$ which is stable under $\varphi^X_t$, there exists an invertible matrix $\mathfrak{M}^t_X$ with coefficients in $\mathcal O(U)$ that satisfies $$\left(\phi^X_t\right)_*\begin{pmatrix}X_1\\\vdots\\X_n\end{pmatrix}=\mathfrak{M}^t_X\begin{pmatrix}X_1\\\vdots\\X_n\end{pmatrix},$$
   for some generators $X_1,\ldots,X_n$ of $\mathcal{F}$ over $U$. As announced earlier, we recover Proposition \ref{prop:symm}, that is, $\left(\phi^X_t\right)_*(\mathcal{F})=\mathcal{F}$.
   
\end{enumerate} 
\end{remark}

 Let $(E,Q)$ and $(E',Q')$ be two universal Lie $\infty$-algebroids of $\mathcal{F}$. A direct consequence of Ricardo Campos's Theorem 4.1 in \cite{CamposRicardo} is that the differential graded Lie algebras $\left( \mathfrak X_\bullet(E)[1],\lb, \mathrm{ad}_Q \right)$ and $\left( \mathfrak X_\bullet(E')[1],\lb, \mathrm{ad}_{Q'}\right)$ are homotopy equivalent over the identity of $M$. This leads to the following statement.
\begin{corollary}
 Let $\varrho\colon\mathfrak{g}\longrightarrow\mathfrak{X}(M)$ be a weak symmetry action of a Lie algebra $\mathfrak{g}$ on $\mathcal F$. Then,   there exist Lie $\infty$-morphisms, $\Phi\colon\mathfrak{g}[1]\rightsquigarrow \left( \mathfrak X_\bullet(E)[1],\lb, \emph{ad}_Q \right)$ and $\Psi\colon\mathfrak{g}[1]\rightsquigarrow \left(\mathfrak X_\bullet(E')[1],\lb, \emph{ad}_{Q'} \right)$ that lift $\varrho$, and $\Phi,\Psi$ make the following diagram commute up to homotopy

\begin{equation}\label{diagram:Campos}\xymatrix{ &\mathfrak{g}[1]\ar@{~>}[dl]_\Phi\ar@{~>}[dr]^\Psi& \\\left( \mathfrak X_\bullet(E)[1],\lb, \emph{ad}_Q \right)\ar@{<~>}[rr]^\sim & &\left( \mathfrak X_\bullet(E')[1],\lb, \emph{ad}_{Q'} \right).}\end{equation}

\end{corollary}

\begin{proof}
The composition of $\Phi$ with the horizontal map in the diagram \eqref{diagram:Campos} is a lift of the action ~$\varrho$. It is necessarily homotopy equivalent to $\Psi$ by item ${2}$ in Theorem \ref{main}.
\end{proof}

\subsection{Cohomology of longitudinal graded vector fields}In this section, we study the cohomology of longitudinal vector fields, which will help in proving the main results stated in the beginning of Section \ref{sec:3}.\\

Let $\mathcal{F}$ be a singular foliation on $M$.
\begin{definition}
Let $E$ be a splitted graded manifold over $M$ with sheaf of function $\E=~\Gamma(S(E^*))$. A vector field $L\in\mathfrak{X}(E)$ is said to be a \emph{longitudinal vector field for $\mathcal{F}$} if there exists vector fields $X_1,\ldots,X_k\in \mathcal{F}$ and functions $\Theta_1,\ldots,\Theta_k\in\E$ such that, \begin{equation}
  L(f)=\sum_{i=1}^kX_i[f]\Theta_i,\qquad \forall f\in\mathcal{O}.
\end{equation}
\end{definition}
\begin{example}\label{longi:examples}Here are some examples.
\begin{enumerate}
\item Vertical\footnote{We say a vector field on $E$ is \emph{vertical} if it is $\mathcal{O}$-linear.} vector fields are longitudinal.
    \item For any $Q$-manifold $(E,Q)$ over a manifold $M$. The homological vector field $Q\in\mathfrak{X}(E)$ is a longitudinal vector field for its basic singular foliation $\mathcal{F}:=\rho(\Gamma(E_{-1}))$.
    \item {Longitudinal vector fields are precisely of the form ${\sum_{i=1}^k\Theta_iX_i+ V},$ for $X_1,\ldots,X_k\in \mathcal{F}$,\,$\Theta_1,\ldots,\Theta_k\in\E$ and $V\in \mathfrak{X}(E)$ a vertical vector field on $E$.}
    \item For $(E, Q)$ a $Q$-manifold and $\mathcal{F}:=\rho(\Gamma(E_{-1}))$ its basic singular foliation. For any extension of a symmetry $X\in\mathfrak{s}(\mathcal{F})$ of $\mathcal{F}$ to a degree zero vector field $\widehat{X}\in\mathfrak{X}(E)$, the degree $+1$ vector field $[Q,\widehat{X}]$ is longitudinal for $\mathcal{F}$.
    \item[] Let us show this last point using local coordinates $(x_1, \dots,x_n) $ on $M$ and a local trivialization $\xi^1,\xi^2,\ldots$ of graded sections in $\Gamma(E^*)$. The vector fields $Q $ and $\widehat{X} $ take the form:
    \begin{equation}\label{eq:sym_long}
    \begin{array}{rcl}
    Q&=&\displaystyle{\sum_j\sum_{k,\, |\xi^k|=1}Q^j_k(x) \xi^k\frac{\partial}{\partial x_j} + \sum_{j}\sum_{k,\iota_1,\ldots,\iota_k}\frac{1}{k!}Q^j_{\iota_1,\ldots,\iota_k}(x) \xi^1\odot\cdots\odot\xi^k\frac{\partial}{\partial \xi^j}} \\
    \widehat{X}&= &\displaystyle{X+ \sum_{j}\sum_{k,\iota_1,\ldots,\iota_k}\frac{1}{k!}X^j_{\iota_1,\ldots,\iota_k}(x)\xi^1\odot\cdots\odot\xi^k\frac{\partial}{\partial \xi^j}} \end{array}
    \end{equation}
    where $\displaystyle{X=\sum_{i=1}^n X_i(x) \frac{\partial}{\partial x_i}}$. By using Equation \eqref{eq:sym_long} we note that all the terms of $[Q, \widehat{X}]$ are vertical except maybe for the ones where the vector field $X$ appears. For $k\geq 1$, the vector field   $[Q^j_{\iota_1,\ldots,\iota_k}\xi^1\odot\cdots\odot\xi^k\frac{\partial}{\partial \xi^j}, X]$ is vertical; and for every fixed $k$, one has  \begin{align*}
        \left[\sum_{j=1}^n Q^j_k\xi^k\frac{\partial}{\partial x_j}, X\right]=\xi^k\left[\sum_{j=1}^n Q^j_k\frac{\partial}{\partial x_j}, X\right].
    \end{align*}
    
    Now, $\displaystyle{\left[\sum_{j=1}^n Q^j_k\frac{\partial}{\partial x_j}, X\right]\in\mathcal{F}}$, since $X$ is a symmetry for $\mathcal{F}$ and $\displaystyle{\sum_{j=1}^n Q^j_k\frac{\partial}{\partial x_j}\in\mathcal{F}}$. 
\end{enumerate}
\end{example}

\begin{remark}\label{longitudinal-stable}
Longitudinal vector fields are stable under the graded Lie bracket.
\end{remark}

Let us make two points on vector fields on $E$. 
\begin{enumerate}
\item Sections of the graded vector bundle $E$ are identified with derivations under the isomorphism mapping  $e\in\Gamma(E)\longmapsto \iota_e \in\mathfrak{X}(E)$. This allows us to identify a vertical vector field with (maybe infinite) sums of  tensor products of the form  $\Theta\otimes e$ with $\Theta\in \E, e\in \Gamma(E)$. 

\item {A $TM$-connection $\nabla$ on the graded bundle $E$, i.e., a collection of $TM$-connections $\nabla^i$ on $E_{-i}$ for $i\geq 1$, induces for $X\in \mathfrak{X}(M)$ a vector field of degree zero $\widetilde{\nabla}_X\in \mathfrak{X}(E)$ by setting for $f\in\mathcal{O}$, $\widetilde{\nabla}_X(f):= X[f]$ and $\widetilde{\nabla}_X(\xi):=\nabla^{i,*}_X(\xi)$ for every homogeneous element $\xi\in \Gamma(E_{-i}^*)$, where $\nabla^{i,*}_X$ is the dual $TM$-connection. Upon choosing a $TM$-connection on $E$ as above, we give a $\mathbb{N}_0\times \mathbb{Z}_{-}$ grading to vector fields on $E$ by the identification below:} \begin{align}\label{eq:identification}
    \mathfrak{X}_k(E)&\simeq\bigoplus_{j\geq 1}\E_{k+j}\otimes_\mathcal{O}\Gamma(E_{-j})\,\oplus\,\E_{k}\otimes_\mathcal{O}\mathfrak{X}(M)\\&\nonumber\simeq \oplus_{j\geq 1}\Gamma(S(E^*)_{k+j}\otimes E_{-j} )\,\oplus\,\Gamma(S(E^*)_k\otimes TM)
\end{align}
for all $k\in\mathbb{Z}$.
Therefore, one can realize a vector field $P\in\mathfrak{X}_k(E)$ as a sequence $P=(p_0,p_1,\ldots)$, where $p_0\in \Gamma(S(E^*)_k\otimes TM)$ and $p_i\in\Gamma(S(E^*)_{k+i}\otimes E_{-i})$ for  $i\geq 1$ are called \emph{components} of $P$. In the diagram \eqref{longitudinal:complex}, $P=(p_0,p_1,\ldots)$ is represented as an element of the anti-diagonal and $p_i$ is on column $i$. We say that $P$ is of \emph{depth} $n\in \mathbb{N}$ if $p_i=0$ for all $i< n $. In particular, vector fields of depth greater or equal to $1$ are vertical. Under the {isomorphism} \eqref{eq:identification}, the differential map $\mathrm{ad}_Q$ takes the form\begin{equation}
    D=D^h+\sum_{s\geq 0}D^{v_s}
\end{equation}with $D^2=0$. Here, $D^h=\mathrm{id}\otimes\dd\,\text{\,or\,}\,\mathrm{id}\otimes\rho$, and $$D^{v_s}\colon \Gamma(S(E^*)_k\otimes E_{-i})\to \Gamma(S(E^*)_{k+s+1}\otimes~E_{-i-s})$$ for $i\geq 0,\,s\geq 0$, where $E_{0}:=TM$. We denote the latter complex by $(\mathfrak{L}, D)$.  {The maps $D^{v_s}$, for $s\geq 1$, are represented as up-left-pointing arrows, and  $D^{v_0}$ by vertical arrows, in the following diagram, whose lines are complexes of $\mathcal{O}$-modules given by the differential map $D^h$:}
\begin{equation}\label{longitudinal:complex}
\xymatrix{& \vdots &&\vdots && \vdots\\\cdots\ar[r] &\Gamma(S(E^*)_{k+2}\otimes E_{-2})\ar[u]_{D^{v_0}} \ar^{D^h=\mathrm{id}\otimes\dd}[rr]&&\Gamma(S(E^*)_{k+2}\otimes E_{-1})\ar[u]_{D^{v_0}} \ar^{D^h=\mathrm{id}\otimes\rho}[rr]&& \Gamma(S(E^*)_{k+2}\otimes TM)\ar[u]_{D^{v_0}}\\\cdots\ar[r] & \Gamma(S(E^*)_{k+1}\otimes E_{-2})\ar[luu]^>>>>>{D^{v_1}}\ar[u]^\ddots_{D^{v_0}} \ar^{D^h=\mathrm{id}\otimes\dd}[rr]&&\Gamma(S(E^*)_{k+1}\otimes E_{-1})\ar[luu]_>>>{D^{v_2}}^>>>>{D^{v_1}}\ar[u]^\ddots_{D^{v_0}} \ar^{D^h=\mathrm{id}\otimes\rho}[rr]&& \Gamma(S(E^*)_{k+1}\otimes TM)\ar[lluu]_>>>{D^{v_1}}\ar[u]^\ddots_{D^{v_0}}\\\cdots \ar[r]&\Gamma(S(E^*)_k\otimes E_{-2})\ar[luu]\ar[u]^\ddots_{D^{v_0}}\ar^{D^h=\mathrm{id}\otimes\dd}[rr]&&\Gamma(S(E^*)_k\otimes E_{-1})\ar@/^/[llluuu]_>>>>>>>{D^{v_2}}\ar[lluu]_>>>>>>>{D^{v_1}}\ar[u]^\ddots_{D^{v_0}} \ar^{D^h=\mathrm{id}\otimes\rho}[rr]&&{\Gamma(S(E^*)_k\otimes TM)}\ar@/^/[llluuu]\ar[lluu]_>>>>>>>>{D^{v_1}}\ar[u]^\ddots_{D^{v_0}}\\ & \vdots\ar[u]_{D^{v_0}} &&\vdots\ar[u]_{D^{v_0}} && \vdots\ar[u]_{D^{v_0}}\\& \texttt{column $2$}&&\texttt{column $1$} && \texttt{column $0$}\\}
\end{equation}
\end{enumerate}

\begin{remark}
{For $j\geq 0$,  $\Theta \in \E$ and $\xi\in \Gamma(E_{-j})$ one has  $D^{v_0}(\Theta\otimes \xi)=Q(\Theta)\otimes \xi+(-1)^{|\Theta|}\Theta\odot D^{v_0}(1\otimes \xi)$ and $D^{v_i}(\Theta\otimes \xi)+(-1)^{|\Theta|}D^{v_i}(1\otimes \xi)$ for every $i\geq 1$. Here, $E_{0}:=TM$.}
\end{remark}
Under this correspondence, we understand longitudinal vector fields as the following. 
\begin{lemma}
A graded vector field $P=(p_0,p_1,\ldots,)\in\mathfrak{L}$ is longitudinal if $p_0\in\E\otimes_\mathcal{O}\mathcal{F}$.
\end{lemma}


The following theorem is crucial for the rest of this paper.
\begin{theorem}\label{thm:longitudinal}
Let $(E,Q)$ be a universal $Q$-manifold of $\mathcal{F}$. 
\begin{enumerate}
    \item Longitudinal vector fields form an acyclic complex.
    
     More precisely, any longitudinal vector field on $E$ which is an $\mathrm{ad}_Q$-cocycle is the image through $\mathrm{ad}_Q$ of some vertical vector field on $E$.
     
     \item More generally, if a vector field on $E$ of depth $n$ is an $\mathrm{ad}_Q$-cocycle, then it is the image through $\mathrm{ad}_Q$ of some  vector field on $E$ of depth $n+1$. 
\end{enumerate}
\end{theorem}
\begin{proof}
$(E,Q)$ is a universal $Q$-manifold of $\mathcal{F}$ implies that lines in \eqref{longitudinal:complex} are exact when we restrict the $0$-th column to sections in $\E\otimes_\mathcal{O}\mathcal{F}$. It is now a diagram chasing phenomena. Let $P=(p_0,p_1,\ldots,)\in\mathfrak{L}$ be a longitudinal element which is a $D$-cocycle. By longitudinality there exists an element  $b_1\in\Gamma(S(E^*)\otimes E_{-1})$ such that $(\mathrm{id}\otimes\rho)(b_1)=p_0$. Set $P_1=(0,b_1,0,\ldots)$, that is, we extend $b_1$ by zero on $\Gamma(S(E^*)\otimes E_{\leq -2})$ and $\Gamma(S(E^*)\otimes TM)$. It is clear that $P-D(P_1)=(0,p'_1,p'_2,\ldots)$ is also a $D$-cocycle of depth $1$. In particular, we have $D^h(p'_1)=0$ by exactness there exists $b_2\in\Gamma(S(E^*)\otimes E_{-2})$ such that $D^h(b_2)=p'_1$. As before put $P_2=(0,0,b_2,0,\ldots)$. Similarly, $P-D(P_1)-D(P_2)=(0,0,p''_2,p''_3,\ldots)$ is a $D$-cocycle. By recursion, we end up to construct $P_1,P_2,\ldots$ that satisfy $P-D(P_1)-D(P_2)+\cdots =0$, that is, there exists an element $B=(0,b_1,b_2,\ldots)\in\mathfrak{L}$ such that $D(B)=P$. This proves item ${1}$.

To prove item ${2}$ it suffices to cross out in the diagram \eqref{longitudinal:complex} the columns numbered $0,\ldots,n-~1$, which does not break exactness. The proof now follows as for item ${1}$.
\end{proof}
In particular, we deduce from item 1 of Theorem \ref{thm:longitudinal} the following exact subcomplex.
\begin{corollary}\label{cor:longitudinal}
Let $(E,Q)$ be a universal $Q$-manifold of $\mathcal{F}$.
 The subcomplex $\mathfrak{V}_Q$ of $(\mathfrak{X}(E),\mathrm{ad}_Q)$ made of vertical vector fields $P\in\mathfrak{X}(E)$ that satisfy $P\circ Q(f)=~0$ for all $f\in\mathcal{O}$
 is acyclic.
\end{corollary}
\begin{proof}
Let $P\in\mathfrak{X}(E)$ be a vertical vector field which is a $\mathrm{ad}_Q$-cocycle. {Notice that we have automatically $P\circ Q(f)=0$ for all $f\in\mathcal{O}$: indeed, $P$ is a $\mathrm{ad}_Q$-cocycle implies $[Q, P](f)=0$ for all $f\in \mathcal O$. Equivalently, $P\circ Q(f)= (-1)^{|P|}Q\circ P(f)$. Since $P$ is vertical, $P(f)=0$, which proves that $P\circ Q(f)=0$}. By Theorem \ref{thm:longitudinal} there exists a vertical vector field $\widetilde{P}\in\mathfrak{X}(E)$ such that $[Q,\widetilde{P}]=P$. Moreover, $\widetilde{P}\in \mathfrak{V}_Q$, since for all $f\in\mathcal{O}$, $$0=[Q,\widetilde{P}](f)=(-1)^{|\widetilde{P}|}\widetilde{P}\circ Q(f).$$This completes the proof.
\end{proof}
The following remark will be used in the proof of Theorem \ref{main}.
\begin{remark}\label{rmk:arity} For a cocycle $P\in \mathfrak{V}_Q$ of degree $0$ one has $P^{(-1)}= 0$ (for degree reason). By Corollary \ref{cor:longitudinal}, $P$ is the image by $\mathrm{ad}_Q$ of an element, $\widetilde P\in \mathfrak V_Q$ i.e., such that $[Q,\widetilde P]=P$. Also, one can choose ${\widetilde{P}}^{(-1)}=0$: we have  $$[Q^{(0)}, {\widetilde{P}}^{(-1)}]=[Q, \widetilde P]^{(-1)}=P^{(-1)}=0.$$ By exactness of $\mathrm{ad}_{Q^{(0)}}$ (see \cite{LLS}), we have $\widetilde{P}^{(-1)}=[Q^{(0)}, \vartheta]$ for some $\mathcal{O}$-linear map $$\vartheta\in \mathrm{Hom}\left(\Gamma(E^*), \Gamma(S^0(E^*))\right)$$ of degree $-2$. Put $\Bar{P}:= \widetilde P -[Q, \vartheta]$, where $\vartheta$ is extended to a derivation of arity $-1$. Clearly, \begin{align*}
    [Q, \Bar{P}]=P\qquad\text{and}\qquad \Bar{P}^{(-1)}= \widetilde{P}^{(-1)}-[Q, \vartheta]^{(-1)}=\widetilde{P}^{(-1)}-[Q^{(0)}, \vartheta]=0.
\end{align*}
Therefore, $P=\mathrm{ad}_Q(\Bar P)$ with ${\Bar P}^{(-1)}=0$.
\end{remark}

\subsection{Proof of the main results}\label{results:proofs}

This section is devoted to the proof of the main results stated in Section \ref{sec:3}. For the notations, see Appendix \ref{sec:2} and  \ref{appendix-homotopies}.\\

We start with the following lemma.
\begin{lemma}\label{lem:length}Assume $(E,Q)$ is a universal Lie $\infty$-algebroid over $M$. Let $\Bar{\Phi}\colon (S_\mathbb{K}^\bullet\mathfrak{g}[1],Q_\mathfrak{g})\longrightarrow (S_\mathbb{K}^\bullet\mathfrak X(E)[1],\Bar{Q})$ be a coalgebra morphism which is a {Lie $\infty$-morphism up to arity $n\geq 1$}, i.e., $\left(\Bar\Phi\circ Q_\mathfrak{g} -\Bar{Q}\circ\Bar\Phi\right)^{(i)}=0$ for all integer $i\in\{0,\ldots,n\}$. Then, $\Bar{\Phi}$ can be extended to a $\infty$-morphism up to arity $n+1$. 
\end{lemma}
\begin{proof} For convenience, we omit the variables. The identity, $$\Bar{Q}\circ\left(\Bar\Phi\circ Q_\mathfrak{g} -\Bar{Q}\circ\Bar\Phi\right)+\left(\Bar\Phi\circ Q_\mathfrak{g} -\Bar{Q}\circ\Bar\Phi\right)\circ Q_\mathfrak{g}=0$$ taken in arity $n+1$ yields, \begin{align*}
    0=\left(\Bar{Q}\circ(\Bar\Phi\circ Q_\mathfrak{g} -\Bar{Q}\circ\Bar\Phi)\right)^{(n+1)}&=\Bar{Q}^{(0)}\circ(\Bar\Phi\circ Q_\mathfrak{g} -\Bar{Q}\circ\Bar\Phi)^{(n+1)}\\&=[Q,(\Bar\Phi\circ Q_\mathfrak{g} -\Bar{Q}\circ\Bar\Phi)^{(n+1)}],
\end{align*}
\text{since $Q_\mathfrak{g}^{(0)}=0$ and $\left(\Bar\Phi\circ Q_\mathfrak{g} -\Bar{Q}\circ\Bar\Phi\right)^{(i)}=0$ for $i\in\{0,\ldots,n\}$}. It is clear that for all $n\geq 0$ the map $\left(\Bar\Phi\circ Q_\mathfrak{g} -\Bar{Q}\circ\Bar\Phi\right)^{(n+1)}\colon S^{n+2}_\mathbb{K}(\mathfrak{g}[1])\longrightarrow\mathfrak{X}_{-n}(E)[1]$ {takes values in vertical vector fields on $E$ because vector fields of degree $n\geq 1$ are vertical for degree reasons}. By virtue of Corollary \ref{cor:longitudinal} there exists a vector field $\zeta\in\mathfrak{X}_{-n-1}(E)[1]$ of degree $-n-1$ such that  \begin{equation}[Q,\Bar{\Phi}^{(n+1)}+\zeta]=\Bar\Phi^{(n)}\circ Q_\mathfrak{g}^{(1)}-\Bar{Q}^{(1)}\circ\Bar\Phi^{(n)}.
\end{equation}
{By replacing the arity $n+1$ of $\Bar{\Phi}$ by $\Bar{\Phi}^{(n+1)}+\zeta$, and keeping the other arities fixed, one obtains a new map $\Bar{\Psi}\colon (S_\mathbb{K}^\bullet\mathfrak{g}[1],Q_\mathfrak{g})\longrightarrow (S_\mathbb{K}^\bullet\mathfrak X(E)[1],\Bar{Q})$ such that $\Bar{\Psi}^{(j)}:=\Bar{\Phi}^{(j)}$ for $j\neq n+1$ and $\Bar{\Psi}^{(n+1)}:=\Bar{\Phi}^{(n+1)}+\zeta$. The map $\Bar{\Psi}$ satisfies \begin{equation}[Q,\Bar{\Psi}^{(n+1)}]=\Bar\Psi^{(n)}\circ Q_\mathfrak{g}^{(1)}-\Bar{Q}^{(1)}\circ\Bar\Psi^{(n)}.
 \end{equation}This implies that $\left(\Bar\Psi\circ Q_\mathfrak{g} -\Bar{Q}\circ\Bar\Psi\right)^{(n+1)}=0$. By construction, $\Bar\Psi$ is     a Lie $\infty$-morphism up to arity $n+1$, i.e., that satisfies $\left(\Bar\Psi\circ Q_\mathfrak{g} -\Bar{Q}\circ\Bar\Psi\right)^{(i)}=0$ for all integer $i\in\{0,\ldots,n+1\}$}. The proof continues by recursion.
\end{proof}

Let $\mathcal F$ be a singular foliation, and $(E,Q)$ a universal Lie $\infty$-algebroid of $\mathcal F$. We start with the following lemma.
\begin{lemma}\label{initia2:lemm}
  For every weak symmetry Lie algebra action of $\mathfrak g$ on $\mathcal{F}$ there exists a linear map, $\Phi_0\colon \mathfrak g[1]\rightarrow \mathfrak X_0(E)[1]$, such that $[Q,\Phi_0(x)]=0$ and $\Phi_0(x)[f]=\varrho(x)[f]$ for all $x\in\mathfrak g[1]$, $f\in\mathcal{O}$. 
\end{lemma}


\begin{proof}
For $x\in\mathfrak{g}$, let $\widehat{\varrho(x)}\in \mathfrak X_0(E)$ be any arbitrary extension of $\varrho(x)\in\mathfrak{s}(\mathcal{F})$ to a degree zero vector field on $E$. Since $\varrho(x)$ is a symmetry of $\mathcal{F}$, the degree $+1$ vector field $[\widehat{\varrho(x)},Q]$ is also a longitudinal vector field on $E$, see Example \ref{longi:examples} item ${3}$. In addition, $[\widehat{\varrho(x)},Q]$ is a $\mathrm{ad}_Q$-cocycle. By item ${1}$ of Theorem \ref{thm:longitudinal}, there exists a vertical vector field $Y(x)\in\mathfrak X_0(E)$ of degree zero such that \begin{equation}
    [Q,Y(x)+ \widehat{\varrho(x)})]=0.
\end{equation}

Let us set for $x\in \mathfrak{g}[1]$,\;$\Phi_0(x):=Y(x)+\widehat{\varrho(x)}$. By construction, we have, $[Q,\Phi_0(x)]=0$ and $\Phi_0(x)[f]=\varrho(x)[f]$ for all $x\in\mathfrak g[1],\,f\in\mathcal{O}$. 
\end{proof}

\begin{proof}[Proof of Theorem \ref{main}] Let us show item ${1}$. Note that
Lemma \ref{initia2:lemm}  gives the existence of a linear map $\Phi_0\colon\mathfrak{g}[1]\longrightarrow \mathfrak{X}_0(E)[1]$ such that, $[Q,\Phi_0(x)]=0$ for all $x\in \mathfrak{g}[1]$. For $x,y\in \mathfrak{g}[1]$, consider \begin{equation}
    \Lambda(x,y)=\Phi_0([x,y]_\mathfrak{g})-[\Phi_0(x),\Phi_0(y)].\end{equation}Since $\varrho([x,y]_\mathfrak{g})-[\varrho(x),\varrho(y)]\in \mathcal F$ for all $x,y\in\mathfrak g[1]$, and since $\rho\colon\Gamma(E_{-1})\longrightarrow\mathcal{F}$ surjective, we have $\varrho([x,y]_\mathfrak{g})-[\varrho(x),\varrho(y)]=\rho\left(\eta(x,y)\right)$ for some element $\eta(x,y)\in\Gamma(E_{-1})$ depending linearly on $x$ and $y$. Now we consider the vertical vector field of degree $-1$, $\iota_{\eta(x,y)}\in\mathfrak{X}_{-1}(E)$ which is defined on $\Gamma(E^*)$ as:
$$\iota_{\eta(x,y)} (\alpha):= \langle\alpha, {\eta(x,y)}\rangle\;\; \text{for all}\;\; \alpha\in\Gamma(E^*),$$ and extended it by derivation on the whole space. For every $f\in\mathcal O$,\begin{align*}
\left(\Lambda(x,y)-[Q,\iota_{\eta(x,y)}]\right)(f)&=\left(\varrho([x,y]_\mathfrak{g})-[\varrho(x),\varrho(y)]-\rho(\eta(x,y) \right)[f]\hspace{1cm}\text{(by definition of $\Phi_0$)}\\&=0\hspace{7.47cm}\text{(by definition of $\eta$)}
\end{align*}
It is clear that $\Lambda(x,y)+[Q,\iota_{\eta(x,y)}]$ is a $\text{ad}_Q$-cocycle. Also, $\left(\Lambda(x,y)+[Q,\iota_{\eta(x,y)}]\right)^{(-1)}=0$. Hence, by Corollary \ref{cor:longitudinal} and Remark \ref{rmk:arity}, $\Lambda(x,y)+[Q,\iota_{\eta(x,y)}]$ is of the form $[Q, \Upsilon(x,y)]$ for some vertical vector field $\Upsilon(x,y)\in\mathfrak{X}_{-1}(E)$ of degree $-1$ with $\Upsilon(x,y)^{(-1)}=0$. For all $x,y\in \mathfrak{g}[1]$, we define the Taylor coefficient $\Phi_1\colon\wedge^2\mathfrak{g}[1]\longrightarrow \mathfrak{X}_{-1}(E)[1]$  as $\Phi_1(x,y):=\Upsilon(x,y)+\iota_{\eta(x,y)}$. By construction,  we have the following relation
\begin{equation}
    \Phi_0([x,y]_\mathfrak{g})-[\Phi_0(x),\Phi_0(y)]=[Q, \Phi_1(x,y)],\;\forall x,y\in \mathfrak{g}[1].
\end{equation}


{So far, in the construction of the Lie $\infty$-morphism, we have shown the existence of a Lie $\infty$-morphism $\Bar{\Phi}\colon S_\mathbb{K}^\bullet\mathfrak{g}[1]\longrightarrow S_\mathbb{K}^\bullet\left(\mathfrak{X}(E)[1] \right)$ up to arity $1$, that is $(\Bar{\Phi}\circ Q_\mathfrak{g})^{(i)}=(\Bar{Q}\circ \Bar{\Phi})^{(i)}$ for $i=0,1$. The proof continues by recursion  by applying directly Lemma \ref{lem:length}. This proves item $1$ of the theorem.}
\end{proof}
Before proving item ${2}$ of Theorem \ref{main} we will need the following lemma. For convenience,  we sometimes omit the variables in $\mathfrak{g}$. See Appendix \ref{appendix-homotopies} for the notations.
\begin{lemma}\label{lemma:homotopy}
For any two Lie $\infty$-morphisms $\Gamma,\Omega\colon (S_\mathbb{K}^\bullet(\mathfrak{g}[1]),Q_\mathfrak{g})\longrightarrow (S_\mathbb{K}^\bullet(\mathfrak{X}(E)[1]),\Bar{Q})$ which coincide up to arity $n\geq 0$, i.e. $\Gamma^{(i)}=\Omega^{(i)}$, for $0\leq i\leq n$, their difference in arity $n+1$, namely, $$\Gamma^{(n+1)}-\Omega^{(n+1)}\colon S_\mathbb{K}^{n+2}(\mathfrak{g}[1])\longrightarrow\mathfrak{X}_{-n-1}(E)[1]$$ is valued in $\mathrm{ad}_Q$-coboundary. 
\end{lemma}

\begin{proof}
Indeed, a direct computation yields \begin{align*}
    \Bar{Q}\circ(\Gamma-\Omega)=(\Gamma-\Omega)\circ Q_\mathfrak{g}&\Longrightarrow \Bar{Q}^{(0)}\circ(\Gamma-\Omega)^{(n+1)}-\underbrace{\left((\Gamma-\Omega)\circ Q_\mathfrak{g}\right)^{(n+1)}}_{=0}=0\\&\Longrightarrow[Q,\Gamma^{(n+1)}-\Omega^{(n+1)}]=0\\&\Longrightarrow \Gamma^{(n+1)}-\Omega^{(n+1)}=[Q,
    H^{(n+1)}]\hspace{1cm}\text{(by item ${1}$ of Theorem \ref{thm:longitudinal})}
\end{align*}for some linear map $H^{(n+1)}\colon S_\mathbb{K}^{n+2}(\mathfrak{g}[1])\longrightarrow\mathfrak{X}_{-n-2}(E)[1]$.
\end{proof}

Let us show item ${2}$ of Theorem \ref{main}. Let $\Phi,\Psi\colon \mathfrak g[1]\rightsquigarrow\mathfrak X({E})[1]$ be two different lifts of the action $\mathfrak{g}\longrightarrow\mathfrak{X}(M)$. We denote by $\Bar{\Phi},\Bar{\Psi}\colon S_\mathbb{K}^\bullet(\mathfrak{g}[1])\longrightarrow S_\mathbb{K}^\bullet(\mathfrak{X}(E)[1])$ the unique comorphisms given respectively by the Taylor's coefficients  \begin{equation}
    \begin{cases}
    &\Bar{\Phi}^{(r)}\colon S_\mathbb{K}^{r+1}(\mathfrak{g}[1])\xrightarrow{\Phi_{r}}\mathfrak{X}_{-r}(E)[1]\\&\Bar{\Psi}^{(r)}\colon S_\mathbb{K}^{r+1}(\mathfrak{g}[1])\xrightarrow{\Psi_{r}}\mathfrak{X}_{-r}(E)[1]
    \end{cases},\; \text{for $r\geq 0$}
\end{equation}

For any $x\in \mathfrak{g}[1]$, the degree zero vector field\, $\Phi_{0}(x)-\Psi_{0}(x)\in\mathfrak{X}_0(E)$ is vertical. Moreover, we have, $[Q,\Phi_{0}(x)-\Psi_{0}(x)]=0$. By Corollary \ref{cor:longitudinal} there exists a vector field $H_{0}\in\mathfrak{X}_{-1}(E)$ of degree $-1$, such that $\Psi_{0}(x)-\Phi_{0}(x)=[Q,H_{0}(x)]$

\begin{equation}
   \xymatrix{&\mathfrak{g}[1]\ar[d]^{\Psi_{0}-\Phi_{0}}  \ar@{-->}[ld]_{H_{0}}\\\mathfrak{X}_{-1}(E)[1]\ar[r]^{\text{ad}_Q} & \mathfrak{X}_0(E)[1]}  
\end{equation}


Consider the following differential equation \begin{equation}\label{diff-eq1}
    \begin{cases}
    \frac{\dd\Xi_t}{\dd t}&=\Bar{Q}\circ H_t+H_t\circ Q_\mathfrak g,\hspace{1cm}t\in[0,1]\\\Xi_0&=\Bar \Phi
    \end{cases}
\end{equation}where $(\Xi_t)_{t\in[0,1]}$ is as in Definition \ref{homp:def}, and for $t\in[0,1]$, $H_t$ is the unique $\Xi_t$-coderivation where the only non-zero arity is $H^{(0)}=H_{0}$. Equation \eqref{diff-eq1} gives a homotopy between $\Bar{\Phi}$ and $\Xi_1$. When we consider the arity zero component in Equation \eqref{diff-eq1}, one obtains\begin{align*}
    \frac{\dd\Xi_t^{(0)}}{\dd t}&=\Bar{Q}^{(0)}\circ H_t^{(0)}+H_t^{(0)}\circ Q_\mathfrak g^{(0)}\\&=[Q,H_{0}]\\&=\Psi_{0}-\Phi_{0}=\Bar{\Psi}^{(0)}-\Bar{\Phi}^{(0)}.
\end{align*}Therefore, $\Xi_t^{(0)}=\Bar{\Phi}^{(0)}+t(\Bar{\Psi}^{(0)}-\Bar{\Phi}^{(0)})$, and $\Bar{\Phi}\sim \Xi_1$ with $\Bar{\Psi}^{(0)}=\Xi_1^{(0)}$. Using Lemma \ref{lemma:homotopy}, the image of  any element through the map  $\Bar{\Psi}^{(1)}-\Xi_1^{(1)}\colon S_\mathbb{K}^{2}(\mathfrak{g}[1])\longrightarrow\mathfrak{X}_{-1}(E)[1]$ is a $\text{ad}_Q$-coboundary. Thus, $\Bar{\Psi}^{(1)}-\Xi_1^{(1)}$ can be written as \begin{equation}
    \Bar{\Psi}^{(1)}-\Xi_1^{(1)}=[Q, H^{(1)}],\quad \text{with}\, \;H^{(1)}\colon S_\mathbb{K}^{2}(\mathfrak{g}[1])\longrightarrow\mathfrak{X}_{-2}(E)[1].
\end{equation}Let us go one step further by considering the differential equation on $[0,1]$ given by \begin{equation}\label{equa:diff2}
    \begin{cases}
    \frac{\dd\Theta_t}{\dd t}&=\Bar{Q}\circ H_t+H_t\circ Q_\mathfrak g\\\Xi_0&=\Bar \Xi_1
    \end{cases}
\end{equation}Here $H_t$ is the extension of $H^{(1)}$ as the unique $\Theta_t$-coderivation where all its arities vanish  except the arity 1 which is given by $H^{(1)}$. In arity zero, $(\Theta^{(0)}_t)_{t\in[0,1]}$ is constant and has value $\Theta_1^{(0)}=\Bar{\Psi}^{(0)}$. In arity one we have, \begin{align*}
    \frac{\dd\Theta_t^{(1)}}{\dd t}&=\Bar{Q}^{(0)}\circ H_t^{(1)}\\&=[Q,H^{(1)}]=\Bar{\Psi}^{(1)}-\Xi_1^{(1)}.
\end{align*}Hence, $\Theta_t^{(1)}=\Bar{\Phi}^{(1)}+t(\Bar{\Psi}^{(1)}-\Xi_1^{(1)})$ with $\Bar{\Psi}^{(i)}=\Theta_1^{(i)}$ for $i=0,1$. We then continue this
procedure by gluing all these homotopies as in \cite{CLRL}, p. 40-41. We will obtain at last a Lie $\infty$-morphism $\Omega$ such that  $\Bar{\Phi}\sim\Omega$ and $\Omega^{(i)}=\Bar{\Psi}^{(i)}$ for $i\geq 0$. That means $\Omega=\Bar{\Psi}$, therefore $\Bar{\Phi}\sim\Psi$. This proves item ${2}$. of Theorem \ref{main}.\\

Let us prove item ${3}$ of Theorem \ref{main}. Given two equivalent weak symmetry actions $\varrho,\varrho'$ of $\mathfrak{g}$ on a singular foliation $\mathcal{F}$, i.e., $\varrho,\varrho'$ differ by a linear map $\mathfrak g\longrightarrow \mathfrak{X}(M)$ of the form $x\mapsto \rho(\beta(x))$ for some linear map $\beta\colon\mathfrak g\longrightarrow \Gamma(E_{-1})$. Let $\Phi,\Phi'\colon\mathfrak{g}[1]\rightsquigarrow \left( \mathfrak X_\bullet(E)[1],\lb, \text{ad}_Q \right)$ be a lift into a Lie $\infty$-morphism of the action $\varrho$ and $\varrho'$ respectively. One has for all $x\in \mathfrak{g}[1]$ and $f\in\mathcal{O}$, \begin{align*}
        \left(\Phi_0(x)-\Psi_0(x)-[Q,\iota_{\varphi(x)}]\right)(f)&=\rho(\varphi(x))[f]-\langle Q(f),\varphi(x)\rangle\\&=0.
    \end{align*}
 {Also, \begin{align*}
     [Q, \Phi_0(x)-\Psi_0(x)-[Q,\iota_{\varphi(x)}]]&=[Q, \Phi_0(x)]-[Q,\Psi_0(x)]-[Q,[Q,\iota_{\varphi(x)}]]\\&=0,\qquad\text{(since $[Q, \Phi_0(x)]=[Q,\Psi_0(x)]=[Q,[Q,\iota_{\varphi(x)}]]=0$)}
 .\end{align*}} By Corollary \ref{cor:longitudinal} there exists a vertical derivation $\widehat{H}(x)\in \mathfrak{X}_{-1}(E)$ of degree $-1$ depending linearly on $x\in \mathfrak{g}[1]$ such that $$\Phi_0(x)-\Psi_0(x)=[Q,\widehat{H}(x)+\iota_{\varphi(x)}].$$ Let $H(x):=\widehat{H}(x)+\iota_{\varphi(x)}$, for $x\in \mathfrak{g}[1]$. The proof continues the same as for item ${2}$ of Theorem \ref{main}

\subsection{Particular examples}
We recall that for a regular foliation $\mathcal{F}$ on a manifold $M$ (i.e., $\mathcal{F}=\Gamma(F)$ for some involutive subvector bundle $F\subseteq TM$), the Lie algebroid $E_{-1}=F[1]$ whose sections form $\mathcal F$, is a universal Lie $\infty$-algebroid of $\mathcal F$. {In particular, $E_{-i}=0$ for $i\geq 2$}. Its corresponding $Q$-manifold is given by the leafwise De Rham differential on $\Gamma(\wedge^\bullet F^*)$. {Also, for any symmetry $X\in \mathfrak{s}(\mathcal{F})$, i.e., any vector field $X\in \mathfrak{X}(M)$ such that $[X, \mathcal{F}]\subset \mathcal{F}$, the Lie derivative $\mathcal{L}_X\colon \Gamma(\wedge^kT^*M)\longrightarrow \Gamma(\wedge^kT^*M)$ along $X$ induces a vector field in $\mathfrak{X}_0(E)$ i.e., a degree zero derivation of $\Gamma(S^\bullet (E^*))$.}

\begin{example}
Let $\mathcal{F}$ be a regular foliation on a manifold $M$. Every weak symmetry action $\mathfrak{g}\longrightarrow ~\mathfrak X(M),\,x\longmapsto \varrho(x)$,  of $\mathcal F$, can be lifted to Lie $\infty$-morphism $\Phi\colon\mathfrak{g}[1]\rightsquigarrow \left( \mathfrak X_\bullet(E)[1],\lb, \text{ad}_Q \right)$ given explicitly as follows:
\begin{align}
   x\in\, &\mathfrak{g}[1]\longmapsto \Phi_0(x)={\mathcal{L}}_{\varrho(x)}\in\mathfrak{X}_0(E)[1]\\x\wedge y\in\wedge^2&\mathfrak g[1]\longmapsto \Phi_1(x,y)=\iota_{\chi(x,y)}\in\mathfrak{X}_{-1}(E)[1]\end{align}
   and $\left(\Phi_i\colon \wedge^{i+1}\mathfrak{g}[1]\longrightarrow \mathfrak{X}_{-i}(E)\right)\equiv 0$, for all $i\geq 2$, where $\chi(x,y):=\varrho([x,y]_\mathfrak{g})-[\varrho(x),\varrho(y)]$ for $x,y\in\mathfrak{g}$. 
\end{example}
\begin{example}
Let $\mathcal F$ be a singular foliation on a manifold $M$ together with a strict symmetry action $\varrho\colon\mathfrak{g}\longrightarrow\mathfrak{X}(M)$ such that $\varrho({\mathfrak{g}})\subset\mathcal{F}$. Hence, $C^\infty(M)\varrho({\mathfrak{g}})$ is a singular foliation which is the image of the transformation Lie algebroid $\mathfrak g\times M$. The  universality theorem (see \cite{LLS,CLRL}) provides the existence of a Lie $\infty$-morphism $\nu\colon\mathfrak{g}[1]\longrightarrow E$. Let us call its Taylor coefficients $\nu_n\colon\wedge^{n+1}\mathfrak{g}[1]\longrightarrow E_{-n-1},\,n\geq 0$. We may take for example the $0$-th and $1$-th Taylor coefficients of a Lie $\infty$-morphism that lifts $\varrho$ as: \begin{align*}\label{natural}\Phi_0(x)&:=[Q,\iota_{\nu_0(x)}]\in\mathfrak{X}_{0}(E)[1],\; \text{for}\; x\in \mathfrak g[1].\\\Phi_1(x,y)&:=[Q,\iota_{\nu_1(x,y)}]^{(-1)}-\sum_{k\geq 0}[[Q,\iota_{\nu_0(x)}],\iota_{\nu_0(y)}]^{(k)}\in\mathfrak{X}_{-1}(E)[1],\; \text{for}\; x,y\in \mathfrak g[1].\end{align*}
Note that in this case the action $\varrho$ is equivalent to zero, therefore by item ${3}$ of Theorem \ref{main} the Lie $\infty$-morphism $\Phi$ is homotopic to zero. 
\end{example}
\section{Lifts of weak symmetry actions and Lie $\infty$-algebroids}\label{sec:4}

In this section, we consider the Lie algebra $\mathfrak{g}[1]$ as the trivial vector bundle over $M$  with fiber $\mathfrak{g}[1]$. 

The following proposition says that any lift of strict symmetry action of $\mathfrak{g}$ on a singular foliation $\mathcal{F}$ induces a Lie $\infty$-algebroids with some special properties and vice versa. See \cite{MEHTA2012576}, Proposition 3.3, for a proof of the following statement.
\begin{proposition}\label{alt:thm-res}Let $(E,Q)$ be a Lie $\infty$-algebroid over a singular foliation $\mathcal F$. 
Every Lie $\infty$-morphism 
  $\Phi\colon(\mathfrak{g}[1],\lb_\mathfrak{g})\rightsquigarrow(\mathfrak X_\bullet(E)[1],\lb,\mathrm{ad}_{Q})$ with $\mathfrak{g}$ of finite dimension
  induces a Lie $\infty$-algebroid  $(E\oplus\mathfrak{g}[1],Q')$ with \begin{equation}\label{def:Q}
  Q':=\dd^{\text{CE}}+ Q + \sum_{
k\geq 1,i_1,\ldots,i_k=1,\ldots,\mathrm{dim}(\mathfrak{g})}\frac{1}{k!}\xi^{i_1}\odot\cdots\odot\xi^{i_k}\Phi_{k-1}(\xi_{i_1},\ldots,\xi_{i_{k}}),
\end{equation}
where $\dd^{\text{CE}}$ is the Chevalley-Eilenberg complex of $\mathfrak{g}$,  $\xi^{1},\ldots,\xi^{\mathrm{dim}(\mathfrak{g})}\in\mathfrak{g}^
*$ is the dual basis of some basis $\xi_{1},\ldots,\xi_{\mathrm{dim}(\mathfrak{g})}\in\mathfrak{g}$ and for all $k\geq 0$, $\Phi_k\colon S^{k+1}\mathfrak{g}[1]\longrightarrow \mathfrak{X}_{-k}(E)[1]$ is the $k$-th Taylor coefficients of $\Phi$.\\

In the dual point of view, \eqref{def:Q} corresponds to a Lie $\infty$-algebroid over the complex
\begin{equation}\label{semi-resol1}
 \begin{array}{c} \cdots\stackrel{\ell_1}{\longrightarrow}  E_{-3}\stackrel{\ell_1}{\longrightarrow}  E_{-2} \stackrel{\ell_1}{\longrightarrow}\mathfrak g[1]\oplus E_{-1} \stackrel{\rho'}{\longrightarrow}TM\end{array}
\end{equation}
whose brackets satisfy
\begin{enumerate}\item the anchor map $\rho'$ sends an element $x\oplus e\in\mathfrak{g}[1]\oplus E_{-1}$ to $\varrho(x)+\rho(e)\in\varrho(\mathfrak g)+T\mathcal F$, {where $\varrho\colon \mathfrak{g}\longrightarrow \mathfrak{X}(W)$ is the weak symmetry action induced by $\Phi$,} \item the binary bracket satisfies $$\ell_2\left(\Gamma(E_{-1}),\Gamma(E_{-1})\right)\subset \Gamma(E_{-1})\quad \text{and}\quad\ell_2(\Gamma(E_{-1}),x)\subset \Gamma(E_{-1}),\; \forall\, x\in\mathfrak{g}[1]$$\item the $\mathfrak g[1]$-component of the binary bracket on constant sections of $\mathfrak{g}[1]\times M$ is the Lie bracket of $\mathfrak{g}[1]$.\end{enumerate}Conversely, if there exists a Lie $\infty$-algebroid $(E',Q')$ whose underlying complex of vector bundles is of the form $\eqref{semi-resol1}$ and that satisfies item ${1}$, $\emph{2}$ and $\emph{3}$, then there is a Lie $\infty$-morphism $$\Phi\colon(\mathfrak{g}[1],\lb_\mathfrak{g})\rightsquigarrow \left( \mathfrak X_\bullet(E)[1],\lb, \emph{ad}_{Q} \right)$$  which is defined on a given basis $\xi_1,\ldots,\xi_d$ of $\mathfrak{g}$ by:
\begin{equation}\label{expl:lift}
\Phi_{k-1}(\xi_{i_1},\ldots,\xi_{i_k})=\emph{pr}\circ[\cdots[[Q',\iota_{\xi_{i_1}}],\iota_{\xi_{i_2}}],\ldots,\iota_{\xi_{i_k}}]\subset\mathfrak{X}(E)[1],
\;k\in\mathbb{N},\end{equation}
where $\emph{pr}$ stands for the projection map $\mathfrak{X}{(E')}[1]\longrightarrow \mathfrak{X}{(E)[1]}$. 
\end{proposition}
\begin{proof}We explain the idea of the proof. A direct computation gives the first implication. Conversely, let us denote by $Q'$ the homological vector fields of Lie $\infty$-algebroid whose underlying complex of vector bundles is of the form $\eqref{semi-resol1}$. The map defined in Equation \eqref{expl:lift} is indeed a lift into a Lie $\infty$-morphism of the weak symmetry action $\varrho$:

\begin{itemize}
    \item It is not difficult to check that for any $\xi\in\mathfrak{g}[1]$, one has $[Q,\Phi_0(\xi)]=0$.
    \item The fact that $\Phi$ defines a Lie $\infty$-morphism can be found using Voronov trick \cite{Voronov2}, i.e, doing Jacobi's identity inside the null derivation\begin{equation}\label{trick}
        0=\text{pr}\circ[\cdots[[[Q',Q'],\iota_{\xi_{i_1}}],\iota_{\xi_{i_2}}],\ldots,\iota_{\xi_{i_k}}].
    \end{equation}
\end{itemize}
A direct computation of Equation \eqref{trick} falls exactly on the requirements of Definition \ref{def:morph}.

\end{proof}
\begin{remark}
Proposition \ref{alt:thm-res} is stated in the finite dimensional context, i.e., it needs $\mathfrak g$ to be finite dimensional and requires the existence of a geometric resolution for the singular foliation $\mathcal{F}$. The next theorem  proves that: given a weak symmetry action of a Lie algebra $\mathfrak{g}$ (maybe of infinite dimensional) on a Lie-Rinehart algebra $\mathcal{F}\subset \mathfrak{X}(M)$ (we do not require $\mathcal{F}$ being locally finitely generated), such a Lie $\infty$-algebroid {(maybe of infinite dimension in the sense of Definition 1.14 in \cite{CLRL})} with the properties (1), (2) and (3) described at the sections level of the complex \eqref{semi-resol1} in Proposition \ref{alt:thm-res} exists.
\end{remark}
\begin{theorem}\label{alt-thm-res}Let $\varrho\colon\mathfrak{g}\longrightarrow \mathfrak{X}(M)$ be a weak symmetry action on a singular foliation $\mathcal F$. Let $\left((\mathcal{K}_{-i})_{i\geq 1},\dd, \rho\right)$ be a free resolution\footnote{{Possibly of infinite dimension or infinite length.}} of the singular foliation $\mathcal F$ over $M$. 
The complex of trivial vector bundles over $M$
\begin{equation}\label{semi-resol}
  \begin{array}{c} \cdots\stackrel{\dd}{\longrightarrow}  E_{-3}\stackrel{\dd}{\longrightarrow}  E_{-2} \stackrel{\dd}{\longrightarrow}\mathfrak g[1]\oplus E_{-1} \stackrel{\rho'}{\longrightarrow}TM
    \end{array}
\end{equation}
where $\Gamma(E_{-1})=\mathcal{K}_{-i}$, comes equipped with a Lie $\infty$-algebroid structure
\begin{enumerate}
\item whose unary bracket is $\dd$ and whose anchor map $\rho'$, sends an element $x\oplus e\in\mathfrak{g}[1]\oplus E_{-1}$ to $\varrho(x)+\rho(e)\in\varrho(\mathfrak{g})+T\mathcal F$,
\item the binary bracket satisfies $$\ell_2\left(\Gamma(E_{-1}),\Gamma(E_{-1})\right)\subset \Gamma(E_{-1})\quad \text{and}\quad\ell_2(\Gamma(E_{-1}),\Gamma(\mathfrak{g}[1]))\subset \Gamma(E_{-1}),$$
\item the $\mathfrak g[1]$-component of the binary bracket on constant sections of $\mathfrak{g}[1]\times M$ is the Lie bracket of $\mathfrak{g}$.
\end{enumerate}
\end{theorem}
For a proof, see Appendix \ref{proof:Lie-infty}.

\begin{remark}
 When we have $\varrho(\mathfrak{g})\cap T_m\mathcal F=0$ for all $m$ in $M$,  Equation \eqref{semi-resol} is a free resolution of the singular foliation $C^{\infty}(M)\varrho(\mathfrak{g})+\mathcal F$ and we can apply directly the Theorem 2.1 in \cite{CLRL}. Otherwise, we need to show there is no obstruction in degree $-1$ while doing the construction of the brackets if the result still needs to hold. 
\end{remark}

\section{On weak and strict symmetries: an obstruction theory}\label{sec:5}
In this section, we apply theorems in Section \ref{sec:3}
to define a class obstructing the existence of strict symmetry action equivalent to a given weak symmetry action. We apply these results to the problem of extending a strict Lie algebra action on an affine sub-variety to the ambient space. \\

Let us start with some generalities that we will use throughout of this section. Assume we are given\begin{itemize}
\item a Lie algebra $\mathfrak g$ with Lie bracket $\lb_\mathfrak g$,
 \item a universal Lie $\infty$-algebroid $(E,Q)$  of a singular foliation $\mathcal{F}$, 
\item a weak symmetry action $\varrho\colon \mathfrak g\longrightarrow \mathfrak X(M)$ of $\mathfrak g$ on a singular foliation $\mathcal{F}$, together  with $\eta\colon\wedge^2\mathfrak{g}\longrightarrow \Gamma(E_{-1})$   such that $x,y\in\mathfrak{g}$\begin{equation}\label{eq:def:eta}
    \varrho([x,y]_\mathfrak{g})-[\varrho(x),\varrho(y)]=\rho(\eta(x,y)).
\end{equation}

 \end{itemize}  
Theorem \ref{main} states that $\varrho\colon\mathfrak g\rightarrow \mathfrak{X}(M)$ admits a lift to a Lie $\infty$-morphism \begin{equation}
    \label{eq:Lie-morp-obstruction-theory}\Phi\colon (\mathfrak{g}[1],\lb_\mathfrak{g})\rightsquigarrow (\mathfrak X_\bullet(E)[1],\lb,\text{ad}_{Q}).
\end{equation}Equivalently, if $\mathfrak{g}$ is of finite dimension, the Lie $\infty$-morphism \eqref{eq:Lie-morp-obstruction-theory} corresponds (by Proposition \ref{alt:thm-res}) to a Lie $\infty$-algebroid $(E',Q')$ over $M$ such that
    
 \begin{itemize}
     \item $(E, Q)$ is included as a sub-Lie $\infty$-algebroid in a Lie algebroid $(E', Q')$ over $M$,\item its underlying complex is, $E'_{-1}:=~\mathfrak{g}[1]\oplus E_{-1}$, and for any $i\geq 2$, $E_{-i}'=E_{-i}$, 
    namely \begin{equation}\label{semi-res}
  \begin{array}{c} \cdots\stackrel{\dd=\ell_1}{\longrightarrow}  E_{-3}\stackrel{\dd=\ell_1}{\longrightarrow}  E_{-2} \stackrel{\dd=\ell_1}{\longrightarrow}\mathfrak g[1]\oplus E_{-1} \stackrel{{\rho'}}{\longrightarrow}TM,
    \end{array}
\end{equation}\item we have,  $$\ell'_2(x\oplus 0,y\oplus 0)=[x,y]_{\mathfrak{g}}\oplus \eta(x,y)$$ 
and $$\ell_2'(x,\Gamma(E_{-1}))\subset\Gamma(E_{-1})$$ for all $x\in\mathfrak{g}[1]$.
\end{itemize}   

\begin{remark}
It is important to notice that the Lie $\infty$-algebroid $(E',Q')$ can be constructed  even if $\mathfrak{g}$ and $(E,\dd)$ are of infinite dimensions (see Theorem \ref{alt-thm-res}).
\end{remark}


\begin{remark}
In Equation \eqref{semi-res}, the complex $(E,\ell_1)$ can be chosen to be  minimal at a point $m\in M$, i.e., ${\ell_1}_{|_m}=0$, provided that a geometric resolution of $\mathcal{F}$ exists. By Proposition 4.14 in \cite{LLS} the isotropy Lie algebra $\mathfrak{g}_m=\frac{\mathcal{F}(m)}{\mathcal{I}_m\mathcal{F}}$ of the singular foliation $\mathcal{F}$ at the point $m\in M$ is isomorphic to $\ker(\rho_m)$. This allows to denote the latter space also by $\mathfrak{g}_m$.
\end{remark}

\begin{lemma}\label{class}
{Let $m\in M$ be a fixed point of the $\mathfrak g$-action $\varrho$}. 
Assume that the underlying complex $(E,\ell_1)$ is minimal at a point $m$, i.e., ${\ell_1}_{|_m}=0$. The map
\begin{equation}\label{eq:action-CE}
    {\nu}\colon\mathfrak{g}\longrightarrow \emph{End}\left(\mathfrak{g}_m\right),\;x\longmapsto\ell_2'(x\,,\cdot)_{|_m}
\end{equation} satisfies \begin{itemize}
    \item[(a)] $\nu([x,y]_\mathfrak{g})-[\nu(x),\nu(y)]+\ell_2(\,\cdot,\eta(x,y))_{|_m}=0$,
    \item[(b)] $\nu(z)\left(\eta(x,y)_{|_m}\right)-\eta([x,y]_\mathfrak{g},z)_{|_m}+\circlearrowleft(x,y,z)=0$.
\end{itemize}

\end{lemma}
\begin{proof}
The map $\nu$ in \eqref{eq:action-CE} is well-defined since $\varrho_{|_m}=0$ and $\mathfrak{g}_m=\ker \rho_m$. The Jacobi identity on elements $x,y\in \mathfrak{g}[1],\,e\in \Gamma(E_{-1})$, evaluated at the point $m$, implies that $$\nu({[x,y]_\mathfrak{g}})(e_{|_m})-[\nu(x),\nu(y)](e_{|_m})+\ell_2(\eta(x,y),e)_{|_m}=0.$$ This proves item (a). Likewise,  Jacobi identity on  elements $x,y,z\in\mathfrak{g}[1]$ together with  ${\ell_1}_{|_m}=0$ give:\begin{align*}
    \ell'_2(\ell_2'(x,y),z)_{|_m} +\circlearrowleft(x,y,z)=0\,&\Longrightarrow\,\ell'_2([x,y]_\mathfrak{g},z)_{|_m}+\ell'_2(\eta(x,y),z)_{|_m} +\circlearrowleft(x,y,z)=0,\\&\Longrightarrow \,\nu(z)\left(\eta(x,y)_{|_m}\right)-\eta([x,y]_\mathfrak{g},z)_{|_m}+\circlearrowleft(x,y,z)=0.
\end{align*}
Here we have used the definition of $\ell'_2$ on degree $-1$ elements and Jacobi identity for the bracket $[\cdot\,,\cdot]_\mathfrak{g}$. This proves item (b).
\end{proof}

By Lemma \ref{class}, $\mathfrak{g}_m$ is equipped with a $\mathfrak{g}$-module structure when $\eta(x,y)_{|_m}$ is  for all $x,y\in \mathfrak{g}$ valued in the center of the isotropy Lie algebra $\mathfrak{g}_m$. {
The following proposition is built on this last point. It defines an obstruction class mentioned earlier in the introduction of the paper as an obstruction of the possibility of turning a weak symmetry action into a strict symmetry action. Recall that $\eta$ is defined by Equation \eqref{eq:def:eta}. Notice that if $m\in M$ is a fixed point of the $\mathfrak g$-action $\varrho$, then  this implies in particular that $\eta(x,y)|_m\in \ker \rho_m$.}

\begin{proposition}\label{Prop:class}
{Let $m\in M$ be a fixed point of the $\mathfrak g$-action $\varrho$}. Assume that 
\begin{itemize}
\item the underlying complex $(E,\ell_1)$  of $(E,Q)$ is minimal at $m$,
      \item for all $x,y\in \mathfrak{g}$, $\eta(x,y)_{|_m}$\,is valued in the center\footnote{In particular, when the $2$-ary bracket $\ell_2$ is zero at $m$ when applied to two elements of degree $-1$, i.e., $\mathfrak{g}_m$ is Abelian, we have $Z(\mathfrak{g}_m)=\mathfrak{g}_m$.} $Z(\mathfrak{g}_m)$ of  $\mathfrak{g}_m$.
\end{itemize}
 Then,
\begin{enumerate}
    \item {For every $x\in \mathfrak g[1]$, $\ell_2'(x,\cdot\,)$ preserves $Z(\mathfrak{g}_m)$}, and the restriction of the $2$-ary bracket 
\[\ell_2'\colon\mathfrak{g}\otimes Z(\mathfrak{g}_m)\longrightarrow Z(\mathfrak{g}_m)\]endows $Z(\mathfrak{g}_m)$ with a $\mathfrak{g}$-module structure which does not
depend neither on the choice of a weak symmetry action  $\varrho$ nor of a universal Lie $\infty$-algebroid of $\mathcal{F}$, nor of the Lie $\infty$-morphism $\Phi\colon \mathfrak{g}[1]\rightsquigarrow\mathfrak{X}(E)[1]$.
\item the restriction of the map $\eta\colon\wedge^2\mathfrak{g}\longrightarrow {E_{-1}}$ at $m$ \[\eta_{|_m}\colon\wedge^2\mathfrak{g}\longrightarrow Z(\mathfrak{g}_m)\]is a $2$-cocycle for the Chevalley-Eilenberg complex of $\mathfrak{g}$ valued in $Z(\mathfrak{g}_m)$,
\item the cohomology class of this cocycle does not depend on the representatives of the equivalence class of $\varrho$, nor on the choices made in the construction,
\item  if $\varrho$ is equivalent to a strict symmetry action, then $\eta_{|_m}$ is exact.
\end{enumerate}
\end{proposition}
\begin{proof}
{$(E, \ell_1)$ being minimal at $m$,  $\ell_2'|_m$ satisfies the Jacobi identity. In particular,  for every $x\in \mathfrak g[1]$, $\ell_2'(x,\cdot\,)|_m$ preserves $Z(\mathfrak{g}_m)$. By item $(a)$ of Lemma \ref{class}, the restriction of the $2$-ary bracket 
\[\ell_2'\colon\mathfrak{g}\otimes Z(\mathfrak{g}_m)\longrightarrow Z(\mathfrak{g}_m)\] endows $Z(\mathfrak{g}_m)$ with a $\mathfrak{g}$-module structure, since $\ell_2(\cdot,\,\eta(x,y))|_m=0$ by assumption. It is easy to see that if we change the action $\varrho$ to $\varrho+\rho\circ\beta$  for some vector bundle morphism $\beta\colon\mathfrak{g}\longrightarrow E_{-1}$ such that $\beta|_m\colon\mathfrak{g}\longrightarrow Z(\mathfrak{g}_m)$, the new $2$-ary bracket between sections of $\mathfrak{g}[1]$ and $E_{-1}$ constructed as in the  proof of Theorem \ref{alt-thm-res} is modified by $(x,e)\mapsto\ell_2'(x,e)+\ell_2(\beta(x),e)$. The second term of the latter vanishes at $m$, by definition of $\beta|_m$. As a result, the action of $\ell_2'$ on $Z(\mathfrak{g}_m)$ does not depend on the choices made in the construction. This proves  item ${1}$.}

{Item ${2}$ follows from item (b) of Lemma \ref{class} that tells that $\eta_{|_m}\colon\wedge^2\mathfrak{g}\longrightarrow Z(\mathfrak{g}_m)$ is a $2$-cocycle for the Chevalley-Eilenberg complex of $\mathfrak{g}$ valued in $Z(\mathfrak{g}_m)$.}

{Let $\varrho'$ be a weak symmetry action of $\mathfrak{g}$ on $\mathcal{F}$ which is equivalent to $\varrho$, i.e., there exists a vector bundle morphism $\beta\colon\mathfrak{g}\longrightarrow E_{-1}$ (with $\beta|_m\colon\mathfrak{g}\longrightarrow Z(\mathfrak{g}_m)$) such that $\varrho'(x)=\varrho(x)+\rho(\beta(x))$ for all $x\in\mathfrak{g}$. Let $\eta'\colon\wedge^2\mathfrak{g}\longrightarrow E_{-1}$ be such that $\varrho'([x,y]_\mathfrak{g})-[\varrho'(x),\varrho'(y)]=\rho(\eta'(x,y))$ for all $x,y\in\mathfrak{g}$. Following the constructions in the proof of Theorem \ref{alt-thm-res}, this implies that\begin{equation}\label{eq:eta-exa}
    \eta'(x,y)=\eta(x,y)+\beta([x,y]_\mathfrak{g})- \ell_2'(x,\beta(y))+\ell'_2(y,\beta(x))-\ell_2(\beta(x),\beta(y))),\quad \text{for all $x,y\in\mathfrak{g}$}.
\end{equation}
Equation \eqref{eq:eta-exa} implies that $\eta'(x,y)|_m-\eta(x,y)|_m=\dd^{CE}(\beta|_m)(x,y)$, where $\dd^{CE}$ stands for the Chevalley-Eilenberg differential. As a consequence, $\eta'|_m$ and $\eta|_m$ define the same class in the Chevalley-Eilenberg complex of $\mathfrak{g}$ valued in $Z(\mathfrak{g}_m)$. This proves item $3$ and $4$.}
\end{proof}

\begin{remark}
Even when ${\ell_2}_{|_m}\neq 0$, we can have some obstruction, but they are not given by cohomology classes because they are not given by linear equations. More precisely, it is obvious that the weak symmetry action $\varrho$ is not equivalent to strict one if the Maurer-Cartan-like equation \eqref{eq:eta-exa} has no solution $\beta$ with ${\eta'}_{|_m}=0$.
\end{remark}

The following is a direct consequence of Proposition \ref{Prop:class}.
\begin{corollary}\label{prop:class}
Let $m\in M$ be a fixed point for the $\mathfrak g$-action $\varrho$. 
Assume that the isotropy Lie algebra $\mathfrak{g}_m$ of $\mathcal F$ at $m$ is Abelian. Then

 \begin{enumerate}
    \item $\mathfrak{g}_m$ is a $\mathfrak{g}$-module.
    \item The bilinear map, $\eta_{|_m}\colon\wedge^2\mathfrak g\to \mathfrak{g}_m$, is a Chevalley-Eilenberg  $2$-cocycle of $\mathfrak g$ valued in $\mathfrak{g}_m$. \item Its  class $\emph{cl}(\eta|_m)\in H^2(\mathfrak g,\mathfrak{g}_m)$ does not depend on the choices made in the construction.
    \item Furthermore, $\emph{cl}(\eta|_m)$ is an obstruction of having a strict symmetry action equivalent to ~$\varrho$.
\end{enumerate}
\end{corollary}

\begin{example}
We return to Example \ref{ex:isotropy} and consider the $\mathfrak{g}^k_m$-action on $\mathcal{I}^{k+1}_m\mathcal{F}$. Every point $m\in M$ is a fixed point for the $\mathfrak{g}^k_m$-action of item ${2}$ of Example \ref{ex:isotropy}. Since the isotropy Lie algebra $\mathfrak{g}^k_m$ is Abelian for every $k\geq 2$ the following assertions hold by Corollary \ref{prop:class}:

\begin{enumerate}
\item For each $k\geq 1$, the vector space $\mathfrak{g}_m^{k+1}$ is a $\mathfrak{g}_m^{k}$-module.
\item The obstruction of having a strict symmetry action equivalent to $\varrho_k$ is a Chevalley-Eilenberg cocycle valued in $\mathfrak{g}_m^{k+1}$.
    \end{enumerate}
\end{example}
Here is a particular case of this example.
\begin{example}
Let $\mathcal{F}:=\mathcal{I}_0^3\mathfrak{X}(\mathbb{R}^n)$ be the singular foliation generated by vector fields vanishing to order $3$ at the origin. The quotient $\mathfrak{g}:=\frac{\mathcal{I}_0^2\mathfrak{X}(\mathbb{R}^n)}{\mathcal{I}_0^3\mathfrak{X}(\mathbb{R}^n)}$ is a trivial Lie algebra. There is a weak symmetry action of $\mathfrak{g}$ on $\mathcal{F}$ which assigns to an element in $\mathfrak{g}$ a representative in $\mathcal{I}^2_0\mathfrak{X}(\mathbb{R}^n)$. In this case, the isotropy Lie algebra of $\mathcal{F}$ at zero is Abelian  and $\ell'_2(\mathfrak{g},\mathfrak{g}_0)_{|_0}=0$. Thus, the  action  of $\mathfrak{g}$ on $\mathfrak{g}_0$ is trivial. One can choose $\eta\colon\wedge^2\mathfrak{g}\longrightarrow \mathfrak{g}_0$ such that $\eta\left(\overline{x_i^2\frac{\partial}{\partial x_i}},\overline{x_i^2\frac{\partial}{\partial x_j}}\right)=2e_{ij}$, with $e_{ij}$ a constant section in a set of generators of degree $-1$ whose image by the anchor is $x_i^3\frac{\partial}{\partial x_j}$. Therefore, $\eta_{|_0}\left(\overline{x_i^2\frac{\partial}{\partial x_i}},\overline{x_i^2\frac{\partial}{\partial x_j}}\right)\neq 0$. This implies that the class of $\eta$ is not zero at the origin. Therefore, by item ${2}$ of Corollary \ref{prop:class} the weak symmetry action of $\mathfrak{g}$ on $\mathcal{F}$ is not equivalent to a strict one.
\end{example}
Also, we have the following consequence of Corollary \ref{prop:class} for Lie algebra actions on affine varieties, as in Example \ref{ex:affine-action}. Before going to Corollary \ref{cor:affine} let us write definitions and some facts.\\

\noindent
\textbf{Settings:} Let $W$ be an affine variety realized as a subvariety of $\mathbb{C}^d$, and defined by some ideal $\mathcal I_W\subset\mathbb{C}[x_1,\ldots,x_d]$. We denote by $\mathfrak{X}(W):=\mathrm{Der}(\mathcal O_W)$ the Lie algebra of vector fields on $W$, where $\mathcal O_W$  is  coordinates ring of $W$. 
\begin{definition}
A point $p\in W$ is said to be $\emph{strongly singular}$ if for all $f\in \mathcal I_W$, $\dd_pf\equiv0$ or equivalently if for all $f\in \mathcal I_W$ and $X\in \mathfrak{X}(\mathbb{C}^d)$, one has $X[f](p)\in \mathcal I_p$.
\end{definition}

\begin{example}
Any singular point of a hypersurface $W$ defined by a polynomial $\varphi\in\mathbb{C}[x_1,\ldots,x_d]$ is strongly singular.
\end{example}
The lemma below is immediate.
\begin{lemma}\label{lemma:strong-sing}
In a strongly singular point, the isotropy Lie algebra of the singular foliation $\mathcal{F}=\mathcal I_W\mathfrak{X}(\mathbb{C}^d)$ is Abelian.
\end{lemma}
{The following corollary answers the question of Example \ref{ex:affine-action}. Here, $\emph{cl}(\eta|_p)$ is as in Corollary \ref{prop:class}.}
\begin{corollary}\label{cor:affine}
Let $\varrho\colon\mathfrak{g}\longrightarrow \mathfrak{X}(W) $ be a Lie algebra morphism.
\begin{enumerate}
    \item Any extension $\widetilde{\varrho}$ as in Example \ref{ex:affine-action} is a weak symmetry action for the singular foliation $\mathcal{F}=\mathcal I_W\mathfrak{X}(\mathbb{C}^d)$.
    
    \item 
    {Let $p\in W$ be a fixed point for the $\mathfrak{g}$-action $\varrho$ which is also a strongly singular point $p$ in $W$. If class $cl(\eta|_p)$ does not vanish, then the Lie algebra morphism $\varrho\colon\mathfrak{g}\longrightarrow \mathfrak{X}(W)$ can not be extended to a Lie algebra morphism $\widetilde \varrho\colon\mathfrak{g}\longrightarrow \mathfrak X(\mathbb C^d)$.}
\end{enumerate}
\end{corollary}
\begin{proof}
{This first item follows from Example \ref{ex:affine-action}. By Lemma \ref{lemma:strong-sing}, the isotropy Lie algebra of $\mathcal{F}=\mathcal I_W\mathfrak{X}(\mathbb{C}^d)$ is Abelian in every strongly singular points of an affine variety $W$. Thus, item 2 of Corollary \ref{cor:affine} follows from Corollary \ref{prop:class}.}
\end{proof}

Let us give examples of Lie algebra actions on an affine variety that do not extend to the ambient space.
\begin{example}

Let $W\subset\mathbb{C}^2$ be the affine variety generated by the polynomial $\varphi=FG$ with $F,G \in\mathbb{C}[x,y]=:\mathcal{O}$. We consider the vector fields $U=F\mathcal{X}_G,\, V=G\mathcal{X}_F\in\mathfrak X(\mathbb{C}^2)$, where $\mathcal{X}_F$ and $\mathcal{X}_G$ are Hamiltonian vector fields w.r.t the Poisson structure $\{x,y\}:=1$. Note that $U, V$  are tangent to $W$, i.e. $U[\varphi], V[\varphi]\in \langle \varphi \rangle$. It is easily checked that $[U,V]=\varphi \mathcal{X}_{\{F,G\}}$.

The action of the trivial Lie algebra  $\mathfrak{g}=\mathbb{R}^2$ on $W$ that sends its canonical basis $(e_1, e_2)$ to $U$, and $V$ respectively, is a weak symmetry action on the singular foliation $\mathcal{F}^\varphi:=\langle\varphi\rangle\mathfrak{X}(\mathbb{C}^2)$ of vector fields vanishing on $W$, and induces a Lie algebra map, \begin{equation}\label{eq:LA-Mor-AV}
   \varrho\colon\mathfrak g\longrightarrow \mathfrak{X}(W)
.\end{equation}
A universal Lie $\infty$-algebroid of $\mathcal{F}^\varphi$ is a Lie algebroid (see Example 3.19 of \cite{CLRL}) because,
$$ \xymatrix{0\ar[r] & \mathcal{O}\mu \otimes_\mathcal O \X(\mathbb C^2) \ar^{\varphi\frac{\partial}{\partial\mu}\otimes_\mathcal O \text{id}}[rr] & & \mathcal{F}^\varphi}$$
is a $\mathcal{O}$-module isomorphism, ($\mathcal{F}_\varphi$ is a projective module). Here, $\mu$ is a degree $ -1$ variable, so that $\mu^2=0$. The universal algebroid structure over that resolution is given on the set of generators by: \begin{equation}\label{eq:bracket2}
    \ell_2\left(\mu\otimes_\mathcal O \dfrac{\partial}{\partial x},\mu\otimes_\mathcal O \dfrac{\partial}{\partial y}\right):=\frac{\partial\varphi}{\partial x}\,\mu\otimes_\mathcal O \dfrac{\partial}{\partial y}-\frac{\partial\varphi}{\partial y}\,\mu\otimes_\mathcal O \dfrac{\partial}{\partial x}\end{equation} and $\ell_k:=0$ for every $k\geq 3$. Since $\mathcal{X}_{\{F,G\}}=\dfrac{\partial \{F,G\}}{\partial y}\dfrac{\partial }{\partial x}-\dfrac{\partial \{F,G\}}{\partial x}\dfrac{\partial }{\partial y} $, we have \begin{equation}\label{eta:Poisson}
         \eta(e_1,e_2):=\dfrac{\partial \{F,G\}}{\partial y}\,\mu\otimes_\mathcal{O}\dfrac{\partial}{\partial x}-\dfrac{\partial \{F,G\}}{\partial x}\,\mu\otimes_\mathcal{O}\dfrac{\partial}{\partial y}.\end{equation}
Take, for example,  $F(x,y)= y-x^2$ and $G(x,y)=y+x^2$. The isotropy Lie algebra $\mathfrak{g}_{(0,0)}$ of $\mathcal{F}^\varphi$ is Abelian by Equation \eqref{eq:bracket2}, i.e. $\ell_2|_{(0,0)}=0$. By Corollary \ref{prop:class} (1), $\mathfrak{g}_{(0,0)}$ is a $\mathbb{R}^2$-module. A direct computation shows that the action on $\mathfrak{g}_{(0,0)}$ is not trivial, but takes values in $\mathcal{O}\, \mu\otimes_\mathcal O \dfrac{\partial}{\partial x}$. Besides, Equation \eqref{eta:Poisson} applied to $\{F,G\}=4x$ gives \begin{equation}\label{eta:Poisson2}
            \eta(e_1,e_2)=-4\,\mu\otimes_\mathcal O \dfrac{\partial}{\partial y}.
\end{equation}
\noindent
If $\eta_{|_{(0,0)}}$ were a coboundary of Chevalley-Eilenberg, we would have (in the notations of Proposition \ref{Prop:class}) that \begin{equation}\label{eq:nonzero-class}
    \eta(x,y)_{|_{(0,0)}}=\beta([x,y]_\mathfrak{\mathbb{R}^2})- \ell_2'(x,\beta(y))+\ell'_2(y,\beta(x))\in \mathcal{O}\, \mu\otimes_\mathcal O \dfrac{\partial}{\partial x},\quad \text{for all $x,y\in \mathfrak{g}$}
\end{equation} for some linear map\footnote{Here, $\ell_2'$ is as Proposition \ref{Prop:class}.} $\beta\colon  \mathfrak{g}\longrightarrow \mathfrak{g}_{(0,0)}$. Therefore, Equation \eqref{eq:nonzero-class} has no solution in view of Equation \eqref{eta:Poisson2}, since $\eta_{|_{(0,0)}}\neq 0$. In other words, its class $\mathrm{cl}(\eta)$ does not vanish at $(0,0)$. By Corollary \ref{cor:affine} (2), the action $\varrho$ given in Equation \eqref{eq:LA-Mor-AV} cannot be extended to the ambient space.
\end{example}

\section{Bi-submersion towers and symmetries}
We end the paper by introducing the notion \textquotedblleft   bi-submersion towers\textquotedblright. The work contained in this section is entirely original, except for the notion given in Definition \ref{def:tower} that arose in a discussion between C. Laurent-Gengoux, L.~ Ryvkin, and I, and will be the object of a separate study.
\subsection{Definitions and existence}Let us firstly recall the  definition of bi-submersion introduced in \cite{AndroulidakisIakovos}.

\begin{definition}
Let $M$ be a manifold endowed with a singular foliation $\mathcal F $.
A \emph{bi-submersion} $\xymatrix{B\ar@/^/[r]^{s}\ar@/_/[r]_{t}&M}$ over $\mathcal F$ is a triple $(B,s,t) $ where:
\begin{itemize}
    \item $B$ is a manifold,
    \item $s,t\colon B \to M$ are 
    submersions, respectively called \emph{source} and \emph{target},
\end{itemize}
such that the pull-back singular foliations $ s^{-1} \mathcal F$ and $t^{-1} \mathcal F $ are both equal to the space of vector fields of the form $\xi+\zeta $ with $\xi\in \Gamma(\ker (\dd s)) $ and $\zeta \in \Gamma(\ker (\dd t))$. Namely,

\begin{equation}
\label{eq:defbisub}
  s^{-1} \mathcal F=t^{-1} \mathcal F = \Gamma(\ker (\dd s))+\Gamma(\ker (\dd t)).
\end{equation}
In that case, we also say that $(B,s,t)$ is a bi-submersion over $(M, \mathcal{F})$.
\end{definition}

\begin{example}\label{ex:holonomy-biss}
Let $\mathcal{F}$ be a singular foliation on a manifold $M$. Let $x\in M$. Let $X_1,\ldots, X_n$ be vector fields in $\mathcal{F}$ whose class in $\mathcal{F}_x:=\mathcal{F}/\mathcal{I}_x\mathcal{F}$ generate the latter. We know from  \cite{AndroulidakisIakovos} that there is an open neighborhood $\mathcal{W}$ of $(x, 0)\in M \times \mathbb R^n$ such that $(\mathcal{W}, t, s)$ is a bi-submersion over $\mathcal{F}$, here\begin{equation}
    s(x,y)=x\quad  \text{and}\quad \displaystyle{t (x, y)=\exp_x\left(\sum_{i=1}^ny_iX_i\right)}=\varphi_1^{\sum_{i=1}^ny_iX_i}(x)
\end{equation}
where for $X\in\mathfrak X(M)$,\;$\varphi^X_1$ denotes the time-1 flow of $X$. 
 Such bi-submersions are called \emph{path holonomy bi-submersions} \cite{AndroulidakisIakovos2}.
 \end{example}

Now we can introduce the following definition.
\begin{definition}\label{def:tower}
A \emph{bi-submersion tower over a singular foliation $\mathcal{F}$ on $M$} is a (finite or infinite) sequence of manifolds and maps as follows
\begin{equation}\label{eq:tower}\mathcal{T}_B:\xymatrix{\cdots\phantom{X}\ar@/^/[r]^{s_{i+1}}\ar@/_/[r]_{t_{i+1}}&{B_{i+1}}\ar@/^/[r]^{s_{i}}\ar@/_/[r]_{t_{i}}&{B_{i}}\ar@/^/[r]^{s_{i-1}}\ar@/_/[r]_{t_{i-1}}&{\cdots}\ar@/^/[r]^{s_1}\ar@/_/[r]_{t_1}&B_1\ar@/^/[r]^{s_0}\ar@/_/[r]_{t_0}&B_0,}
\end{equation}
together with a sequence $\mathcal F_i $ of singular foliations on $B_i$, with the convention that $B_0=M$ and $\mathcal F_0 = \mathcal F $, such that
\begin{itemize}
\item for all $i \geq 1$, $\mathcal F_i\subset\Gamma(\ker ds_{i-1})\cap\Gamma(\ker dt_{i-1})$, 
\item for each $i\geq 1$, $\xymatrix{{B_{i+1}}\ar@/^/[r]^{s_{i}}\ar@/_/[r]_{t_{i}}&{B_{i}}}$ is a bi-submersion over $\mathcal F_i$.
    
\end{itemize}
A bi-submersion tower over $(M,\mathcal F) $ shall be denoted as $(B_{i+1},s_i,t_i,\mathcal F_i)_{i \geq 0}$.
The bi-submersion tower over $\mathcal{F}$ in \eqref{eq:tower} is said to be of \emph{of length $n\in\mathbb{N}$} if $B_j=B_n, s_j=t_j=\mathrm{id}\;$ and $\mathcal{F}_j=\{0\}$ for all $j\geq n$.
\end{definition}

\begin{remark}\label{Rmk:Geom_consequences}
Let us spell out some consequences of the axioms. For $i\geq 1$,
two points $b,b'\in B_i$ of the same leaf of $\mathcal F_i$ satisfy $s_{i-1}(b)=s_{i-1}(b')$ and $t_{i-1}(b)=t_{i-1}(b')$. Also, for all $b\in B_i,\; T_b\mathcal{F}_i\subset (\ker ds_{i-1})_{|_b}\cap(\ker dt_{i-1})_{|_b}$.
\end{remark}

Let us explain how such towers can be constructed out of a singular foliation. Let $\mathcal{F}$ be a singular foliation on $M$. Then, 
\begin{enumerate}
\item By Proposition 2.10 in \cite{AndroulidakisIakovos}, there always exists a bi-submersion $\xymatrix{B_1\ar@/^/[r]^{s_0}\ar@/_/[r]_{t_0}&M}$ over $\mathcal F$.
\item The $C^\infty(B_1)$-module $\Gamma(\ker ds_{0})\cap\Gamma(\ker dt_{0}) $ is closed under Lie bracket. When it is locally finitely generated, it is a singular foliation on $B_1$.  
Then, it admits a bi-submersion $\xymatrix{B_2\ar@/^/[r]^{s_1}\ar@/_/[r]_{t_1}&B_1}$.
We now have obtained the two first terms of a bi-submersion tower.
\item We can then continue this construction provided that $\Gamma(\ker ds_{1})\cap\Gamma(\ker dt_{1}) $ is locally finitely generated as a $C^\infty(B_2)$-module, and that it is so at each step\footnote{In the real analytic case, the module $\Gamma(\ker ds_1)\cap\Gamma(\ker dt_1)$ is locally finitely generated because of the noetherianity of the ring of germs of real analytic functions \cite{FrischJacques,Yum-TongSiu}.}. 
\end{enumerate}

\begin{definition}
A bi-submersion tower  $ \mathcal{T}_B=(B_{i+1},s_i,t_i,\mathcal F_i)$  over $ (M,\mathcal F)$ is called \emph{exact bi-submersion tower over $ (M,\mathcal F)$} when $ \mathcal F_{i+1}= \Gamma(\ker ( ds_i))\cap  \Gamma(\ker ( dt_i))$ for all $i \geq 0$.
It is called a \emph{path holonomy bi-submersion tower} (resp. \emph{path holonomy atlas bi-submersion tower}) if  $\xymatrix{B_{i+1}\ar@/^/[r]^{s_i}\ar@/_/[r]_{t_i}&B_i}$ is a path holonomy bi-submersion (resp. an Androulidakis-Skandalis' path holonomy atlas\footnote{See \cite{AndroulidakisIakovos}, Example  3.4 (3).} over $\mathcal F_i $ for each $i \geq 0$. When a path holonomy bi-submersion tower is exact, we speak of \emph{exact path holonomy bi-submersion tower}.
\end{definition}

The following theorem gives a condition which is equivalent to the existence of a bi-submersion tower over a singular foliation. The proof uses Proposition \ref{prop:left-arrows} and Lemma \ref{lemma:invariant-vector-fields} which are stated in the next section. 
\begin{theorem}\label{thm:equivalence}
Let $\mathcal F$ be a singular foliation on $M$. The following items are equivalent:
\begin{enumerate}
    \item $\mathcal F$ admits a geometric resolution.
    \item There exists an exact path holonomy bi-submersion tower over $(M, \mathcal F)$.
\end{enumerate}
\end{theorem}
\begin{convention}\label{conven:proj}For a submersion $\phi\colon M\longrightarrow N$ and a smooth map $ \psi\colon M\longrightarrow N$, we denote by ${}^\psi\Gamma(\ker d\phi)$ the space of $\psi$-projectable vector fields in $\Gamma(\ker d\phi)\subset \mathfrak{X}(M)$.\end{convention}
\begin{proof}
\noindent
$(1)\Rightarrow (2)$ : Assume that $\mathcal F$ admits a geometric resolution $(E, \dd, \rho)$. 
In particular, $(E_{-1}, \rho)$ is an anchored bundle over $\mathcal{F}$. {We need to show by recursion on $i\geq 0$ that $\Gamma(\ker ds_i
    )\cap\Gamma(\ker dt_i
    )$ is locally finitely generated because $\ker\dd^{(i+1)}$ or $\ker\rho$ is locally  finitely generated. We actually repeat at each step  $i\geq 0$, the general fact  that the pull-back complex of vector bundle by the submersion $\varphi=t_0\circ t_1\circ \cdots \circ t_i\colon B_{i+1}\longrightarrow M$ 
\begin{equation}\label{eq:pull-back-geo-resol}
         \xymatrix{\cdots\ar[r]& \varphi ^*E_{-i-3}\ar[r]^{\tiny{\varphi ^*\dd^{(i+3)}}}&\varphi ^*E_{-i-2}\ar[r]^{\tiny{\varphi ^*\dd^{(i+2)}}}&\varphi ^*E_{-i-1}\ar[r]& TB_{i+1} }
     \end{equation} remains exact at the sections level (at degree\footnote{We shall  understand that the degree of elements of $\varphi ^*E_{-i-j}$ is $-j$.} $\leq -1$), since $C^\infty(B_{i+1})$ is a flat $C^\infty(M)$-module. In addition, for $i\geq 1,$ the complex \eqref{eq:pull-back-geo-resol} defines a geometric resolution of $\mathcal{F}_{i+1}$.}
    
Let $(B_1, s_0, t_0)$  be a path holonomy bi-submersion over $(M,\mathcal{F})$. Consider the map 
\begin{align}\label{eq:induced-map-tower}
    R\colon\Gamma(t^*_0E_{-1})&\longrightarrow {}^{t_0}\Gamma(\ker ds_0
    )\subset\mathfrak{X}(B_1)\\\nonumber t^*_0e\;&\longmapsto \overrightarrow{e}
\end{align}
defined as in Proposition \ref{prop:left-arrows}. {By Lemma \ref{lemma:invariant-vector-fields} (1), the map $R$ in \eqref{eq:induced-map-tower}  comes from a vector bundle morphism $t_0^* E_{-1} \longrightarrow \ker ds_0$ 
 and is surjective on an open subset $V_1\subset B_1$, by item $2$ of Lemma \ref{lemma:invariant-vector-fields}}. 
 In particular the map $R$ in \eqref{eq:induced-map-tower} restricts to a surjective map
\begin{align}\label{eq:restriction-map-tower}
    \ker(t_0^*\rho)&\longrightarrow \ker\left( dt_0|_{\Gamma(\ker ds_0)}\right)=\Gamma(\ker ds_0
    )\cap\Gamma(\ker dt_0
    )\subset\mathfrak{X}(B_1)
\end{align}
By exactness in degree $-1$, $\ker\rho=\dd^{(2)}(\Gamma(E_{-2}))$. Therefore,  $\ker (t^*_0\rho)$ is locally finitely generated. By surjectivity of the map \eqref{eq:restriction-map-tower}, $\Gamma(\ker ds_0
    )\cap\Gamma(\ker dt_0)=:\mathcal{F}_1$ is also locally finitely generated on $V_1$, in particular $\mathcal{F}_1$ is a singular foliation on $V_1\subset B_1$ (we may assume that $B_1=V_1$). Thus, one can take a path holonomy bi-submersion $(B_2,s_1,t_1)$ over $(B_1,\mathcal{F}_1)$. The proof continues the same as the previous step.

    Let us  make a step further for clarity. The composition \begin{equation*}
        \xymatrix{\Gamma(E_{-2})\ar@{.>>}@/_1pc/[rr]\ar@{->}[r]^<<<<{\dd^{(2)}}&\mathrm{im}(\dd^{(2)})=\ker\rho\ar@{->>}[r]&\mathcal{F}_1}
    \end{equation*} together with $E_{-2}$ is an anchored bundle over $\mathcal{F}_1$. Let $\varphi=t_0\circ t_1$.  Just like in the first step,  define the surjective $C^\infty(B_2)$-linear map,
    \begin{align}\label{eq:induced-map-tower2}
    \Gamma(\varphi^*E_{-2})&\longrightarrow {}^{t_1}\Gamma(\ker ds_1
    )\subset\mathfrak{X}(B_2)\\\nonumber \varphi^*e\;&\longmapsto \overleftarrow{e}.
\end{align}
By Lemma \ref{lemma:invariant-vector-fields} (2), the map \ref{eq:induced-map-tower2} restricts (upon taking $B_2$ smaller) to a surjective map
\begin{align}\label{eq:restriction-map-tower2}
    \ker\left(\Gamma(\varphi^*E_{-2})\stackrel{\varphi^*\dd^{(2)}}{\longrightarrow} \Gamma(\varphi^*E_{-1})\right)&\longrightarrow \ker\left( dt_1|_{\Gamma(\ker ds_1)}\right)=\Gamma(\ker ds_1
    )\cap\Gamma(\ker dt_1
    )\subset\mathfrak{X}(B_2)
\end{align}
By exactness in degree $-2$, the $C^{\infty}(M)$-module $$\ker\left(\Gamma(E_{-2})\stackrel{\dd^{(2)}}{\longrightarrow}\Gamma(E_{-1})\right)=\dd^{(3)}(\Gamma(E_{-3}))$$ is (locally) finitely generated, hence $\mathcal{F}_2:=\Gamma(\ker ds_1
    )\cap\Gamma(\ker dt_1
    )$ is a singular foliation on $B_2$. Thus, one can take a path holonomy bi-submersion $(B_3,s_2,t_2)$ over $(B_2,\mathcal{F}_2)$. By recursion on $i\geq 1$, we use a path holonomy bi-submersion $(B_{i+1},s_i,t_i)$ over $(B_i,\mathcal{F}_i)$ and  construct an anchor bundle over $\mathcal{F}_i$  by the composition \begin{equation*}
        \xymatrix{\Gamma(\varphi^* E_{-i-1})\ar@{.>>}@/_1pc/[rr]\ar@{->}[r]^<<<<{\varphi^*\dd^{({i+1})}}&\mathrm{im}(\varphi^*\dd^{(i+1)})=\ker(\varphi^*\dd^{(i)})\ar@{->>}[r]&\mathcal{F}_i}
    \end{equation*} with $\varphi=t_0\circ t_1\circ \cdots \circ t_{i-1}\colon B_{i}\longrightarrow M$ and show as for $i=0, 1$ that $\mathcal{F}_{i+1}:=\Gamma(\ker ds_i
    )\cap\Gamma(\ker dt_i
    )$ is a singular foliation on $B_{i+1}$. The proof follows.\\

\noindent
$(2)\Rightarrow (1)$ is proven by Lemma \ref{prop:sequence} and Remark \ref{rmk:pull-back-on-M} below.
\end{proof}
In the following lemma we deduce out of any bi-submersion tower over a singular foliation, a complex of vector bundles over different base manifolds, and discuss exactness. In Remark \ref{rmk:pull-back-on-M}, we give conditions to have a complex vector bundles over $M$.
\begin{lemma}\label{prop:sequence}
Let $\mathcal{F}$ be a singular foliation on $M$. Assume that there exists a bi-submersion tower $\mathcal{T}_B=(B_i,t_i,s_i, \mathcal{F}_i)_{i\geq 0}$ over $\mathcal F$. Then,
\begin{equation}\label{eq:exa-bisub}
        \xymatrix{
\cdots\ar[r]&\ker ds_2 \ar[d] \ar[r]^{ dt_2}&\ker ds_1\ar[d] \ar[r]^{ dt_1}&\ker ds_0 \ar[d] \ar[r]^{ dt_0} &TM\ar[d]\\
\cdots\ar[r]&B_3 \ar[r]_{t_2}&B_2\ar[r]_{t_1}&B_1 \ar[r]_{t_0} &M.}
    \end{equation}
 is a complex of vector bundles, which is exact at the sections level\footnote{Let us explain the notion of exactness at the level of sections when the base manifolds are not the same: what we mean is that for all $n\geq 0$, $\Gamma(\ker dt_n)\cap \Gamma(\ker ds_{n})$ is equal to the $t_{n+1}$-projectable vector fields in $\Gamma(\ker ds_{n+1})$.
 
 Equivalently, it means that the pull-back of the vector bundles in \eqref{bi-thm:eq} to any one of the manifold $B_m $ with $m \geq n $ is exact at the level of sections, i.e., \begin{equation}\label{complex-projectable}\xymatrix{\Gamma(t_{n+1,m}^* \ker ds_{n+1})\ar[r]^{ dt_{n+1}}&\Gamma(t_{n,m}^* \ker ds_n) \ar[r]^{ dt_n}&\Gamma(t_{n-1,m}^*\ker ds_{n-1}})\end{equation} is a short exact sequence of $C^\infty(B_m) $-modules, with $ t_{n,m}= t_{n} \circ \dots \circ t_{m}$ for all $m \geq n$.
 }
 if $\mathcal{T}_B$ is an exact bi-submersion tower, i.e., if $\mathcal F_i=\Gamma(\ker ds_{i-1})\cap\Gamma(\ker dt_{i-1})$ for all $i\geq 1$.
\end{lemma}
\begin{proof}
 For any element $b \in B_{i+1}$ and any vector $v\in \ker ds_i \subset T_{b} B_{i+1}$ one has,\begin{align*}
    & dt_i(v)\in T_{t_i(b)}\mathcal{F}_{i},\quad\text{(since $\Gamma(\ker ds_i)\subset t^{-1}_i(\mathcal{F}_{i})$)}.\\\Longrightarrow\quad&  dt_i(v)\in\left(\ker ds_{i-1}\cap\ker dt_{i-1}\right)|_{t_i(b)},\quad \text{{by Definition \ref{def:tower}}}.\\\Longrightarrow \quad& dt_i(v)\in \ker ds_{i-1}\quad \text{and}\quad dt_{i-1}\circ dt_i (v)=0,\; \text{for all $i\geq 1$}.
\end{align*}
This shows the sequence \eqref{eq:exa-bisub} is a well-defined complex of vector bundles.

Let us prove that it is exact when  $\mathcal F_i=\Gamma(\ker ds_{i-1})\cap\Gamma(\ker dt_{i-1})$ for all $i\geq 1$. Let $\xi\in\Gamma\left(\ker ds_{i-1}\right)$ be a $t_{i-1}$-projectable vector field that projects to zero, i.e.  $ dt_{i-1}(\xi)=0$. This implies that   $\xi\in\Gamma(\ker ds_{i-1})\cap\Gamma(\ker dt_{i-1})=\mathcal F_i$. Since $t_{i}$ is a submersion, there exists a $t_i$-projectable vector field $\zeta\in t_{i}^{-1}(\mathcal{F}_i)$ that satisfies $ dt_{i}(\zeta)=\xi$. The vector field $\zeta$ can be written as $\zeta=\zeta_1+\zeta_2$ with $\zeta_1\in\Gamma\left(\ker dt_{i}\right)$ and $\zeta_2\in\Gamma\left(\ker ds_{i}\right)$, because $t_i^{-1}(\mathcal{F}_i)=\Gamma(\ker ds_{i})+\Gamma(\ker dt_{i})$. One has, $ dt_i(\zeta_2)=\xi$. A similar argument shows that the map, $\Gamma(\ker ds_0)\overset{ dt_0}{\longrightarrow}t^*_0\mathcal{F}$, is surjective. This proves exactness in all degree.
\end{proof}

\begin{remark}\label{rmk:pull-back-on-M}One of the consequence of Lemma \ref{prop:sequence} is that: 
\begin{enumerate}
    \item If there exists a sequence of maps \begin{equation}\label{sequence:sect}
    \xymatrix{M\ar@{>->}[r]^{\varepsilon_0}&B_1\ar@{>->}[r]^{\varepsilon_1}&B_2\,\ar@{>->}[r]^{\varepsilon_2}&\cdots}
\end{equation}where for all $i\geq 0$, $\varepsilon_i$ is a section for both $s_i$ and $t_i$ then the pull-back of \eqref{eq:exa-bisub} on $M$ through the sections $(\varepsilon_i)_{i\geq 0}$ i.e., \begin{equation}\label{eq:comp-vect-bund}
    \xymatrix{\cdots\ar[r]^{dt_3}&\varepsilon_{2,0}^*\ker ds_2\ar[r]^{dt_2}&\varepsilon_{1,0}^*\ker ds_1\ar[r]^{dt_1}&\varepsilon_0^*\ker ds_0\ar[r]^{dt_0}&TM}
\end{equation}
is a complex of vector bundles, with the convention $\varepsilon_{n,0}= \varepsilon_{n}\circ\cdots \circ\varepsilon_{0}$. If $\mathcal{T}_B$ is an exact bi-submersion tower then, \eqref{eq:comp-vect-bund} is a geometric resolution of $\mathcal{F}$.
\item In case that $\mathcal{T}_B$ is an exact path holonomy bi-submersion tower, such a sequence \eqref{sequence:sect} always exists, since the bi-submersions $(B_{i+1}, s_i, t_i)$ are as in Example \ref{ex:holonomy-biss}. For such bi-submersions, the zero section $x\mapsto ( x, 0)$ is a section for both $s_i$ and $t_i$.
\end{enumerate}
\end{remark}

Theorem \ref{thm:equivalence} is now proven. Here is a consequence.
\begin{corollary}\label{cor:exa}
Let $\mathcal{F}$ be a singular foliation on $M$. Assume that there exists an exact  bi-submersion tower $\mathcal{T}_B=(B_i,t_i,s_i, \mathcal{F}_i)_{i\geq 0}$ over $\mathcal F$ of length $n+1$. 
Then, the pull-back of the sequence of vector bundles
  \begin{equation}\label{bi-thm:eq}
      \xymatrix{\ker ds_{n}\ar[rrrrd] \ar[r]^-{ dt_{n}}&t^*_{n}\ker ds_{n-1}\ar[r]\ar[rrrd]&{\;\cdots}\ar[rrd] \ar[r]^-{ dt_2}&t_{2,n}^*\ker ds_1\ar[rd] \ar[r]^-{ dt_1}&TB_{n+1}\times_{TM}\ker ds_0 \ar[d] \ar[r]^-{\mathrm{pr}_1} &TB_{n+1}\ar[ld]\\& & &&B_{n+1}
  \footnote{ $TB_{n+1}\times_{TM}\ker ds_0=\{(u,v)\in TB_{n+1}\times\ker ds_0\mid  dt_{0,n}(u)= dt_0(v)\}$.}
  }
 \end{equation}
    is a geometric resolution of the pull-back foliation $t_{0,n}^{-1}(\mathcal{F})\subset \mathfrak{X}(B_{n+1})$, where  $\mathrm{pr}_1$ is the projection on $TB_{n+1}$ and for $i\geq 1$, $t_{i,j}$ is the composition $t_i\circ\cdots\circ t_j\colon B_{j+1}\to B_{i}$.
\end{corollary}
\begin{proof}
By Lemma \ref{prop:sequence}, the complex in Equation \eqref{bi-thm:eq} is exact. By construction, the projection of the fiber product $TB_{n+1}\times_{TM}\ker ds_0$ to $TB_{n+1}$ induces the singular foliation $t_{0,n}^{-1}(\mathcal{F})$. 
\end{proof}

\subsection{Lift of a symmetry to the bi-submersion tower}
Let us investigate what an action   $\varrho\colon\mathfrak g\rightarrow \mathfrak{X}(M)$ of a Lie algebra $\mathfrak{g}$ on a singular foliation $(M,\mathcal F)$ would induce on a bi-submersion tower $\mathcal{T}_B$ over $\mathcal{F}$.\\

We start with some vocabulary and preliminary results.

\begin{definition}Let $(B,s,t)$ be a bi-submersion over a singular foliation $\mathcal F$ on a manifold $M$. We call \emph{lift} of a vector field $X\in\mathfrak X(M)$ to the bi-submersion $(B,s,t)$ a vector field $\widetilde{X}\in \mathfrak{X}(B)$ which is both $s$-projectable on $X$ and $t$-projectable on $X$. 
\end{definition}

The coming proposition means that the notion of lift to a bi-submersion only makes sense for symmetries of the singular foliation.

\begin{prop}\label{lift-symmetry-bi-susbmersion} If a vector field on $M$ admits a lift to a bi-submersion $(B,s,t)$ over a singular foliation $\mathcal F$, then it is a symmetry of $\mathcal F$.
\end{prop}
\begin{proof}
Let $\widetilde{X}\in\mathfrak{X}(B)$ be a lift of $X\in\mathfrak{X}(M)$. Since $\widetilde{X}$  is $s$-projectable, $[\widetilde{X},\Gamma(\ker \dd s)]\subset\Gamma(\ker \dd s)$. Since $\widetilde{X}$  is $t$-projectable, $[\widetilde{X},\Gamma(\ker \dd t)]\subset\Gamma(\ker \dd t)$. Hence:
\begin{align*} [\widetilde{X},s^{-1}(\mathcal{F})] &= [\widetilde{X},\Gamma(\ker \dd s)+\Gamma(\ker \dd t)] \\ &= [\widetilde{X},\Gamma(\ker (\dd s)]+[\widetilde{X},\Gamma(\ker \dd t)]  \\ &\subset  \Gamma(\ker \dd s)+\Gamma(\ker \dd t) =s^{-1}(\mathcal{F}) .\end{align*}
In words, $\widetilde{X}$ is a symmetry of $s^{-1} \mathcal F$.
Since $\widetilde{X}$ projects through $s$ to $X$, $X$ is a symmetry of ~$\mathcal F $.
\end{proof}

We investigate the existence of lifts of symmetries of $ \mathcal F$ to bi-submersions over $\mathcal F $.

\begin{remark}\label{rmk:uniqueness} For $X, Y\in\mathfrak{s}(\mathcal F)$, 
\begin{enumerate}
    \item the lift $\widetilde{X}$ to a given bi-submersion is not unique, even when it exists. However, two different lifts of a $X\in\mathfrak s(\mathcal{F})$ to a bi-submersion $(B,s,t)$ differ by an element of the intersection $\Gamma(\ker (d s))\cap\Gamma(\ker (d t))$. 
    \item the lift $\widetilde X$ is a symmetry of $\Gamma(\ker (d s))\cap\Gamma(\ker (d t))$, i.e., $[\widetilde{X}, \Gamma(\ker (d s))\cap\Gamma(\ker (d t))]\subset \Gamma(\ker (d s))\cap\Gamma(\ker (d t))$,  since $\widetilde{X}$ is $s$-projectable and $t$-projectable.
    \item If the lifts  $\widetilde{X}$ and $\widetilde{Y}$ exist, then $\widetilde{[X,Y]}$ exists and   \begin{equation}
   \widetilde{[X,Y]}-[\widetilde{X},\widetilde{Y}]\in \Gamma(\ker\dd s)\cap\Gamma(\ker\dd t).
\end{equation}
\end{enumerate}
\end{remark}

As the following example shows, the lift of a symmetry  to a bi-submersion may not exist.

\begin{example}
Consider the trivial foliation $\mathcal{F}:=\{0\}$ on $M$. For any diffeomorphism $\phi\colon M\longrightarrow M$, $(M,\text{id}, \phi)$ is a bi-submersion over $\mathcal{F}$. Every vector field $X\in\mathfrak{X}(M)$ is a symmetry of $\mathcal{F}$. If it exists, its lift has to be given by $\widetilde{X}=X $ since the source map is the identity. But $\widetilde{X}=X $ is $t$-projectable if and only if
$X$ is $\phi$-invariant. A non $\phi$-invariant vector field $X$ therefore admits no lift to $(M,\text{id}, \phi)$.
\end{example}

However, internal symmetries, i.e., elements in $\mathcal F$ admit lifts to any bi-submersion.

\begin{prop}\label{prop:lift-internal}
Let $(B,s,t)$ be a bi-submersion over a singular foliation $\mathcal F$ on a manifold $M$. 
Every internal symmetry, i.e., every vector field in $X\in\mathcal{F}$, admits a lift $\widetilde{X}\in \mathfrak{X}(B)$ to $(B,s,t)$. {Moreover, $\widetilde{X}$ can be chosen to be of the form $$\widetilde{X}:=X_s^t+X_t^s$$ with $X^t_s\in \Gamma(\ker (\dd s))$ and $X^s_t\in\Gamma(\ker (\dd t))$}. 
\end{prop}
\begin{proof}
Let $X\in\mathcal{F}$. Since $s\colon B\longrightarrow M $ is a submersion, there exists $X^s\in\mathfrak{X}(B)$ $s$-projectable on $X$. 
Since $t $ is a submersion, there exists $X^t\in\mathfrak{X}(B)$ $t$-projectable on $X$.
By construction $X^s \in s^{-1}(\mathcal F) $
 and  $X^t \in t^{-1}(\mathcal F) $.
 Using the property \eqref{eq:defbisub} of the bi-submersion $(B,s,t)$, the vector fields  $X^s$ and $X^t$ decompose as $$\begin{cases}
X^s=X^s_s+X^s_t\quad \text{with\; $X^s_s\in \Gamma(\ker (\dd s))$,\;$X^s_t\in\Gamma(\ker (\dd t))$},\\X^t=X^t_s+X^t_t\quad \text{with\; $X^t_s\in \Gamma(\ker (\dd s))$,\;$X^t_t\in\Gamma(\ker (\dd t))$}.
\end{cases}$$
By construction, $ X^s_t  $ is $s$-projectable to $X$ and $t$-projectable to $0$ while $X^t_s  $ is $s$-projectable to $0$ and $t$-projectable to $X$.
It follows that, $\widetilde{X}:=X_s^t+X_t^s$, is a lift of $X$ to the bi-submersion $(B,s,t)$. 
\end{proof}


{The proof we give for Theorem \ref{thm:equivalence} uses the notion of left-invariant, resp. right-invariant, vector fields on a bi-submersion over a singular foliation. We define the latter in the next proposition. It uses the notion of anchored bundle over a singular foliation and almost Lie algebroid, see \cite{LLS, LLL} for more details. } 

\begin{proposition}\label{prop:left-arrows}
Let $(B,s,t)$ be a bi-submersion over a singular foliation $\mathcal F$ on a manifold $M$. Let $(A, \rho)$ be an anchored bundle over $\mathcal{F}$, i.e., $A\longrightarrow M$ is a vector bundle and $\rho\colon A\longrightarrow TM$ is a vector bundle morphism such that $\rho(\Gamma(A))=\mathcal F$. There exists two maps
\begin{equation}
\label{def:gauchedroite}
     \begin{array}{rcl} \Gamma(A) & \longrightarrow & \mathfrak X(B) \\ a & \longmapsto & \overleftarrow{a}\\  a & \longmapsto & \overrightarrow{a}\end{array}
\end{equation}
fulfilling the following conditions:
\begin{enumerate}
\item the vector field $\overrightarrow{a}\in \mathfrak{X}(B)$
(resp. $\overleftarrow{a}\in \mathfrak{X}(B)$) is $t$-related (resp. $s$-related) with $\rho(a) \in~\mathcal{F}
$,
    \item the vector field $\overrightarrow{a}$
(resp. $\overleftarrow {a}$) is tangent to the fibers of $s$ (resp. $t$),
\item $\overrightarrow{fa}=t^*(f)\overrightarrow{a}$ and $\overleftarrow{fa}=s^*(f)\overleftarrow{a}$ for all $a\in \Gamma(A),\,f\in C^\infty(M)$.
\end{enumerate}

{The vector fields $\overleftarrow{a}$ (resp. $\overrightarrow{a}$) for $a\in\Gamma(A)$ are called \emph{left-invariant} (resp. \emph{right-invariant}) vector fields of $(B,s,t)$.}
\end{proposition}
\begin{proof}
By Proposition \ref{prop:lift-internal}, given a section $a\in\Gamma(A)$ the vector field $\rho(a)\in\mathcal{F}$ admits a lift $\widetilde {\rho(a)}\in \mathfrak{X}(B)$ on $(B,s,t)$ of the form $$\widetilde{\rho(a)}:=\rho(a)_s^t+\rho(a)_t^s$$ with $\rho(a)^t_s\in \Gamma(\ker (\dd s))$ and $\rho(a)^s_t\in\Gamma(\ker (\dd t))$. Also, $dt(\rho(a)_s^t)=\rho(a)$ and  $ds(\rho(a)_t^s)=\rho(a)$. Let $b\in B$ and $\mathcal{U}_b$ an open neighborhood of $b$. Let $(a_1,\ldots,a_r)$ be a local trivialization of $A$ on the open subset $\mathcal U=s(\mathcal{U}_b)\subset M$.  We define a map $R_U$ on local generators by 
\begin{align}\label{R-inv:equation}
    R_U\colon\Gamma_{\mathcal{U}_b}(t^*A)&\longrightarrow \Gamma_{\mathcal{U}_b}(\ker (\dd s))\\\nonumber t^*a_i\;&\longmapsto \rho(a_i)^t_s
\end{align}
and extend by $C^\infty(\mathcal U_b)$-linearity. 
These maps can be glued using partitions of unity. More precisely, let $(\chi_\lambda)_{\lambda\in \Lambda}$ be a partition of unity subordinate
to an open cover $(\mathcal{U}_\lambda)_{\lambda\in \Lambda}$ by open sets that trivialize the vector bundle $A$. We define a map $R$ on $\Gamma(t^*A)$ as $$\displaystyle{\sum_{\lambda\in \Lambda}\chi_\lambda R_{\mathcal{U}_\lambda}} .$$ 
Now for  $a\in \Gamma(a)$ we define $\overrightarrow{a}:= R(s^*a)$. The map $\overleftarrow{\bullet}$ is defined similarly. Item $1$, $2$ and $3$ hold by construction.

Assume that  $(A,\rho)$ is equipped with an almost Lie algebroid bracket $\lb_A$. For all $a,b\in \Gamma(A)$ one has
\begin{align*}
    ds\left(\overleftarrow{[a,b]_{A}}- [\overleftarrow{a},\overleftarrow{b}]\right)&=\rho([a, b]_{A})-[\rho(a),\rho(b)]\\&=0,\end{align*}because $\overleftarrow{a}$ is $s$-projectable to $\rho(a)$ and $\rho$ is a morphism of brackets. Since left-invariant vector fields are tangent to the fibers of $t$, one has $dt\left(\overleftarrow{[a,b]_{A}}- [\overleftarrow{a},\overleftarrow{b}]\right)=~0$. The proof is similar for $\overrightarrow{[a,b]_{A}}- [\overrightarrow{a},\overrightarrow{b}]$. This ends the proof. 
\end{proof}
The following lemma is important in the proof of Theorem \ref{thm:equivalence}.
\begin{lemma}\label{lemma:invariant-vector-fields}
Let $(B,s,t)$ be any bi-submersion over a singular foliation $\mathcal F$ on a manifold $M$, and $(A,\rho)$ an anchored bundle over $\mathcal{F}$. 
\begin{enumerate}
   \item  {There exists vector bundle morphisms $R\colon t^*A \longrightarrow \ker ds$ 
 and $L\colon s^*A \longrightarrow \ker dt$ inducing \eqref{def:gauchedroite}}.
 
\item  {Let $x\in M$.  If $(B,s,t)$ is a path holonomy bi-submersion over $\mathcal{F}$ near $(x, 0)$  then, every $b\in B$ such that $t(b)=x$  admits a neighborhood $V$ such that every $t$-projectable vector field of $\Gamma_V(\ker ds)$ 
is of the form $R(\xi)$ 
for some $\xi\in \Gamma_{V}(t^*A)$.}
\end{enumerate}
\end{lemma}
\begin{remark}
In item 1 of Lemma \ref{lemma:invariant-vector-fields}, by \textquotedblleft\,inducing \eqref{def:gauchedroite}\textquotedblright  we mean that  for every $a\in \Gamma(A)$, $\overrightarrow{a}:=R(t^*a)$ and $\overleftarrow{a}:=L(s^*a)$.
\end{remark}
\begin{proof}
{Item $1$ is obtained in the proof of  Proposition \ref{prop:left-arrows}. For instance, the $C^\infty(B)$-linear map $R\colon \Gamma(t^*A) \longrightarrow \Gamma(\ker ds)$ defined in \eqref{R-inv:equation} corresponds to a morphism of vector bundles  $t^*A \longrightarrow \ker ds$}. Let us prove item ${2}$. By assumption, $B$ is a neighborhood of $(0,x)$ in $M\times \mathbb{R}^n$ with $n=\mathrm{rk}_x(\mathcal{F})=\dim (\mathcal{F}_x:=\mathcal{F}/\mathcal{I}_x\mathcal{F})$  near $x\in M$ (see Example \ref{ex:holonomy-biss}). Let $b\in B$ and $\mathcal{U}_b$ an open neighborhood of $b$. Let $(a_1,\ldots,a_r)$ be a local trivialization of $A$ on the open subset $\mathcal U=t(\mathcal{U}_b)\subset M$.  One has by definition of right-invariant vector fields of $(B,s, t)$ that $dt(\overrightarrow{a_i})=t^*\rho(a_i)$ for $i=1,\ldots, \mathrm{rk}(A)$. The vector fields $\rho(a_i)$ are generators of $\mathcal{F}$ on $\mathcal{U}$. We necessarily have $n\leq\mathrm{rk}(A)$. Since the $\rho(a_i)(x)$'s are generators of $\mathcal{F}_x$, without loss of generality we can assume that $\rho(a_1)(x), \ldots, \rho(a_n)(x)$ is  a basis of $\mathcal{F}_x$. Since $\mathrm{rk}(\ker ds)=n$,  $\left(\overrightarrow{a_i}(b)\right)_{i=1,\ldots, n}$ form a basis of $\ker ds|_x$. Therefore, the $\overrightarrow{a_i}$’s are independent at every point of some neighborhood $V\subset\mathcal{U}_b$ of $b$ i.e., they form a local trivialization of the  vector bundle $\ker ds\longrightarrow B$. As a result, vector fields of $\Gamma_V(\ker ds)$ are of the form $\sum _i^nf_i\overrightarrow{a_i}$ with $f_i\in C^\infty(V)$ for $i=1, \ldots,n$. This ends the proof.
\end{proof}

We can now state one of the important results of this section, {which is the converse of Proposition \ref{lift-symmetry-bi-susbmersion}}. It uses several concepts introduced in \cite{AndroulidakisIakovos}, which are recalled in the proof. 

\begin{prop}\label{prop:lift}
Let $\mathcal{F}$ be a singular foliation on a manifold $M$. Every symmetry $X\in \mathfrak{s}(\mathcal F)$ admits a lift
 \begin{enumerate}
    \item to any path holonomy bi-submersion  $(B,s,t)$,
    \item to Androulidakis-Skandalis' path holonomy atlas,
    \item to a neighborhood of any point in a bi-submersion through which there exists a local bisection that induces the identity.
\end{enumerate}
\end{prop}
\begin{proof}[Proof (of Proposition \ref{prop:lift})]
Let $X\in\mathfrak{s}(\mathcal{F})$. Assume that $(B,s,t)=(\mathcal{W}, s_0,t_0)$ is a path holonomy bi-submersion associated to some generators $X_1,\ldots,X_n\in\mathcal{F}$ as in Example \ref{ex:holonomy-biss}. Fix $(u,y=(y_1,\ldots,y_n))\in \mathcal{W}\subset M\times\mathbb{R}^n$, set $Y:=\sum_{i=1}^dy_iX_i$. {By assumption, $[Y, X]\in \mathcal{F}$. This implies that $d\varphi_1^Y(X)=(\varphi^Y_1)_*(X)\in X+\mathcal{F}$. Indeed, for $t$ such that  the flow $\varphi_t^Y$  of $Y$ is defined, one has  \begin{align*}
     d\varphi_1^Y(X)&=d\varphi_0^Y(X)+\int_0^1\frac{d}{dt}(d\varphi_t^Y)dt\\&= X+\int_0^1\underbrace{d\varphi_t^Y([Y,X])}_{\in \mathcal{F}}dt,\qquad\text{since $d\varphi_t^Y(\mathcal{F})=\mathcal{F}$}.
\end{align*}
Let $ Z_y=\int_0^1 d\varphi_t^Y([Y,X])dt$. 
When $\mathcal F $ is closed for Fréchet topology, it is clear that $ Z_y$ belongs to $\mathcal F $. We claim that it is in fact always true: Upon restricting to an open subset of $M$ if necessary, Item $2$ in Remark \ref{rmk:flot} implies that one can find local generators $ X_1, \dots, X_r$ of $ \mathcal F$, such that
 $$  d\varphi_t^Y([Y,X]) = \sum_{i=1}^r F_t^i X_i  $$ for some smooth functions $ F_t^i$ depending smoothly on the variable $t$. By integration, $Z_y = \sum_{i=1}^r \int_0^1 F_t^i dt \, X_i $ belongs to $ \mathcal F $.} In conclusion,  there exists $Z_y\in\mathcal{F}$ depending smoothly on $y$ such that $dt_0(X,0)=X+Z_y$. Take $\widetilde{Z}_y\in t_0^{-1}(\mathcal{F})$ such that $dt_0(\widetilde{Z}_y)=Z_y$. One has, 
\begin{equation*}
    dt_0\left((X,0)-\widetilde{Z}_y\right)=X.
\end{equation*}
Also, we can write $\widetilde{Z}_y=\widetilde{Z}_y^1+\widetilde{Z}_y^2$, with $\widetilde{Z}_y^1\in\Gamma(\ker  ds_0)$, $\widetilde{Z}_y^2\in\Gamma(\ker dt_0)$, since $\widetilde{Z}_y\in t_0^{-1}(\mathcal{F})$. By construction, the vector field $\widetilde{X}:=(X,-\widetilde{Z}_y^1)$ is a lift of $X$ to the bi-submersion $(\mathcal{W},s_0,t_0)$. This proves item ${1}$.\\

If $X_B\in\mathfrak{X}(B)$ and  $X_{B'}\in\mathfrak{X}(B')$ are two lifts of the symmetry $X$ on the path holonomy bi-submersions  $(B, s,t )$ and $(B',s',t')$ respectively, then $(X_B,X_{B'})$ is a lift of $X$ on the composition bi-submersion $B\circ B'$. This proves item ${2}$, since the Androulidakis-Skandalis' path holonomy atlas is made of fibered products and inverse of holonomy path holonomy bi-submersions \cite{AndroulidakisIakovos}.

Item $2$ in Proposition 2.10 of \cite{AndroulidakisIakovos} states that if the identity of $M$ is carried by $(B,s,t)$ at some point $v\in B$, then there exists an open neighborhood $V\subset B$ of $v$ that satisfies $s_{|_V}=s_0\circ g$ and $t_{|_V}=t_0\circ g$, for some submersion $g\colon V\longrightarrow \mathcal{W}$, for $\mathcal{W}$ of the form as in item 1. Thus, for all $X\in\mathfrak{s}(\mathcal{F})$ there exists a vector field $\widetilde{X}\in\mathfrak{X}(V)$ fulfilling $ ds_{|_V}(\widetilde{X})=dt_{|_V}(\widetilde{X})=X$. This proves item $3$.
\end{proof}

\begin{definition}
A \emph{symmetry of the tower of bi-submersion} $\mathcal{T}_B=(B_{i+1},s_i,t_i,\mathcal F_i)_{i \geq 0} $ over a singular foliation $\mathcal{F}_0=\mathcal{F}$ is a family $(X^i)_{i \geq 0} $ of vector fields with  the \emph{$i$-th component} $X^i $ in $\mathfrak{s}(\mathcal{F}_i)$  such that $ ds_{i-1}(X^i)= dt_{i-1}(X^i)=X^{i-1}$ for all $i\geq 1$. We denote by $\mathfrak{s}(\mathcal{T}_B)$ the Lie algebra of symmetries of $\mathcal{T}_B$.
\end{definition}

The next theorem gives a class of bi-submersion tower to which any symmetry of the base singular foliation $\mathcal F $ lifts.

\begin{theorem}\label{thm:symm-tower-path} Let $\mathcal F $  be a  foliation. Let $\mathcal{T}_B$ be an exact
path holonomy bi-submersion tower (or an exact path holonomy atlas bi-submersion tower). 
A vector field $X\in\mathfrak{X}(M)$ is a symmetry of $\mathcal{F}$, i.e. $[X,\mathcal{F}]\subset\mathcal{F}$, if and only if it is the component on $M$ of a symmetry of $\mathcal{T}_B$.
\end{theorem}
\begin{proof}
For any symmetry $(X^i)_{i \geq 0} $ of the bi-submersion tower $\mathcal{T}_B$ the first component $X^0\in \mathfrak X(M)$ is a symmetry of $\mathcal F$,  by Proposition \ref{lift-symmetry-bi-susbmersion}. The other implication is a consequence of item 1. resp. item 2. in Proposition \ref{prop:lift} and Remark \ref{rmk:uniqueness}. It is due to the fact that the tower $\mathcal{T}_B$ is generated by path  holonomy bi-submersions, and then we can lift  symmetries at every stage of the tower $\mathcal{T}_B$. Indeed, assume for instance that $\mathcal{T}_B$ is an exact
path holonomy bi-submersion tower. By Proposition \ref{prop:lift},  any symmetry $X\in \mathfrak{s}(\mathcal{F})$ admits a lift $X^1\in \mathfrak{X}{(B_1)}$. Moreover, $X^1$ is a symmetry of the singular foliation $\mathcal{F}_1=\Gamma(\ker (d s_0))\cap\Gamma(\ker (d t_0))$, by Remark \ref{rmk:uniqueness}(3). We continue by recursion to construct $X^i\in \mathfrak{s}(\mathcal{F}_i)$ for $i\geq 2$.
\end{proof}

\begin{remark}\label{rmk:sym-chain-map}
Let $(X^i)_{i\geq 0}$ be a lift of $X^0:=X\in\mathfrak{s}(\mathcal{F})$. For $i\geq 1$, $\nabla^i_X:=\mathrm{ad}_{X^i}$ preserves $\Gamma(\ker  ds_{i-1})$, since $X^i$ is $s_{i-1}$-projectable. Altogether, they define a chain map $(\nabla_X^i)_{i\geq 0}$ at the section level of the complex \eqref{eq:exa-bisub}, on projectable vector fields in  \eqref{complex-projectable}, since for every $i\geq 0$ and any $t_{i}$-projectable vector field $\xi\in \ker  ds_{i}$, \begin{align*}
     dt_{i}([X^{i+1},\xi])&=[ dt_{i}(X^{i+1}),  dt_{i}(\xi)]\\&=[X^{i},  dt_{i}(\xi)], 
\end{align*}that is $ dt_{i}\circ \nabla^{i+1}_X=\nabla^i_X\circ  dt_{i}$. 
\end{remark}

\begin{remark}
In \cite{GarmendiaAlfonso2}, under some assumptions, it is shown that if a Lie group  $G$ acts on a foliated manifold $(M,\mathcal{F})$, then it acts on its holonomy groupoid. It is likely that this result follows from Theorem \ref{thm:symm-tower-path}, this will be addressed in another study.
\end{remark}
\subsection{Lifts of actions of a Lie algebra on a bi-submersion tower}
We end the section with the following constructions and some natural questions.\\

Let $\mathcal{T}_B=(B_{i+1},s_i,t_i,\mathcal F_i)_{i \geq 0} $ be an exact path holonomy  bi-submersion tower over a singular foliation $(M,\mathcal F)$ of length $n+1$. 

By Theorem \ref{thm:symm-tower-path}, any vector field $X\in \mathfrak{s}(\mathcal{F})$ lifts to a symmetry $(X^i)_{i\geq 0}$ of $\mathcal{T}_B$. Once a lift is chosen, we can define a linear map, 
$$X\in \mathfrak{s}(\mathcal{F})\mapsto (X^i)_{i\geq 1}\in \mathfrak{s}(\mathcal{T}_B).$$
Let $\varrho\colon\mathfrak g\rightarrow \mathfrak{X}(M)$ be a strict symmetry action of a Lie algebra $\mathfrak{g}$ on $(M,\mathcal F)$.  For $x\in\mathfrak g$, there exists   $(\varrho(x)^i)_{i \geq 0}$, with  $\varrho(x)^i\in\mathfrak{s}(\mathcal{F}_i)\subset \mathfrak{X}(B_i)$ a symmetry of $\mathcal{T}_B$ such that $X_0=\varrho(x)\in \mathfrak s(\mathcal F)$, by Theorem \ref{thm:symm-tower-path}. Consider the composition, \begin{equation}
x\in\mathfrak g\longmapsto \varrho(x)\in \mathfrak{s}(\mathcal F)\longmapsto (\varrho(x)^i)_{i \geq 0}\in \mathfrak{s}(\mathcal{T}_B)\mapsto \varrho(x)^1\in \mathfrak{X}(B_1).
\end{equation}

\begin{lemma}\label{lemma:sym-action-tower}
For all $x,y\in \mathfrak{g}$, $$[\varrho(x),\varrho(y)]^1-[\varrho(x)^1, \varrho(y)^1]= dt_1(C_1(x,y))$$
with $C_1(x,y)\in \Gamma(\ker  ds_1\rightarrow B_2)$ a $t_1$-projectable vector field, for some bilinear map  $$C_1\colon \wedge^2\mathfrak g \longrightarrow \Gamma(\ker  ds_1\rightarrow B_2).$$
\end{lemma}
\begin{proof}
This follows from Lemma \ref{prop:sequence}, because $[\varrho(x),\varrho(y)]^1-[\varrho(x)^1, \varrho(y)^1]\in \Gamma(\ker ds_{0})\cap\Gamma(\ker dt_{0})$. 
\end{proof}

\begin{theorem}\label{thm:final-sym}
The map $C_1\colon \wedge^2\mathfrak g \longrightarrow \Gamma(\ker  ds_1\rightarrow B_2)$ of Lemma \ref{lemma:sym-action-tower} satisfies for all $x,y, z\in \mathfrak{g}$,
\begin{equation}
    C_1([x,y]_\mathfrak{g},z) + \nabla_{\varrho(x)}^2(C_1(y,z)) + \circlearrowleft(x,y,z)= dt_2(C_2(x,y,z)) 
\end{equation}
for some tri-linear map $C_2\colon\wedge^3\mathfrak g \longrightarrow \Gamma(\ker  ds_2\rightarrow B_3)$. Here, $\nabla^2$ is, as in Remark \ref{rmk:sym-chain-map}.
\end{theorem}
\begin{proof}
For $x,y, z\in \mathfrak{g}$, 
\begin{align*}
    dt_1\left(C_1([x,y]_\mathfrak g,z)\right) +\circlearrowleft(x,y,z)&=[\varrho([x,y]_\mathfrak g),\varrho(z)]^1-[\varrho([x,y]_\mathfrak g)^1,\varrho(z)^1] +\circlearrowleft(x,y,z)\\&=\cancel{[[\varrho(x),\varrho(y)],\varrho(z)]^1}-[[\varrho(x),\varrho(y)]^1,\varrho(z)^1]+\circlearrowleft(x,y,z)\\&=-\cancel{[[\varrho(x)^1,\varrho(y)^1],\varrho(z)^1]}+[dt_1(C_1(x,y)),\varrho(z)^1]+\circlearrowleft(x,y,z)\\&=dt_1([C_1(x,y),\varrho(z)^2]) +\circlearrowleft(x,y,z).
\end{align*}
We have used Jacobi identity and $dt_1(\varrho(z)^2)=\varrho(z)^1$. This implies that 
\begin{equation}
    dt_1\left(C_1([x,y]_\mathfrak g,z)-[C_1(x,y),\varrho(z)^2] +\circlearrowleft(x,y,z)\right)=0.
\end{equation}Again Lemma \ref{prop:sequence} implies the result.\end{proof}

\hspace{1cm}

\noindent
Here is a natural  question:

\noindent
\textbf{Question}: {Can we construct a Lie $\infty$-algebra structure on $\displaystyle{\oplus_{i=0}^{+\infty}{}^{t_i}\Gamma(\ker ds_{i})}$ such that the construction in Theorem \ref{thm:final-sym} continues to a Lie $\infty$-morphism from $\mathfrak{g}[1]$ to $\displaystyle{\oplus_{i=0}^{+\infty}{}^{t_i}\Gamma(\ker ds_{i})}$?}  {where ${}^{t_i}\Gamma(\ker ds_{i})$ is defined as in Convention \ref{conven:proj}}. 


\appendix\label{app:A}
\section{Universal Lie $\infty$-algebroids}\label{sec:2}
Let us now recall the definition of Lie $\infty$-algebroids over a manifold and their morphisms and homotopies. Most definitions of this section can be found in \cite{Poncin,LLS,CLRL} and our conventions and notations are those of \cite{LLS,CLRL}.

{In the definition below, we restrict ourselves to the case of finite rank. Recall that finitely generated projective modules, by Serre-Swan theorem \cite{SwanRichardG}, are the module of sections of vector bundles. This is the setting in this article, except for Theorem \ref{alt-thm-res} where infinite rank Lie algebroid are allowed see e.g., \cite{CLRL}.}
\begin{definition}
A \emph{Lie $\infty$-algebroid over $M$} is the datum of a sequence $E = (E_{-i} ),\,1\leq i<\infty$ of vector bundles over $M$ together with a structure of Lie $\infty$-algebra $(\ell_k)_{k\geq 1}$ on the sheaf of sections of $E$ and a vector bundle morphism, $\rho\colon E_{-1}\rightarrow TM$, called \emph{anchor map} such that the $k$-ary brackets $\ell_k,\,k\neq 2$ are $\mathcal O$-multilinear and such that
\begin{equation}
  \ell_2(e_1, f e_2) =\rho(e_1)[f]e_2 + f\ell_2(e_1, e_2) 
\end{equation}
for all $e_1\in \Gamma(E_{-1} ), e_2\in\Gamma(E_\bullet)$ and $f\in\mathcal O$.\\

The sequence

 \begin{equation}
        \xymatrix{\cdots\ar[r]^{\ell_1}&E_{-2}\ar[r]^{\ell_1}&E_{-1}\ar[r]^{{\rho}}&{TM},}
    \end{equation}
    
    is a complex called the \emph{linear part} of the Lie $\infty$-algebroid.
\end{definition}

\begin{remark}
Any Lie $\infty$-algebroid on $M$ has an induced singular foliation on $M$ which is given by the image of the anchor map, that we call the \emph{basic singular foliation}.
\end{remark}

There is an alternative definition for Lie $\infty$-algebroids in terms of $Q$-manifolds with purely non-negative degrees. 
\begin{definition}\label{defv2}
A \emph{splitted $NQ$-manifold} is a pair $(E,Q)$ where $E\rightarrow M$ is a sequence of vector bundles over $M$ indexed by negative integers and where $Q$ is a homological vector field of degree $+1$, i.e., $Q\in\text{Der}_1\left( \Gamma(S^\bullet(E^*))\right)$ is such that $[Q,Q]=0$.\\

\noindent
We denote by $\E$ and call \emph{functions} on the splitted $NQ$-manifold $E\longrightarrow M$ the sheaf of graded commutative $\mathcal{O}$-algebras made of
sections of  $S^\bullet(E^*)$.
\end{definition}

There is a one-to-one correspondence between splitted $NQ$-manifolds and Lie $\infty$-algebroids \cite{Poncin, LeanMadeleineJotz2020Maru,Voronov2}. This formulation allows to write in a compact manner morphisms of Lie $\infty$-algebroids i.e., simply as chain maps. Homotopy equivalence can also be defined, see Section 3.4.2  in \cite{LLS} or  \cite{CLRL} for more details. From now on, we write $(E,Q)$ to denote a Lie $\infty$-algebroid over $M$.\\


Let us recall from \cite{LLS,CLRL} the following definition and theorem. 

\begin{definition}\label{def:geom-resol}
Let $\mathcal{F}\subset \mathfrak{X}(M)$ be a singular foliation on a manifold $M$. A \emph{geometric resolution} of the singular foliation $\mathcal{F}$ is a projective resolution $((P_{-i})_{i\geq 1},(\dd^{(i)})_{i\geq 2},\rho)$ of $\mathcal{F}$ as a $\mathcal{O}$-module that corresponds to a sequence of vector bundles $(E,\Bar{\dd},\Bar{\rho})$ over $M$ 
 \begin{equation}
        \xymatrix{\cdots\ar[r]^{\Bar{\dd}^{(3)}}&E_{-2}\ar[r]^{\Bar{\dd}^{(2)}}&E_{-1}\ar[r]^{\Bar{\rho}}&{TM},}
    \end{equation}
i.e.,
\begin{itemize}
    \item for $i\geq 1$ the $\mathcal{O}$-module of sections of $E_{-i}$ is $P_{-i}=\Gamma(E_{-i})$
    \item for $i\geq 2$, the induced maps at the sections level $$\Bar{\dd}^{(i)}\colon \Gamma(E_{-i})\longrightarrow \Gamma(E_{-i+1})\quad\text{or}\quad \Bar{\rho}\colon \Gamma(E_{-1})\longrightarrow \mathcal{F}$$ coincide with $\dd^{(i)}\colon P_{-i}\longrightarrow P_{-i+1}$ or with $\rho\colon P_{-1}\longrightarrow \mathcal{F}$ respectively.
\end{itemize}
For convenience, we denote by $\Bar{\dd}$ and $\Bar{\rho}$ the same as $\dd$ and $\rho$ respectively. Also, we call $\rho\colon E_{-1}\longrightarrow TM$ the \emph{geometric resolution anchor}. A geometric resolution is said to be \emph{minimal} at a point $m\in M$ if, for all $i\geq 2$, the linear maps $\dd^{(i)}_{|_m}\colon {E_{-i}}_{|_m}\longrightarrow  {E _{-i+1}}_{|_m}$ vanish. 
\end{definition}

\begin{theorem}\cite{LLS, CLRL,LavauSylvain}
Let $\mathcal{F}$ be a singular foliation on $M$.
Any geometric resolution  of ~$\mathcal F$ 
\begin{equation}
    \label{eq:resolutions}
\cdots \stackrel{\dd} \longrightarrow E_{-3} \stackrel{\dd}{\longrightarrow} E_{-2} \stackrel{\dd}{\longrightarrow} E_{-1} \stackrel{\rho}{\longrightarrow} TM \end{equation}
  comes equipped with a Lie $\infty $-algebroid structure whose unary bracket is $\dd $ and whose anchor map is $\rho$. Such a Lie $\infty $-algebroid structure is unique up to homotopy and is called a \emph{universal Lie $\infty $-algebroid of} $\mathcal F$.
  
\end{theorem}
\begin{remark}
For a given Lie $\infty$-algebroid $(E,Q)$, the triple $(\mathfrak{X}_\bullet(E), \lb, \text{ad}_Q)$ is a differential graded Lie algebra, where $\mathfrak{X}_\bullet(E)$ stands for the module of graded vector fields (=graded derivations of $\E$) on $E$, the bracket $\lb$ is the graded commutator of derivations and $\text{ad}_Q:=[Q,\cdot\,]$. 
\end{remark}
\section{Lie $\infty$-morphisms of DGLA and homotopies}\label{appendix-homotopies}
Let us recall the definitions of Lie $\infty$-morphisms and homotopies between differential graded Lie algebras in terms of coderivations. We restrict ourselves to a special case that we need for this paper.\\

\noindent
\textbf{Comorphisms and coderivations}. Let $\mathfrak{g}$ and $\mathfrak h$ be graded vector spaces over $\mathbb{K}$. 

\begin{definition}
A linear map $\Phi\colon S^\bullet_\mathbb{K}(\mathfrak{g}[1])\longrightarrow S^\bullet_\mathbb{K}(\mathfrak{h}[1])$ is said to be of \emph{arity $r\in \mathbb{N}_0$}, if it sends  $S^{k}_\mathbb{K}(\mathfrak{g}[1])$ to  $S^{k-r}_\mathbb{K}(\mathfrak{h}[1])$, for all $k\geq r$. Every linear map $\Phi$ can be decomposed as formal sum:\begin{equation}
    \Phi=\sum_{k\in \mathbb{Z}}\Phi^{(k)}
\end{equation}
where, for all $k\in \mathbb{N}_0$, \,$\Phi^{(k)}\colon S^\bullet_\mathbb{K} (\mathfrak{g}[1])\longrightarrow S^{\bullet-k}_\mathbb{K}(\mathfrak{h}[1])$ is a linear map of arity $k$. Therefore, a linear map $\Phi\colon S^\bullet_\mathbb{K}(\mathfrak{g}[1])\longrightarrow S^\bullet_\mathbb{K}(\mathfrak{h}[1])$ is of arity $r\in \mathbb{N}_0$ if and only if $\Phi^{(r)}$ is the unique non-zero term, namely $\Phi^{(k)}=0$, for $k\neq r$.
\end{definition}

Let us denote by $\Delta$ the natural coalgebra structure of $S_\mathbb{K}^\bullet(\mathfrak{g}[1])$ and by $\Delta'$ the one on $S_\mathbb{K}^\bullet(\mathfrak{h}[1])$. Given any linear map $\Phi\colon S_\mathbb{K}^\bullet(\mathfrak{g}[1])\longrightarrow \mathfrak{h}[1]$, we  denote by $\Phi_k\colon S_\mathbb{K}^{k+1}(\mathfrak{g}[1])\longrightarrow \mathfrak{h}[1]$ for $k\in \mathbb{N}_0$ the restriction of $\Phi$ to $S_\mathbb{K}^{k+1}(\mathfrak{g}[1])$. The linear map $\Phi$ can be extended to a unique comorphism $\Bar\Phi\colon S_\mathbb{K}^\bullet (\mathfrak{g}[1])\longrightarrow S_\mathbb{K}^\bullet (\mathfrak{h}[1])$ by taking for $r\in \mathbb{N}$ the component on $S_\mathbb{K}^r \mathfrak{h}[1]$ to be for $x_1,\ldots,x_k\in \mathfrak{g}[1]$
\begin{equation}
    \sum_{i_1 +\cdot+i_r =k}\sum_{\sigma\in \mathfrak{S}(i_1,\ldots,i_r )}\epsilon(\sigma)\frac{1}{r!}\prod_{j=1}^r\Phi_{i_j-1}(x_{\sigma(i_{1}+\cdots+i_{j-1}+1)},\ldots, x_{\sigma(i_{1}+\cdots+i_j)}).
\end{equation}
where $\mathfrak{S}(i_1,\ldots,i_r )$ is the set of $(i_1,\ldots,i_r)$-shuffles, with $i_1,\ldots,i_r\in \mathbb{N}$. Also, $\prod$ stands for the product of $S^\bullet_\mathbb{K}(\mathfrak{h}[1])$.\\

Every comorphism from $S_\mathbb{K}^\bullet(\mathfrak{g}[1])$ to $S_\mathbb{K}^\bullet (\mathfrak{h}[1])$ is of this form \cite{Kassel}. That is, a comorphism $\Phi\colon S_\mathbb{K}^\bullet (\mathfrak{g}[1])\rightarrow S_\mathbb{K}^\bullet (\mathfrak{h}[1])$ is entirely determined by the collection indexed by $k \in \mathbb N $ of maps called  \emph{$k$-th Taylor coefficient}:
\begin{equation}
\label{eq:Taylor}
    \Phi_k \colon S^{k+1}_\mathbb{K}(\mathfrak{g}[1])\stackrel{\Phi}{\longrightarrow} S^\bullet_\mathbb{K}(\mathfrak{h}[1])\stackrel{\text{pr}}{\longrightarrow} \mathfrak{h}[1],
\end{equation}
with $\text{pr}$ being the projection onto the term of arity $1$, i.e.
$\text{pr}\colon S^\bullet_\mathbb{K}(\mathfrak{h}[1])\rightarrow S^{1}_\mathbb K (\mathfrak{h}[1])\simeq \mathfrak{h}[1]$. Notice that the component $\Phi^{(k)}$ of arity $k$ of $\Phi$ coincides with $k$-th Taylor coefficient $\Phi_k$ on $S^{k+1}_\mathbb{K}(\mathfrak{g}[1])$. Hence, a comorphism $\Phi\colon S_\mathbb{K}^\bullet (\mathfrak{g}[1])\rightarrow S_\mathbb{K}^\bullet(\mathfrak{h}[1])$ admits a  decomposition of the form:\begin{equation}
    \Phi=\sum_{k \geq 0}\Phi^{(k)}
\end{equation}
with a sum that runs on $k\geq 0$ and not on $k\in \mathbb{Z}$. 

\begin{definition}\label{coder}Let $\Phi\colon S^\bullet_\mathbb{K}(\mathfrak{g}[1])\mapsto S^\bullet_\mathbb{K} (\mathfrak{h}[1])$ be a graded comorphism. A \emph{$\Phi$-coderivation} of degree $k$ on $S_\mathbb{K}^\bullet(\mathfrak{g}[1])$ is a degree $k\in\mathbb{N}_0$ linear map $\mathcal H : S^\bullet_\mathbb{K} (\mathfrak{g}[1])\mapsto S^\bullet_\mathbb{K}(\mathfrak{h}[1])$ which satisfies the so-called (co)Leibniz identity:
\begin{equation} \label{eq:Phicoder} 
	\Delta'\circ \mathcal H = (\mathcal H\otimes \Phi)\circ \Delta + (\Phi \otimes\mathcal H) \circ \Delta.
\end{equation}When $\mathfrak{g}=\mathfrak{h}$ and $\Phi=\mathrm{id}$, we say that $\mathcal{H}$ is a \emph{coderivation}.
\end{definition}The same results on comorphisms hold for coderivations \cite{Kassel}.\\

\noindent
\textbf{Lie $\infty$-morphisms of differential graded Lie algebras}. Let $(\mathfrak{g},\lb_\mathfrak{g})$ be a Lie algebra and $(E, Q)$ a Lie $\infty$-algebroid over $M$. The graded symmetric Lie bracket on $\mathfrak{X}(E)[1]$ is of degree $+1$ and given on homogeneous elements $u,v\in\mathfrak{X}(E)[1]$ as  $$\{u,v\}:=(-1)^{\lvert v\rvert}[u,v].$$ 
In the sequel, we write $(\mathfrak{X}(E)[1],\lb,\mathrm{ad}_Q)$ instead of $(\mathfrak{X}(E)[1],\{\cdot\,,\cdot\},\mathrm{ad}_Q)$.\\

Let $\left(S_\mathbb{K}^\bullet(\mathfrak{g}[1]), Q_\mathfrak{g}\right)$ respectively $(S_\mathbb{K}^\bullet(\mathfrak{X}(E)[1]),\Bar{Q})$ be the corresponding formulations in terms of coderivations of the differential graded Lie algebras $(\mathfrak{g}[1],\lb_\mathfrak{g})$ and $(\mathfrak{X}(E)[1],\lb,\text{ad}_Q)$. Precisely, $Q_\mathfrak{g}$ is the coderivation defined  for every homogeneous monomial $x_1\wedge\cdots\wedge x_k \in S^k_\mathbb{K}(\mathfrak g[1])$ by
\begin{equation}
   Q_\mathfrak{g}(x_1\wedge\cdots\wedge x_k) :=\sum_{1\leq i<j\leq k}(-1)^{i+j-1}[x_i,x_j]_\mathfrak{g}\wedge x_1\wedge\cdots \widehat{x}_i\cdots \widehat{x}_j\cdots\wedge x_k,
\end{equation}
and  $\Bar{Q}=\Bar{Q}^{(0)}+\Bar{Q}^{(1)}$ is the coderivation of degree $+1$ where the only non-zero Taylor's coefficients are, $\Bar{Q}^{(0)}\colon S_\mathbb{K}^1(\mathfrak{X}(E)[1])\stackrel{\text{ad}_Q}{\longrightarrow}\mathfrak{X}(E)[1]$ and $\Bar{Q}^{(1)}\colon S_\mathbb{K}^2(\mathfrak{X}(E)[1])\stackrel{\{ \cdot\,,\cdot\}}{\longrightarrow}\mathfrak{X}(E)[1]$.
\begin{definition}\label{def:morph}\cite{LadaTom1994ShLa}
A \emph{Lie $\infty$-morphism} $\Phi\colon (\mathfrak{g}[1],\lb_\mathfrak{g})\rightsquigarrow (\mathfrak X_\bullet(E)[1],\lb,\text{ad}_{Q})$ is a graded coalgebra morphism $\Bar{\Phi}\colon (S_\mathbb{K}^\bullet(\mathfrak{g}[1]),Q_\mathfrak{g})\longrightarrow (S_\mathbb{K}^\bullet\left(\mathfrak X(E)[1]\right),\Bar{Q})$ of degree zero which satisfies,
\begin{equation}
    \Bar\Phi\circ Q_\mathfrak{g} =\Bar{Q}\circ\Bar\Phi.
\end{equation}In order words, it is the datum of degree zero linear maps $\left(\Bar{\Phi}_k\colon S_\mathbb{K}^{k+1}\mathfrak{g}[1]\longrightarrow \mathfrak X_{-k}(E)[1]\right)_{k\geq 0}$ that satisfies \begin{align}\label{infty-morph-axiom}
  \nonumber\sum_{1\leq i<j\leq n+2}(-1)^{i+j-1}\Bar{\Phi}_{n}([x_i,x_j]_\mathfrak{g},x_1,\ldots,\widehat{x}_{ij},\ldots,x_{n+2})&=[Q,\Bar{\Phi}_{n+1}(x_{1},\ldots,x_{n+2})] \,+&\\\sum_{\tiny{\begin{array}{c}
        i+j= n \\i\leq j\\\sigma\in\mathfrak{S}_{i+1,j+1}\end{array}}}\epsilon(\sigma)[\Bar{\Phi}_i(x_{\sigma(1)}&,\ldots,x_{\sigma(i+1)}),\Bar{\Phi}_j(x_{\sigma(i+2)},\ldots,x_{\sigma(n+2)})]
\end{align}where $\widehat{x}_{ij}$ means that we take $x_i,x_j$ out of the list. When there is no risk of confusion, we write $\Phi$ for $\Bar{\Phi}$.
\end{definition}
\begin{convention}
In the sequel, $Q_\mathfrak{g}$ and $\Bar Q$ will be in implicit.
\end{convention}
\begin{remark}It is important to notice that:
\begin{enumerate}
    \item Definition \ref{def:morph} and Definition 3.45 in \cite{LLS} are compatible when $M=\{\mathrm{pt}\}$. Therefore,  morphisms in both sense match.

\item In \cite{MEHTA2012576}, Definition \ref{def:morph} corresponds to the definition of actions of a Lie $\infty$-algebras of finite dimension on Lie $\infty$-algebroids of finite rank. Here, we only have a Lie algebra. In contrast to theirs, we do not assume that $\mathfrak{g}$ is finite dimensional.
\end{enumerate}
\end{remark}
\begin{remark}\label{Rmk:CE1}
It follows from the axioms \eqref{infty-morph-axiom} that for all $x,y\in \mathfrak{g}[1]$, $[Q, \Phi_0(x)]=0$ and \begin{equation}\label{eq:second-condition}
    \Phi_0([x,y]_{\mathfrak{g}})-[\Phi_0(x),\Phi_0(y)]=[Q,\Phi_1(x,y)].
\end{equation} 
\end{remark}


The following lemma explains what the $0$-Taylor coefficient of a Lie $\infty$-morphism as in Definition \ref{def:morph}  induces  on the linear part of $(E,Q)$. More details will be given in Proposition \ref{prop:induced-action}.
\begin{convention}\label{conv:arity}
Let $E, F$ be  graded vector bundles over a manifold  $M$.
For a $\mathbb K$-linear map $P\colon \Gamma(S^\bullet(E^*))\longrightarrow~\Gamma(S^\bullet(E^*))$ we denote by $P^{(k)}\colon\Gamma(S^N(E^*))\longrightarrow~\Gamma(S^{N+k}(E^*)), \; N\geq 0$, the $k$-th polynomial degree component of $P$  and is called the \emph{arity $k$ component}  of $P$.
\end{convention}
\begin{lemma}\label{lemma:basic-action}
The $0$-th Taylor coefficient $\Phi_0\colon \mathfrak{g}[1]\longrightarrow \mathfrak X_0(E)$ of a Lie $\infty$-morphism $\Phi$ as in Definition \ref{def:morph} induces 

\begin{enumerate}
    \item a linear map $\varrho\colon \mathfrak{g}\longrightarrow \mathfrak{X}(M),\, x\longmapsto \left(\varrho(x)[f]:=\Phi_0(x)[f],\;\;f\in \mathcal{O}\right)$ and
    \item a linear map
$x\in\mathfrak{g}\longmapsto \nabla_x\in\mathrm{Der}_0(E)$, i.e., for each $x\in \mathfrak{g}$, $\nabla_x\colon \Gamma(E)\longrightarrow \Gamma(E)$ is a degree zero map  that satisfies $$\nabla_x(fe)=~f\nabla_x(e) + \varrho(x)[f]\,e,\;\; \text{for}\;\; f\in \mathcal{O}, e\in \Gamma(E).$$
such that  \begin{equation}\label{eq:dual-actionxx}
     \langle \Phi_0(x)^{(0)}(\alpha),e\rangle=\varrho(x)[\langle\alpha,e\rangle]-\langle\alpha,\nabla_x(e)\rangle,\;\;\text{for all $\alpha\in \Gamma(E^*),e\in \Gamma(E)$}.
 \end{equation}
 $\Phi_0(x)^{(0)}$ stands for the arity zero component of $\Phi_0(x)$.
\end{enumerate}

\end{lemma}
\begin{proof}
We have for every $x\in\mathfrak g[1]$, and $e\in \Gamma(E)$,\,  $$[\Phi_0(x), \iota_e]^{(-1)}=\iota_{\nabla_xe},$$ for some $\mathbb{K}$-bilinear map $\nabla_x\colon \Gamma(E_{-\bullet})\longrightarrow \Gamma(E_{-\bullet})$ that depends linearly on $x\in\mathfrak g[1]$ and that satisfies \begin{equation}
    \label{eq:linear-part-der}
    \nabla_x(fe)=~f\nabla_x(e) + \varrho(x)[f]e,\; \text{for $f\in \mathcal{O}, e\in \Gamma(E)$}.
\end{equation}
To see \eqref{eq:linear-part-der}, compute $[\Phi_0(x), \iota_{fe}]^{(-1)}$: \begin{align*}
    \iota_{\nabla_x(fe)}&=[\Phi_0(x), \iota_{(fe)}]^{(-1)}\\&=\Phi_0(x)[f]\iota_e + f[\Phi_0(x), \iota_e]^{(-1)}\\&=\iota_{\varrho(x)[f]e + \nabla_xe}.
\end{align*}

In particular, one has for all $\alpha\in \Gamma(E^*),e\in \Gamma(E)$, \begin{align*}
    \langle \Phi_0(x)^{(0)}(\alpha),e\rangle&=\Phi_0(x)^{(0)}[\langle \alpha,e\rangle]-[\Phi_0(x)^{(0)}, \iota_e]^{(-1)}(\alpha)\\&=\varrho(x)[\langle\alpha,e\rangle]-\langle\alpha,\nabla_x(e)\rangle.
\end{align*}
\end{proof}

\noindent\textbf{Homotopies}. Now we are defining homotopy between Lie $\infty$-morphisms.
    
\begin{definition}\label{homp:def}
Let $\Bar{\Phi},\Bar{\Psi}\colon (S_\mathbb{K}^\bullet(\mathfrak{g}[1]),Q_\mathfrak{g})\rightsquigarrow \left(S_\mathbb{K}^\bullet(\mathfrak X(E)[1]),\Bar{Q}\right)$ be Lie $\infty$-morphisms. We say $\Bar{\Phi},\Bar{\Psi}$ are \emph{homotopic over the identity of $M$} if the following conditions hold:\begin{enumerate}
    \item there a piecewise rational continuous path $t\in[a,b]\mapsto\Xi_t\colon(S_\mathbb{K}^\bullet\mathfrak{g}[1],Q_\mathfrak{g})\rightsquigarrow \left(S_\mathbb{K}^\bullet(\mathfrak X(E)[1]),\Bar{Q}\right)$ made of Lie $\infty$-morphisms  that coincide with $\Bar{\Phi}$ and $\Bar{\Psi}$ at $t=a$ and $b$, respectively,
     
 \item and a piecewise rational path $t\in[a,b]\mapsto H_t$ of $\Xi_t$-coderivations of degree $-1$ such that \begin{equation}\label{eq:homotopy-formula}
        \frac{\dd\Xi_t }{\dd t}=\Bar{Q}\circ H_t+H_t\circ Q_{\mathfrak{g}}.
    \end{equation}
\end{enumerate}
\end{definition}

\begin{remark}
Homotopy equivalence in the sense of the Definition \ref{homp:def} is an equivalence relation, and it is compatible with composition of Lie $\infty$-morphisms. Also, we can  “glue” infinitely many equivalences, as in Lemma 1.39 in \cite{CLRL}.
\end{remark}

\begin{remark}
Definition \ref{homp:def} is slightly  more general than the equivalence relation \cite{MEHTA2012576}. In \cite{MEHTA2012576}, it is explained that Lie $\infty$-oid morphisms are Maurer-Cartan elements in some Lie $\infty$-algebroid $\mathfrak{g}[1]\oplus E$ of certain form, and they define equivalence as gauge-equivalence of the Maurer-Cartan elements. This gauge equivalence corresponds to homotopies as above, for which all functions are smooth. Also, we do not require nilpotence unlike in Definition 5.1 of  \cite{MEHTA2012576}. Last, we do not assume $\mathfrak{g}$ to be of finite dimension.
\end{remark}

The following Proposition shows that the notion of homotopy given in Definition \ref{homp:def} implies the usual notion of homotopy between chain maps.
\begin{proposition}\label{prop:homotpy-usual}
Let $\Phi,\Psi\colon(\mathfrak{g}[1],\lb_\mathfrak{g})\rightsquigarrow (\mathfrak X_\bullet(E)[1],\lb,\mathrm{ad}_{Q})$ be Lie $\infty$-morphisms which are homotopic. Then, \begin{equation}\label{eq:comp-homtopy}
    \Psi-\Phi=\Bar Q \circ H+H\circ Q_{\mathfrak{g}}
\end{equation}for some $\mathcal{O}$-linear map $H\colon S_\mathbb{K}^\bullet(\mathfrak{g}[1])\longrightarrow S^\bullet_\mathbb{K}(\mathfrak{X}(E)[1])$ of degree $-1$.
\end{proposition}

\begin{proof}
The proof follows by applying the property that says the variation of a piecewise-$C^\infty$ map is equal to the integral of its derivative on \eqref{eq:homotopy-formula}.

\end{proof}

\begin{proposition}\label{prop:induced-action}
Let $\mathfrak{g}$ be a Lie algebra and $(E,Q)$ a Lie $\infty$-algebroid over $M$.  

\begin{enumerate}
    \item Any {Lie $\infty$-morphism} $\Phi\colon (\mathfrak{g}[1],\lb_\mathfrak{g})\rightsquigarrow (\mathfrak X_\bullet(E)[1],\lb,\mathrm{ad}_{Q})$ induces a weak symmetry action of $\mathfrak{g}$ on  the basic singular foliation $\mathcal{F}=\rho(\Gamma(E_{-1}))$ of $(E,Q)$.
    \item Homotopic Lie $\infty$-morphisms ${\Phi},{\Psi}\colon (\mathfrak{g}[1],\lb_\mathfrak{g})\rightsquigarrow (\mathfrak X_\bullet(E)[1],\lb,\mathrm{ad}_{Q})$ induce equivalent weak symmetry actions $\varrho_a, \varrho_b$ of $\mathfrak{g}$ on  the basic singular foliation $\mathcal{F}$.
\end{enumerate}
\end{proposition}

\begin{proof}
Item \emph{1}. is a consequence of Remark \ref{Rmk:CE1}. Indeed,  take $\varrho\colon \mathfrak{g}  \longrightarrow \mathfrak{X}(M)$ as in Lemma \ref{lemma:basic-action}(\emph{1}). We claim that $\varrho$ is a weak symmetry action of $\mathfrak{g}$ on  $\mathcal{F}$: Let $x,y\in \mathfrak g[1]$, and $e\in \Gamma(E_{-1})$ and $f\in \mathcal{O}$.
\begin{itemize}
    \item  $[\Phi_0(x), Q]=0$ entails, \begin{align*}
        \left\langle\Phi_0(x)^{(0)}\left[Q^{(1)}(f)\right],e\right\rangle&=\left\langle Q^{(1)}\left(\Phi_0(x)^{(0)}[f]\right),e\right\rangle\\\varrho(x)[\langle Q[f],e\rangle]-\left\langle Q[f],\nabla_x(e)\right\rangle&=\rho(e)[\varrho(x)],\hspace{0.5cm}\text{(by Lemma \ref{lemma:basic-action} (\emph{2.}))}\\\varrho(x)[\rho(e)][f]-\rho(\nabla_x(e))[f]&=\rho(e)[\varrho(x)]
    \end{align*}
By consequence, $[\varrho(x), \rho(e)]=\rho(\nabla_x(e))\in \mathcal{F}$. Therefore, $[\varrho(x), \mathcal{F}]\subseteq \mathcal{F}$.

    \item There exists a skew-symetric linear map $\eta\colon \wedge^2 \mathfrak{g}\longrightarrow \Gamma(E_{-1})$ such that $\Phi_1(x,y)^{(-1)}=\iota_{\eta(x,y)}$. Therefore, the arity zero of Equation \eqref{eq:second-condition} evaluated at an arbitrary function $f\in \mathcal{O}$ yields:
\begin{align*}
     &\Phi_0([x,y]_{\mathfrak{g}})^{(0)}(f)-[\Phi_0(x),\Phi_0(y)]^{(0)}(f)=[Q,\Phi_1(x,y)]^{(0)}(f)\\\Longrightarrow \;&\Phi_0([x,y]_{\mathfrak{g}})(f)-[\Phi_0(x)^{(0)},\Phi_0(y)^{(0)}](f)=[Q^{(1)},\Phi_1(x,y)^{(-1)}](f)\\\Longrightarrow &\; \varrho([x,y]_\mathfrak{g})[f]-[\varrho(x), \varrho(y)][f]=[Q^{(1)},\iota_{\eta(x,y)}](f)\\\phantom{\Longrightarrow} &\; \phantom{\varrho([x,y]_\mathfrak{g})[f]-[\varrho(x), \varrho(y)][f]}=\rho(\eta(x,y))[f].
\end{align*}Since $f$ is arbitrary, this proves item \emph{1.} Using Proposition \ref{prop:homotpy-usual}, $\Phi\sim\Psi$ implies for $x\in \mathfrak{g}[1]$ that
\begin{align}
   \nonumber \Psi(x)-\Phi(x)&=\Bar{Q}\circ H(x)+\cancel{H\circ Q_{\mathfrak{g}}(x)}\\\label{eq:sym-equi}&=[Q, H(x)]
\end{align}
with $H\colon \mathfrak{g}[1]\longrightarrow \mathfrak{X}_{-1}(E)$ a linear map. Let $\beta\colon \mathfrak{g}[1]\longrightarrow \Gamma(E_{-1})$ be a linear map such that $H(x)^{(-1)}=\iota_{\beta(x)}$. Taking the arity zero of both sides in Equation \eqref{eq:sym-equi} and evaluating at  $f\in \mathcal{O}$ we obtain that \begin{align*}
        \left(\varrho_a(x)-\varrho_b(x)\right)[f]= [Q^{(1)}, H(x)^{(-1)}]= [Q^{(1)}, \iota_{\beta(x)}][f]=\rho(\beta(x))[f].
\end{align*}
Since $f$ is arbitrary, this proves item ${2}$.

\end{itemize}
\end{proof}

Proposition \ref{prop:induced-action} tells us that {Lie $\infty$-morphism} $\Phi\colon (\mathfrak{g}[1],\lb_\mathfrak{g})\rightsquigarrow (\mathfrak X_\bullet(E)[1],\lb,\text{ad}_{Q})$ induces weak symmetry action on the base manifold $M$. In Section \ref{sec:3}, we investigate the opposite direction. We respond to the following question: Do any weak symmetry action of a Lie algebra on a singular foliation comes from a Lie $\infty$-morphism? If so, can we extend in a unique manner?

\section{Proof of Theorem \ref{alt-thm-res}}\label{proof:Lie-infty} 
In this section, we prove Theorem \ref{alt-thm-res}. For $\mathfrak{g}$ a Lie algebra, let $\Gamma(\mathfrak{g})$ stand for sections of $\mathfrak g\times M\longrightarrow M$.

\begin{proof}{(of Theorem \ref{alt-thm-res})}. Let $\varphi\colon \wedge^2 \mathfrak{g}\longrightarrow\mathcal{F}$ be  such that $\varrho([x,y]_\mathfrak{g})-[\varrho(x),\varrho(y)]=\varphi(x,y)\in\mathcal{F}$ for all $x,y\in\mathfrak{g}$. Notice that $\mathcal{A}:=\left(\Gamma (\mathfrak g)\oplus \mathcal{F}, \lb_\mathcal{A}, \rho_\mathcal A\right)$ is a Lie-Rinehart algebra over $\mathcal{O}$, whose Lie bracket and anchor map are given respectively on a set of generators $(x_i)_{i\in I}$  of $\Gamma(\mathfrak g[1])$ and $(X_\lambda)_{\lambda\in \Lambda}$ of  $\mathcal{F}$ by:\begin{enumerate}
    \item for $i,j\in I$ and $\lambda, \beta\in \Lambda$, \begin{equation}\label{LR-lb}
        [(x_i,X_\lambda), (x_j, X_\beta)]_\mathcal{A}:=\left([x_i,x_j]_\mathfrak{g}, [X_\lambda, \varrho(x_j)]-[X_\beta, \varrho(x_i)]-\varphi(x_i,x_j)+[X_\lambda,X_\beta]\right)
    \end{equation}
    \item for $i\in I$ and $\lambda\in \Lambda$ \begin{equation}\label{LR-anch}
        \rho_\mathcal{A}(x_i,X_\lambda)=\varrho(x_i)+X_\lambda
    \end{equation}
\end{enumerate}
We extend the bracket \eqref{LR-lb} by Leibniz identity. Also, $\rho_\mathcal A$ in \eqref{LR-anch} is extended by $\mathcal O$-linearity (it is a morphism by construction).\\

\noindent
For any free resolution $(\mathcal{K}_\bullet, \ell_1, \rho)$ of $\mathcal{F}$, the sequence  \begin{equation}\label{eq:free-resolution}
   \xymatrix{} \cdots\stackrel{\ell_1}{\longrightarrow}\mathcal{K}_{-2}\stackrel{\ell_1}{\longrightarrow}\Gamma (\mathfrak{g}[1])\oplus\mathcal{K}_{-1} \stackrel{ \pi=\mathrm{id}\oplus\rho}{\xrightarrow{\hspace*{1cm}}}\Gamma (\mathfrak{g})\oplus\mathcal{F}
\end{equation} is a free resolution of $\mathcal{A}=\Gamma (\mathfrak{g})\oplus\mathcal{F}$. By Theorem 2.1 of \cite{CLRL}, the complex \eqref{eq:free-resolution} can be equipped with a Lie $\infty$-algebroid whose unary bracket is $\ell_1$ and whose anchor is $\rho':=\rho_\mathcal{A}\circ\pi$. {But we claim that we can add some constraint on the $k$-ary brackets that appear in Theorem 2.1. This requires to adapt its proof to this particular setting.}\\


\noindent 
\textbf{Construction of a $2$-ary bracket on $\mathfrak{g}[1]\oplus E_{-1}$}: Let us denote by $(e_\lambda^{(-1)})_{\lambda \in \Lambda}$ a free basis of $\mathcal{K}_{-1}$. The set $\{X_\lambda=\rho(e_\lambda^{(-1)})\in \mathcal{F}\mid \lambda\in \Lambda\}$ is a set of generators of $\mathcal{F}$. Let $(x_i)_{i\in I}$ be a basis for $\Gamma(\mathfrak g[1])$ There exists elements $c^k_{\lambda\beta}\in\mathcal{O}$
and satisfying the skew-symmetry condition $c^\alpha_{\lambda\beta}=-c^\alpha_{\beta\lambda}$ together with
\begin{equation} \label{def:uijk}
\left[ X_\lambda,X_\beta\right]=\sum_{\alpha\in \Lambda}c^\alpha_{\lambda\beta} X_\alpha \hspace{.5cm} \forall \lambda,\beta \in \Lambda.
\end{equation}
By definition of the weak symmetry action $\varrho$, one has \begin{equation}
  [\varrho(x_i),\rho(e_\lambda^{(-1)})]\in \mathcal{F}\quad\text{and}\quad \varrho([x_i,x_j])_\mathfrak{g}-[\varrho(x_i),\varrho(x_j)]\in\mathcal{F}\;\;\text{for all}\; (i,j)\in I^2\, \text{and}\, \lambda\in \Lambda.  
\end{equation}
 Since $(\mathcal{K}_\bullet, \ell_1, \rho)$ is a free resolution of $\mathcal F$, there exists two $\mathcal O$-bilinear maps $\chi\colon\Gamma({\mathfrak{g}[1]})\times \mathcal{K}_{-1}\to ~\mathcal{K}_{-1}$ and   $\eta\colon\Gamma({\mathfrak{g}[1]})\times\Gamma({\mathfrak{g}[1]})\to \mathcal{K}_{-1}$ defined on generators $x_i, e_\lambda^{(-1)}$ by the relations\begin{equation*}
    [\varrho(x_i),\rho(e_\lambda^{(-1)})]=\rho(\chi(x_i,e_\lambda^{(-1)}))\qquad\text{and}\qquad \varrho([x_i,x_j]_\mathfrak{g})-[\varrho(x_i),\varrho(x_j)]=\rho(\eta(x_i,x_j)).
\end{equation*}
We define a $2$-ary bracket on $\Gamma(\mathfrak{g}[1])\oplus\mathcal{K}_{-1}$ as follows: 
\begin{enumerate}
\item an anchor map by $\rho'(0\oplus e_\lambda^{(-1)})=X_\lambda$, and $\rho'(x_i\oplus 0)=\varrho(x_i)$, for all $i\in I$ and  $\lambda\in \Lambda$,
\item a degree $ +1$ graded symmetric operation $\ell'_2 $ on $\Gamma(\mathfrak{g}[1])\oplus \mathcal{K}_{-1}$ as follows: for all $i,j$ and $\lambda, \beta\in \Lambda$\\

\begin{enumerate}
    \item $\ell'_2\left(0\oplus e_\lambda^{(-1)},0\oplus e_\beta^{(-1)} \right) =0\oplus\sum_{\alpha\in \lambda}c^\alpha_{\lambda\beta} e_\alpha^{(-1)}$,
    \item $\ell'_2\left(x_i\oplus 0,0 \oplus e_\lambda^{(-1)} \right)=0\oplus \chi\left(x_i ,e_\lambda^{(-1)}\right)$,
     \item $\ell'_2\left(x_i\oplus 0,x_j\oplus 0\right)=[x_i,x_j]_{\mathfrak{g}}\oplus\eta(x_i,x_j)$.
\end{enumerate} 
\end{enumerate}  We extend $\ell'_2$ to $\Gamma(\mathfrak{g}[1])\oplus\mathcal{K}_{-1}$  using Leibniz identity with respect to the anchor map $\rho'$.

By construction $\ell'_2$ satisfies the Leibniz identity with respect to the anchor $\rho'$. Moreover, $\rho'$ is  a bracket morphism. We continue exactly as in the proof of Lemma 2.23 in \cite{CLRL} and construct a 2-ary bracket $\ell_2$ of degree +1 whose restriction to $\Gamma (\mathfrak{g}[1])\oplus\mathcal{K}_{-1}$ is $\ell'_2$. It equips the complex \eqref{eq:free-resolution} with an almost Lie algebroid such that $\rho'\left(\Gamma (\mathfrak{g}[1])\oplus\mathcal{K}_{-1}\right)=\mathcal{A}$ and whose restriction to $\mathfrak{g}[1]$ is the Lie bracket of $\mathfrak{g}[1]$. A close look at the construction in Theorem 2.1 of \cite{CLRL} implies that this almost differential graded Lie algebroid can be extended to a Lie $\infty$-algebroid on the complex \eqref{eq:free-resolution}.
\end{proof}
\bibliographystyle{plain}
\bibliography{Sym_biblio}
\end{document}